\documentclass{amsart}
\usepackage{amssymb,amsmath,hyperref,cleveref,bbm}
\usepackage[dvipsnames]{xcolor}
\usepackage{fullpage}
\usepackage{bm}
\usepackage[shortlabels]{enumitem}

\newcommand{\nc}{\newcommand}

\newcommand{\Aut}{{\sf Aut}}
\newcommand{\gen}{{\sf gen}}
\newcommand{\sym}{{\sf sym}}
\newcommand{\Sym}{\operatorname{Sym}}
\newcommand{\Ggo}{\mathfrak{G}}
\newcommand{\B}[1]{\overline{#1}}
\newcommand{\dd}{\mathrm{d}}
\newcommand{\m}[2]{\left\langle #1,#2 \right\rangle}
\newcommand{\mdot}{\m{\cdot}{\cdot}}
\newcommand{\f}[1]{\mathfrak{#1}}
\newcommand{\ad}{\operatorname{ad}}
\newcommand{\R}{\mathbb{R}}
\newcommand{\C}{\mathbb{C}}

\newcommand{\id}{\operatorname{Id}}

\newcommand{\comment}[1]{}
\newcommand{\msf}[1]{\mathsf{#1}}

\newcommand{\mc}[1]{\mathcal{#1}}

\newcommand{\tr}{\operatorname{tr}}

\nc{\mud}{{\bm{\mu}}} 
\nc{\nud}{{\bm{\nu}}}
\nc{\ip}{{\langle \,\cdot \,,\cdot \,\rangle }} 
\nc{\la}{\langle} \nc{\ra}{\rangle}
\nc{\lla}{\langle \! \langle}
\nc{\rra}{\rangle \! \rangle}
\nc{\iip}{{\lla \cdot \, , \, \cdot \rra}}
\nc{\cca}{\mathcal{C}} 
\nc{\gca}{\mathcal{G}}
\nc{\mca}{\mc M}
\nc{\TT}{{T \oplus T^*}}
\nc{\TTH}{{(T \oplus T^*)_H}}
\nc{\lb}{[\cdot, \cdot]}
\nc{\G}{\mathsf{G}}
\nc{\F}{\mathsf{F}}
\nc{\K}{\mathsf{K}}
\nc{\bG}{\overline{\mc G}}
\nc{\infECA}{V}
\nc{\sog}{\mathfrak{so}}
\nc{\ggo}{\mathfrak{g}}
\nc{\lgo}{\mathfrak{l}}
\nc{\gggo}{\mathfrak{g} \oplus \mathfrak{g}^*}
\nc{\Lie}{\operatorname{Lie}}
\nc{\Ll}{\mathsf{L}}
\nc{\End}{\operatorname{End}}
\nc{\Ad}{\operatorname{Ad}}
\nc{\mm}{\operatorname{M}}
\nc{\Ricb}{\operatorname{Ric}^B}
\nc{\scal}{\operatorname{R}}

\theoremstyle{plain}
\newtheorem{maintheorem}{Theorem}

\newtheorem{maincorollary}[maintheorem]{Corollary}

\newtheorem{thm}{Theorem}[section]
\newtheorem{defn}[thm]{Definition}
\newtheorem{prop}[thm]{Proposition}

\newtheorem{lem}[thm]{Lemma}
\newtheorem{cor}[thm]{Corollary}

\theoremstyle{definition}
\newtheorem{quest}[thm]{Question}
\newtheorem{rmk}[thm]{Remark}
\newtheorem{ex}[thm]{Example}

\title{Homogeneous Generalized Ricci flows II}
\author{Elia Fusi}
\address[Elia Fusi]{Dipartimento di Matematica ``Giuseppe Peano'', Universit\`a di Torino, Via Carlo Alberto 10, 10123 Torino, Italy}
\email{elia.fusi@unito.it}

 \author{Ramiro A. Lafuente}
 \address[Ramiro Lafuente]{School of Mathematics and Physics, The University of Queensland, St.\ Lucia QLD 4072, Australia}
\email{r.lafuente@uq.edu.au}

 \author{James Stanfield}
 \address[James Stanfield]{School of Mathematics and Physics, University of Wollongong,  Northfields Ave,  Wollongong, NSW 2522,  Australia}
\email{jstanfield@uow.edu.au}

\date{\today}

\begin{document}

\begin{abstract}
 We relate the generalized Ricci curvature of left-invariant generalized metrics  on a Lie group with the Ricci curvature of certain metric Lie algebras. As an application,  we  prove subconvergence of  rescaled generalized Ricci flows to  expanding generalized Ricci solitons on simply connected Lie groups with positive semi-definite Killing form. Under the same hypotheses,  subconvergence of pluriclosed flows to expanding pluriclosed solitons is also shown. We further prove that left-invariant  Bismut-Ricci flat metrics arise as critical points for the generalized scalar curvature functional on unimodular Lie groups. In view of this,  we establish dynamical stability  for the generalized Ricci flow of so-called standard bi-invariant,  Bismut-flat metrics on compact, simply connected,  semisimple Lie groups and construct  examples of dynamically unstable Bismut-flat metrics which are non-standard.
\end{abstract}

\maketitle

\tableofcontents

\section{Introduction}
    Over the  last decades, Hitchin's \emph{generalized geometry}  has emerged as a natural framework unifying complex and symplectic geometry, with deep connections to mathematical physics, see \cite{CFMP85, CSCW11, GF19, GRFbook,GuaAnnals,Sletters, SV17, SV20, ST17}.  Within this context, the \emph{generalized Ricci flow}  is a flow of \emph{generalized Riemannian metrics} on a given exact Courant algebroid, extending in a natural way the classical Ricci flow \cite{Ham1}. Up to isomorphism, the structure of exact Courant algebroid on a smooth manifold $M$ is   completely determined by the choice of a cohomology class $\alpha\in H^{3}(M, \R)$, while a generalized metric $\mc G$ can be viewed as a pair $(g, H)$, where  $g$ is a Riemannian metric on $M$ and $H\in \alpha$. To any such  $\mc G$, one can associate an endomorphism $\mc Rc(\mc G)$ of $TM\oplus T^*M$, called  the \emph{generalized Ricci curvature},  which uniquely determines the Ricci tensor  of the \emph{Bismut connections} $\nabla^{\pm}$, i.e. the unique metric connections whose torsions are $\pm H$. A smooth family $(\mathcal G_t)_{t \in I}$ is a solution to the generalized Ricci flow if  it satisfies the following weakly parabolic PDE:
    \begin{equation}\label{eqn_GRFintro}
\mc G^{-1}\frac{\partial }{\partial t }\mathcal G=-2\,\mc Rc(\mc G)\,.
    \end{equation} Equivalently, a family $(g_t, H_t)_{t\in I}$ solves the generalized Ricci flow if 
    $$
\begin{cases}
\frac{\partial}{\partial t}g= -2{\rm Ric}(g)+\frac12H^2\,,\\
\frac{\partial}{\partial t}H=\Delta_gH,
\end{cases}
    $$ where ${\rm Ric}(g)$ is the classical Ricci tensor of $g$, and $H^2\in \Gamma( {\rm Sym}^2(T^*M))$ is defined by $H^2(X, X):=|\iota_XH|_g^2$, for any $X \in\Gamma(TM).$  Setting $H_0 = 0$ one recovers  the classical Ricci flow.
    Furthermore, on suitable Hermitian manifolds  the generalized Ricci flow  was shown to be gauge-equivalent to a flow of Hermitian metrics called \emph{pluriclosed flow} introduced in \cite{ST10}, see \cite{Str13, ST12}. The latter has a prominent  conjectural role  in the classification of compact complex surfaces, see \cite{S20conj}.

The main goal of this paper is to study the long-time behaviour 
of the generalized Ricci flow, and the dynamical stability of its fixed points, in the  homogeneous setting.  As in \cite{HGRF}, to enhance clarity, we will focus mainly on invariant generalized Ricci flows on  Lie groups.

  To describe our main results, let  $\msf G$ be an $n$-dimensional Lie group with Lie algebra $\f g$. Following \cite{HGRF}, the choice of an invariant cohomology class $[H]\in H^3( \msf G, \R)^{\msf G}$ and  a left-invariant generalized metric $\mc G$ on $(\msf G, [H])$  determine the following tensors on $\f G:=\f g\oplus \f g^*$: 
  \begin{itemize}
    \item a Lie bracket $[\cdot, \cdot]_H$, called the $H$-\emph{twisted Dorfman bracket}; \footnote{Strictly speaking, different representatives for $[H]$ give rise to different brackets, but these are all isomorphic through B-field transformations.} 
    \item a non-degenerate, symmetric bilinear form $\ip$ of signature $(n, n)$, called the \emph{neutral inner product};  
    \item and a positive-definite inner product, denoted again by $\mc G$.
\end{itemize}
The first key observation in this paper is a new formula relating the generalized Ricci curvature $\mc Rc(\mc G)$ and the \emph{classical} Riemannian Ricci curvature operator ${\rm Ric}_{\mc G}$ of the metric Lie algebra $(\f G, [\cdot, \cdot]_H, \mc G)$.  More precisely,  if $\msf G$ is unimodular, then we obtain the following:
\begin{equation}\label{eqn_formula_intro}
	\mc Rc(\mc G) = \left(\frac 23 \operatorname{Ric}_{\mc G}- \frac 16 B_{\mc G} \right)_{\f{so}},
\end{equation}
 where $B_{\mc G} = \mc G^{-1} \kappa_{\f G} \in {\rm Sym}(\f G, \mc G)$ is the endomorphism associated to the Killing form $\kappa_{\f G}$ of $(\f G,  [\cdot,\cdot ]_H)$, and the subscript $(\cdot)_{\f{so}}\colon {\rm Sym}(\f G, \mc G)\to \f{so}(\f G, \ip)\cap {\rm Sym}(\f G, \mc G) $ denotes the orthogonal projection with respect to the \emph{symmetric metric} \eqref{eqn:g^sym}.
 We refer to \Cref{thm:formula} for a more precise and general statement. 

In \cite{HGRF}, the authors proved that invariant generalized Ricci flows are immortal on Lie groups diffeomorphic to $\R^n$. As a first application of \eqref{eqn_formula_intro}, we describe the asymptotic long-time behaviour of invariant generalized Ricci flow solutions on nilpotent Lie groups, thus addressing the most prominent open question in \cite{HGRF}. In fact, this applies more generally to a large class of solvable Lie groups. 

\begin{maintheorem}\label{main_solv}
Let $\G$ be a simply-connected Lie group with positive semi-definite Killing form, $[H] \in H^3(\G,\R)^\G$. Then, given  any  invariant non-flat generalized Ricci flow solution $(g_t, H_t)_{t\in (0,+\infty)}$  on $(\G,[H])$, and  any sequence  $t_k \to +\infty$,  the rescaled exact Courant algebroids with rescaled metrics $\left( (\G, \tfrac1{t_k} [H]), (\tfrac1{t_k} g_{t_k}, \tfrac1{t_k} H_{t_k}) \right)$ subconverge, in the generalized Cheeger-Gromov sense, to a  generalized Ricci soliton $\left( (\G_{\infty}, [H_\infty]), (g_\infty, H_\infty) \right)$. 
\end{maintheorem}

 By definition, a  \emph{generalized Ricci soliton} is a solution of the generalized Ricci flow $(g_t, H_t)$ for which there exist $c(t)\in \R_+$ and $\varphi_t\in {\rm Diff}(\msf G)$ satisfying $(\varphi_t^*g_t, \varphi_t^*H_t)=(c(t)g_0, c(t)H_0)$, for any $t $, see \cite{HGRF, PR25}. Regarding rescaling of exact Courant algebroids, see \cite{HGRF}. Even though the exact Courant algebroid $(\msf G_{\infty}, [H_{\infty}])$ is left-invariant, it may happen that $\msf G_{\infty}$ is not isomorphic to $\msf G$ (cf.~\cite{nilRF}).

Lie groups with positive semi-definite Killing form are necessarily solvable, see \Cref{lem_posKilling} and \cite[Remark 4.8, (a)]{Heb}. This class  includes all nilpotent and \emph{completely solvable} Lie groups. Thus, the implicit assumption of long-time existence follows from \cite[Theorem A]{HGRF}.


Regarding the Killing form assumption, we remark that even in the classical case ($H=0$) the long-time behaviour is still not fully understood for solvable $\G$. An analogous statement to Theorem \ref{main_solv} for $H=0$ was obtained  in \cite{BL2}, under slightly more general assumptions on $\G$.  We note however that the proof of Theorem \ref{main_solv} is new and conceptually much simpler than that in \cite{BL2}.

In our proof of Theorem \ref{main_solv}, we also obtain further information on the limiting soliton. Namely that $\dd^*_{g_\infty}H_\infty = \iota_{U_\infty}H_\infty$, where $U_\infty := g_{\infty}^{-1}\tr \ad_{(\cdot)} $ is the \emph{mean curvature vector} of $(\msf G,g_\infty)$. In particular, for $\msf G$ unimodular, we give a positive answer to \cite[Q.~7.12]{HGRF}, originally posed for nilpotent groups.

\begin{maincorollary}\label{main_cor}
Let $\msf G$ be a simply-connected, unimodular Lie group with positive semi-definite Killing form. Then, any algebraic generalized Ricci soliton on $\msf G$ has harmonic torsion.
\end{maincorollary}

For $\msf G$ \emph{nilpotent}, Corollary \ref{main_cor} was recently shown in \cite{CCZ26b}. In \cite{CortezDwivediJetter}, the authors give an independent proof for the five-dimensional Heisenberg group, as well as a classification of solitons on other classes of nilmanifolds.



In the nilpotent case, the proof of Theorem \ref{main_solv} relies on \eqref{eqn_formula_intro} and  the interpretation of the Ricci curvature of invariant metrics on nilpotent Lie groups as a moment map for a natural ${\sf GL}_n(\R)$-action on the variety of Lie algebras. However, such an interpretation is no longer valid for solvmanifolds. Nonetheless, to prove Theorem \ref{main_solv},  we introduce a new approach to interpret the Ricci curvature of Riemannian solvmanifolds as a moment map for a natural ${\sf GL}_n(\R)$-action on  an extended space of brackets. This is novel even in the classical setting of invariant Riemannian metrics on Lie groups, and might be of independent interest.  
This  automatically leads to   an analytic monotone decreasing quantity along the generalized Ricci flow, see \Cref{eqn_monqua}.
Finally, \L ojasiewicz inequality for said quantity allows us  to construct   another monotone decreasing quantity along the generalized Ricci flows, see \Cref{eqn_2monqua}, for large times. The conclusion is then proved by standard arguments. Recently, in the setting of nilpotent Lie groups, Gindi independently constructed a monotone functional for homogeneous generalized Ricci flows, and studied their blow-down limits \cite{G26}.

\vskip10pt

Exploiting the gauge equivalence of the generalized Ricci flow and the pluriclosed flow, and adapting the methods used to prove \Cref{main_solv}, we study the asymptotic behaviour of homogeneous  pluriclosed flows on unimodular Lie groups with positive semi-definite  Killing form endowed with a left-invariant complex structure.

\begin{maintheorem}\label{main_PCF}
     Let $(\msf G, J)$ be a simply-connected Lie group with positive semi-definite Killing form endowed with an invariant complex structure. Let $(g_t)_{t\in[0, \infty)}$ be a non-flat, invariant solution to the pluriclosed flow.  Then, for any sequence  $t_k\to+\infty$,  $ \{ (\msf G, J, \frac{1}{t_k}g_{t_k}) \}_{k\in \mathbb{N}}$ subconverges, in the  complex Cheeger-Gromov sense, to an  invariant, non-flat, expanding pluriclosed soliton $(\msf G_{\infty}, J_{\infty}, g_{\infty})$.
\end{maintheorem}
 A \emph{pluriclosed soliton} is a solution of the pluriclosed flow which evolves just by scaling and  the action of biholomorphisms of $(\msf G, J)$, we refer to \Cref{sec_pf} for the precise definition.  \Cref{main_PCF} was already known  for  nilpotent,  almost-abelian and some almost-nilpotent Lie groups, see \cite{AL19, FP22}, for the  universal cover of a locally homogeneous complex surface, see \cite{Bol2016}, and for  the universal cover of an Oeljeklaus–Toma manifold, see \cite{FV23}. In order to deduce it from Theorem \ref{main_solv}, we show that suitable pluriclosed semi-algebraic generalized Ricci solitons are in fact  soliton solutions in the pluriclosed sense, see \Cref{prop_Stanfieldthm}.

 \vskip10pt

An important class of generalized metrics is that of \emph{Bismut-Ricci flat} metrics,  i.e.   satisfying $\mc Rc(\mc G)=0$. Aside from being fixed points of the generalized Ricci flow,  such generalized metrics have attracted a lot of interest in recent years due to their appearance as solutions to the equations of motion in some supergravity theories, see \cite{CSCW11,GFS15}.  Many homogeneous examples of  such  metrics are now available in the literature, see \cite{CCZ26, CFG24, CK23, LW23, PR23, PR24}.  Among Bismut-Ricci flat metrics, the most  distinguished class is that of Bismut-flat metrics, i.e. generalized metrics such that $\nabla^{\pm}$ are flat. By the work of  Cartan and Schouten in \cite{CS26, CS26bis} (see also \cite{AgrF10}),  a complete and simply-connected Bismut-flat manifold is isometric to a compact semisimple Lie group $\msf G$ endowed with the generalized Riemannian metric $\mc G^\pm_{\rm b} = (g_{\rm b}, H^\pm_{g_{\rm b}})$ determined by a bi-invariant metric $g_{{\rm b}}$ on $\msf G$ and one of its associated \emph{Cartan $3$-form} $H^\pm_{g_{{\rm b}}}$, defined by 
\begin{equation}\label{eqn:Cartan}
    H^\pm_{g_{{\rm b}}}(X,Y, Z):=\pm g_{{\rm b}}([X, Y], Z), \qquad X,Y,Z \in \ggo.
\end{equation}
We will refer to the pair $(g_{{\rm b}}, H^\pm_{g_{\rm b}})$ or to the associated generalized metric as a \emph{standard Bismut-flat metric}.
 Examples of non-standard bi-invariant Bismut-flat pairs arise by changing the sign of $H^\pm_{g_b}$ on one or more (but not all) of the simple factors of $\f g$. 
As a  consequence of \Cref{eqn_formula_intro} (see \Cref{cor_riccigrad}),  invariant Bismut-Ricci flat  metrics on a unimodular Lie group $\msf G$  are critical points for the \emph{generalized scalar curvature}:
 $$
 \mc S\colon \mc M_{ \gen}^{\f G}\to \R\,,\quad \mc S(\mc G):={\rm scal}(g)-\frac{1}{12}|H|^2\,,
 $$ where $\mc M_{ \gen}^{\f G}$ is the space of invariant generalized metrics on $(\f G, \ip)$.  In view of this,   we study  invariant Bismut-Ricci flat  metrics from a variational point of view, with a special focus on the compact semisimple case.
 As the main motivation for our next results, we recall that it has been conjectured that the homogeneous generalized Ricci flow is immortal on exact Courant algebroids of the form $(\msf G,\pm[H_{g_b}])$ for $\msf G$ compact semisimple, and that the fixed point $(g_b,\pm H_{g_b})$ is a global attractor, see \cite[Conjecture 4.14]{GRFbook}. 
 

 
 We first show that this is true at least locally, by proving dynamical stability of  standard bi-invariant Bismut-flat metrics under the homogeneous generalized Ricci flow. Before stating the result itself, let us observe that $\mc M_{\gen}^{\f G}$ is a symmetric space of non-positive sectional curvature. We  denote by $d$ the distance  on $\mc M_{\gen}^{\f G}$ induced by the symmetric metric. Then, our next theorem reads as follows.
\begin{maintheorem}\label{main_stabBF}
    Let $\mc G_{{\rm b}}$ be a  standard Bismut-flat generalized metric on  a compact, simply-connected,  semisimple Lie group $\msf G$. Then, there exists $\delta>0$ such that if $d(\mc G_0, \mc G_{{\rm b}})<\delta$, then the solution $\mc G_t$ of the generalized Ricci flow \eqref{eqn_GRFintro} starting at $\mc G_0$ exists for all positive times, and converges exponentially fast in the $C^\infty$ topology to a bi-invariant Bismut-flat  metric.
\end{maintheorem}
 We remark that the limit generalized metric might not coincide with $\mc G_{{\rm b}}$, but it is isometric to it. Said  isometry is completely determined by an automorphism of $\f g$ and an invariant, closed $2$-form, see \Cref{rmk_conj}. This, in particular, implies that the limit generalized metric is represented by the pair $(g_{{\rm b}}, H_{g_{{\rm b}}})$ associated to $\mc G_{{\rm b}}$,  thus proving \cite[Conjecture 4.14]{GRFbook} for sufficiently close initial data. 
 We also note  that  \Cref{main_stabBF} fails to be true for  \emph{non-standard} bi-invariant Bismut flat metrics, i.e., when the torsion is not chosen to be precisely the Cartan $3$-form  associated to the considered bi-invariant metric, see \Cref{main_newBRF} below.

\Cref{main_stabBF} is obtained by first proving \emph{linear stability up to gauge} of  standard bi-invariant Bismut-flat metrics  with respect to $\mathcal S$, see Proposition \ref{prop_linearstab}. Then an application of the classical Center Manifold Theorem, see \cite{Carr81}, is sufficient to deduce the statement.

 Regarding classical Ricci flat metrics, their dynamical stability modulo diffeomorphisms  under both the Ricci flow and the generalized Ricci flow was proved to be  equivalent to being  local maximizers, respectively, of the classical \textit{Perelman's $\lambda$-functional}, see \cite{Ses06, HM14},   and  of a generalization of the $\lambda$-functional adapted to generalized geometry, see \cite{RV21}. 
 In \cite{Lee25}, the dynamical stability of Bismut-Ricci flat metrics for the generalized Ricci flow was  shown   under additional assumptions that are indeed  optimal. Other results concerning the stability of  steady generalized  Ricci solitons can be found in \cite{GU25, Lee24, Lee25CY}.



In our next result, we provide the first examples (as far as we are aware) of Bismut-flat generalized metrics that are dynamically unstable for the generalized Ricci flow:

\begin{maintheorem}\label{main_newBRF}
    Let $\msf K$ be a compact simple Lie group, $\msf G = \msf K \times \msf K$, and $H_{-\kappa} = -\kappa([\cdot,\cdot],\cdot)$ the Cartan $3$-form associated to the Killing form metric. Then there exists a left-invariant Bismut-flat generalized metric $(g,H \in [H_{-\kappa}])$ on $(\msf G,[H_{-\kappa}])$ which is not bi-invariant. Moreover, this metric is dynamically unstable for the invariant generalized Ricci flow on $(\msf G,[H_{-\kappa}]).$
\end{maintheorem}


The lowest dimensional example of semisimple Lie group covered by Theorem \ref{main_newBRF} is $S^3 \times S^3 = \msf{SU}(2)\times \msf{SU}(2)$. Examples of unstable Bismut-\emph{Ricci}-flat metrics were recently announced in \cite{CCZ26}. Theorem \ref{main_newBRF} yields the first example of a non-standard Bismut-flat metric whose torsion class lies in the same cohomology class as the Cartan 3-form \eqref{eqn:Cartan} associated to $g_{\rm b} = -\kappa_\ggo$. Thus, in particular we deduce that standard Bismut-flat metrics are not, in general, global attractors for the generalized Ricci flow, hence disproving \cite[Conjecture 4.14]{GRFbook}. This is in sharp contrast with the complex setting \cite{Bar22, Bar23, GFJS23}. As far as the authors are aware, it remains open whether the standard bi-invariant Bismut-flat metrics on \emph{simple} Lie groups are global attractors. Other examples of non-bi-invariant (pluriclosed) Bismut-flat metrics on compact Lie groups can be found in \cite[Example 4.1]{Bar23}. The new Bismut-flat metrics in Theorem \ref{main_newBRF} are (non-equivariantly) isometric to the bi-invariant (but  non-standard) pair $(-2\kappa_{\f k} \oplus - \kappa_{\f k}, -2H_{\kappa_{\f k}} \oplus H_{\kappa_{\f k}}$) on $\msf G$, and thus Theorem \ref{main_stabBF} does not generalise to  non-standard Bismut-flat metrics, see Remark \ref{rmk_nonStandardUnstable}.


To prove \Cref{main_newBRF}, we exploit an adaptation of the  strategy used in \cite{Jen73} to construct homogeneous Einstein metrics. In particular, we consider the space of left-invariant generalized metrics on $(\msf G,[H_b])$ which are also invariant under the \emph{right} action of the diagonal subgroup $ \Delta \msf K:=\{(k, k )\in \msf G\,  \, |\, \,k \in \msf K\} \leq \msf G$. After characterising left-invariant metrics and $3$-forms cohomologous to $H_{-\kappa_{\f g}}$ which are right $\Delta\msf K$-invariant, we  show that the desired generalized metric  is (non-equivariantly) isometric to a bi-invariant metric on $\msf G$, thus giving Bismut-flatness.  Finally, the instability statement is  proved by showing that the    generalized metric $(g, H)$  is not a local maximum for $\mc S$ and applying an analogue of \cite[Theorem 3]{AK06}  for analytic manifolds (\Cref{lem_unstable}).

Interestingly, for the metrics $(g,H)$ from Theorem \ref{main_newBRF} one has $\mc S(g, H)< \mc S(-\kappa_{\f g}, H_{-\kappa_{\f g}})$. Motivated by this and \Cref{prop_linearstab}, it  is  natural to  ask:
\begin{quest}\label{quest_globmax}
Let $\msf G$ be a compact semisimple Lie group, and $(g_{\rm b},H^\pm_{g_{\rm b}})$ a  standard bi-invariant Bismut-flat pair. Is $\mc S (g_{\rm b},H^\pm_{g_{\rm b}})$ a global maximum for $\mc S$ among  left-invariant generalized metrics on $(\msf G,[H^\pm_{g_{\rm b}}])$?
\end{quest}

Notice that,
using \cite[Theorem 6.2]{HGRF}, this would automatically guarantee long-time existence of homogeneous generalized Ricci flows of left-invariant metrics on such exact Courant algebroids.

\vskip10pt 

The rest of the paper is organized as follows. In \Cref{sec:preliminaries} we recall the preliminary concepts in generalized geometry, including the definition of the generalized Ricci flow and its soliton solutions, with a special focus on the homogeneous case.  In \Cref{section_formulaRc}, we prove \Cref{eqn_formula_intro} together with its consequences.
\Cref{sec_ConvCS} is dedicated to the discussion of the proof of \Cref{main_solv}. In this section, we introduce the space of \emph{augmented Dorfman brackets} and prove that the generalized Ricci curvature arises as the moment map of a suitable action on said space. By means of the latter, we  derive a monotone quantity along the generalized Ricci flow, see \Cref{eqn_monqua},  which allow us to conclude  the proof of \Cref{main_solv} and \Cref{main_cor}. After recalling some basic definitions in Hermitian geometry, in \Cref{sec_pf} we give the proof of \Cref{main_PCF}. Finally, \Cref{sec_stab} and \Cref{sec_BRF} are dedicated to the study  of Bismut-flat metrics  and generalized Ricci solitons on compact semisimple Lie groups.  The main objective of \Cref{sec_stab} is to show \Cref{main_stabBF} while, in \Cref{sec_BRF} we describe the proof of \Cref{main_newBRF}.
\medskip

\noindent \textbf{Conventions and notation.} We will generically   denote with $I$ an interval  of $\R$ of the form $[0, T)$,   with $0<T\le +\infty$, when no relevance is given to the supremum of the interval.\\
Following  \cite{HGRF}, given any inner product $g$ on a Lie algebra  $\f g$ of dimension $n$, we will  denote again  with $g$  the   inner product on $(\f g ^*)^{\otimes p }$ induced by $g$, i.e. the one making  the basis $\{e_{i_1}\otimes \ldots\otimes e_{i_p}\, : \, i_1,\ldots, i_p\in\{1, \ldots, n\}\}$ orthonormal, where $\{e_1, \ldots, e_n\}$ is a $g$-orthonormal basis of $\f g$. With this choice,  the usual adjoint operator $\dd_g^*$  of the exterior derivative $\dd$ satisfies:
\begin{equation}\label{eqn_convdstar}
g(\alpha, \dd_g^*\beta)=\frac{1}{p+1}g(\dd \alpha, \beta) +g(\alpha, \iota_U\beta)\,, \qquad \alpha \in \Lambda^p\f g^*\,, \quad \beta\in\Lambda^{p+1}\f g^*\, ,
\end{equation}
where $U=g^{-1}{\rm tr}\ad_{(\cdot)}$ is the mean curvature vector of the metric Lie algebra $(\f g, g)$. On an inner product space $(V,g)$, we denote the space of $g$-symmetric endomorphisms on $V$ by $\operatorname{Sym}(V,g) \subset \operatorname{End}(V)$.\\

\vskip5pt

\noindent \textbf{Acknowledgements.} This project was initiated during the  \emph{Geometric Analysis and Symmetries} workshop hosted by the MATRIX institute in Creswick, Australia, February 2025, and continued during the \emph{Analytic and Geometric Methods on Complex Manifolds} workshop, also hosted by MATRIX in February 2026. The authors are grateful to the hospitality of MATRIX on both occasions. Likewise, RL and JS  are grateful to the AIM institute, Pasadena, for their hospitality during their workshop on the Bochner Technique. The authors would also like to thank Jiewon Park for her interest in this work, as well as for her very interesting lecture series on applications of {\L}ojasiewicz inequalities held at the AMSI Winter School at the University of Queensland in July 2026, which was particularly helpful for proving Claim 2 of Theorem \ref{thm_subconbra}. JS would like to thank Jeffrey Streets for insightful conversations during his visit to UCI in 2026, and feedback on an initial draft. The authors would also like to thank Luigi Vezzoni for useful discussions during the preparation of the present  work, and Jorge Lauret for pointing out some inaccuracies in a first draft that ultimately led to an improvement of the main theorems.

EF was supported by GNSAGA of INdAM. RL was supported by the Australian Research Council grants DP240101772 and FT240100361. JS was partially supported by the Deutsche Forschungsgemeinschaft (DFG, German Research Foundation) under Germany's Excellence Strategy EXC 2044 –390685587, Mathematics Münster: Dynamics–Geometry–Structure.

\section{Preliminaries}\label{sec:preliminaries}
 We will review briefly some of the basic concepts in generalized geometry which will be useful throughout the paper. Our treatment is based on \cite{GRFbook, HGRF}, where all the proofs and details can be found.
\subsection{Exact Courant algebroids}

Let $M^n$ be a smooth manifold. An \emph{exact Courant algebroid} (ECA) over $M$ is a rank-$2n$ vector bundle $E \to M$ equipped with the following data:
\begin{itemize}
	\item A non-degenerate, symmetric bilinear form $\ip$ of signature $(n,n)$, called the \emph{neutral inner product};
	\item A bracket $\lb$ on $\Gamma(E)$, called the \emph{Dorfman bracket};
	\item A bundle map $\pi : E \to TM$, called the \emph{anchor map};
\end{itemize}
subject to the following, for all $a,b,c \in \Gamma(E)$ and $f\in C^\infty(M)$:
\begin{enumerate}[({ECA}{\arabic*})]
	\item\label{ECA1} $[a,[b,c]] = [[a,b],c] + [b,[a,c]]$;
	\item\label{ECA2} $\pi[a,b] = [\pi a, \pi b]$;
	\item\label{ECA3} $[a,fb] = f[a,b] + \pi(a) fb$;
	\item\label{ECA4} $\pi(a)   \la b,c \ra  = \la [a,b], c \ra + \la b, [a,c]\ra$;
	\item\label{ECA5} $[a,b] + [b,a] =  (\pi^* \circ \dd) \la a,b \ra$;
	\item there is an exact sequence of vector bundles 
	\begin{equation}\label{eqn_exactseq}
		 0 \longrightarrow  T^*M \overset{\pi^*}\longrightarrow E \overset{\pi} \longrightarrow TM \longrightarrow 0 \, .
	\end{equation}
\end{enumerate}
An isomorphism of ECAs is a vector bundle isomorphism preserving all the data in the obvious way, see \cite[Definition 2.11]{GRFbook}. One can also \emph{scale} ECAs, by defining $c\cdot E$ for $c\in \mathbb{R}\backslash\{0\}$ to be the ECA associated to the data $(c \ip, \lb, \pi)$,  see \cite[Lemma 2.5]{HGRF}.
 
 As a fundamental example of an ECA, we have the $H$-\emph{twisted generalized tangent bundle}.
 \begin{ex}\label{Htwisted}
 We can endow the \emph{generalized tangent bundle} $TM\oplus T^*M$ with a structure of exact Courant algebroid  by defining, for any $X+\xi, Y+\eta\in\Gamma(TM\oplus T^*M)$,  
 $$
\la X+\xi, Y+\eta\ra:=\frac12(\xi(Y)+\eta(X))\,, \quad [X+\xi, Y+\eta]_{H}:=[X, Y]+\mathcal L_X\eta-\iota_Yd\xi+\iota_Y\iota_XH\,, \quad \pi(X+\xi):=X\,,
 $$ where $H\in \Omega^3(M)$ satisfies $\dd H=0$. The resulting ECA will be denoted by $(T\oplus T^*)_H$.
 \end{ex}
  \Cref{Htwisted} is not only the easiest example  of ECA one can think of, but it is also essentially the unique one, at least  up to the choice of an isotropic splitting $\sigma : TM \to E$ for \eqref{eqn_exactseq}:
 \begin{prop}\label{prop_isotropicsplitting}
 Let $E$ be an ECA. Then, any isotropic splitting $\sigma$ for \eqref{eqn_exactseq} gives rise to an isomorphism of ECAs
 \[
\TTH \simeq_{\sigma} E, \qquad X + \xi  \, \,  \mapsto  \,\, \sigma X + \tfrac12 \pi^* \xi,
 \] 
 where $H\in \Omega^3(M)$ is defined via
 \[
 H(X, Y, Z):=2\la[\sigma X, \sigma Y], \sigma Z \ra, \qquad X,Y, Z \in \Gamma(TM).
 \]
 \end{prop}
 Choosing two different isotropic splittings $\sigma, \sigma'$ of \eqref{eqn_exactseq} allows us to define 
\begin{equation}\label{eqn_exprb}
 b:=(\pi^*)^{-1}(\sigma-\sigma')\in \Omega^2(M)\,.
 \end{equation}  Then, applying  \Cref{prop_isotropicsplitting}  with the choice of $\sigma'$ yields
 $$
E\simeq_{\sigma'}(\TT)_{H+\dd b}\,.
 $$ More in general, the following classification of exact Courant algebroids holds:
 \begin{thm}\cite{Sletters}
 Isomorphism classes of ECAs over a given manifold $M$ are in one-to-one correspondence with $H^{3}(M, \R)/\Gamma_M$, where $\Gamma_M={\rm Diff}(M)/{\rm Diff}_0(M)$ is the mapping class group of $M$.
 \end{thm}

\subsection{Generalized Riemannian geometry} A \emph{generalized metric} on an ECA $E\to M$ is an endomorphism $\mc G\in \msf O(E, \la\cdot, \cdot\ra)$ such that 
$$
(a, b)\mapsto \la\mc G a, b \ra
$$ 
is symmetric and positive definite  on $\Gamma(E)$. The choice of a generalized metric on  $E$ turns out to be equivalent to a choice of  Riemannian metric $g$ on $M$ and a preferred isotropic splitting $\sigma$, which yields an isometry:
\begin{equation}\label{eqn_isoprefisotropic}
F : \left(\TTH, \mc G(g, 0)\right)\overset{\simeq}\longrightarrow  (E, \mc G) \,,
\end{equation}
where  $H\in \Omega^{3}(M)$, $\dd H=0$, and 
$$
\mc G(g, 0):=\begin{pmatrix} 0& g^{-1}\\
g& 0\end{pmatrix}\,.
$$  


Any other generalized metric on $\TTH$ can be obtained by conjugating $\mc G(g, 0)$ with a $B$-\textit{field transformation}, i.e. it is of the form 
\begin{equation}\label{eqn:gca_gb}
    \mc G(g, b):=e^b\mc G(g, 0)e^{-b}\,, \quad b \in \Omega^2(M)\,.
\end{equation}
Automorphisms of any ECA can be  constructed from a given diffeomorphism of the manifold $M$. Indeed, for any $f \in {\rm Diff}(M)$, we can consider the vector bundle isomorphism covering $f$ 
$$
\bar f \colon \TTH\to \TTH\,,\quad \bar f :=\begin{pmatrix}
\dd f & 0 \\0 & (f^{-1})^*
\end{pmatrix}\,,
$$ which turns out to be an automorphism of $\TTH$ if and only if $f^*H=H$.
In \cite[Proposition 2.2]{GuaAnnals} it is proved that  the automorphism group of any ECA consists only in maps which are the composition of a $B$-field transformation and one of the just defined maps. 

\begin{prop}\label{prop_aut}
Let  $H \in \Omega^3(M)$ be a closed $3$-form. Then, the automorphism group of $\TTH$ is given by:
$$
\mathrm{Aut}(\TTH)=\{\bar fe^b\, \, |\, \, f\in \mathrm{Diff}(M)\,,\, b\in \Omega^2(M)\,, f^*H=H+\dd b\}\,.
$$ Moreover, if $g$ is a Riemannian metric on $M$, then the isometry group of $(\TTH, \mc G(g, 0))$ is given by $$
\msf{Iso}(\TTH, \mc G(g, 0))=\{\bar f\,\,|\,\, f\in \msf{Iso}(M, g)\,, \, f^*H=H\}\,.
$$\end{prop}

We now recall that, given a generalized metric $\mc G $  on an ECA $E$, we can define  the \textit{generalized Ricci tensor}
$$
\mc{R}c(\mc G)\in\f{so}(E, \ip)\,, 
$$ which encodes both information on the curvature of the underlying manifold $M$, and on the structure of  $E$ as ECA. We refer to \cite{GF19, GRFbook} for a detailed description of the  generalized Ricci tensor.  To understand the relevance of the generalized Ricci curvature, we give an explicit expression  in the case  $E=\TTH$ and $\mc G=\mc G(g, 0)$. A standard computation (see e.g.~\cite[Proposition 3.30]{GRFbook}) yields
\begin{equation}\label{eqn_genRicci}
\mc {R}c(\mc G (g, 0))=\begin{pmatrix}g^{-1}{\rm Ric}_{g, H}^B & \frac12g^{-1}\dd_g^*Hg^{-1}\\
-\frac12\dd^*_{g}H & - {\rm Ric}_{g, H}^{B}g^{-1}
\end{pmatrix}\,, 
\end{equation} where ${\rm Ric}_{g, H}^{B}:={\rm Ric}_g-\frac14 H^2$ and $H^2(X, Y):=g(\iota_XH, \iota_YH)$, for any $X, Y \in\Gamma(TM)$.  In these terms, the generalized Ricci tensor fully encodes the Ricci tensor of the \textit{Bismut connections} $\nabla^{\pm}$, i.e. the unique connections such that 
$$
\nabla^{\pm}g=0\,, \quad gT^{\pm}=\pm H\,,
$$ where we denoted with $T^{\pm}$ the torsion of $\nabla^{\pm}$. Indeed, it turns out that ${\rm Ric}_{g, H}^{B}$ and $\pm\dd_g^*H$ are, respectively, the symmetric and skew-symmetric part of the Ricci tensor of $\nabla^{\pm}$. 
Finally, the \textit{generalized scalar curvature} of $(\TTH, \mc G(g, 0))$ is defined by
$$
\mc S_{g, H}={\rm scal}(g)-\frac{1}{12}|H|^2\,,
$$ where ${\rm scal}(g)$ is the Riemannian scalar curvature, see \cite{GRFbook, GF19}. The generalized scalar curvature  has been shown to play a key role in many aspects of generalized geometry, see for instance \cite{ASY24, HGRF, G21, TOVSEW06,GRFEntropyStreets}.

We are now in a position to define the main object in study of this article:

\begin{defn}
Let $E$ be an ECA. A one-parameter family of generalized metrics $(\mc G_t)_{t\in I}$ is a solution of the  \textit{generalized Ricci flow} if the following evolution equation is satisfied:
\begin{equation}\label{eqn_GRF}
\mc G^{-1}\frac{\partial }{\partial t}\mc G=-2\mc Rc(\mc G)\,.
\end{equation}
\end{defn}
In terms of the isotropic splitting canonically defined by the initial generalized metric  $\mc G(0)$, the generalized Ricci flow \eqref{eqn_GRF} can be written as a coupled flow for a Riemannian metric $g$ and a closed $3$-form $H$ on $M$:
\begin{equation}\label{eqn_GRFclassical}
\left\{
\begin{aligned}
\frac{\partial}{\partial t }g&=\, -2{\rm Ric}_{g, H}^B\,, \quad  g(0)=g_0\,, \\
\frac{\partial}{\partial t}H&= \,\Delta_gH\,, \quad  H(0)=H_0\,.
\end{aligned}
\right.
\end{equation} We remark here that \Cref{eqn_GRFclassical} firstly appeared in the mathematical physics literature in the context of string theory with the name of  \textit{renormalization group flow}, see \cite{CFMP85, ST17}.

The fixed points of the generalized Ricci flow are Bismut-Ricci flat metrics, i.e. generalized metrics $\mc G$ satisfying $\mc Rc(\mc G)=0.$ In terms of classical data, solving the previous equation amounts to ask for a Riemannian metric $g$ and a closed $3$-form $H$ such that 
$$
{\rm Ric}_{g, H}^B=0\,,\quad \dd^*_gH=0\,.
$$ In this case, we will refer to the pair $(g, H)$ as a \emph{Bismut-Ricci flat pair}.
Apart from Ricci flat metrics with $H=0$, a large class of Bismut-Ricci flat pairs is given by \emph{Bismut-flat pairs} $(g, H)$, i.e. those pairs whose Bismut connections $\nabla^{\pm}$ are flat.  Such pairs were firstly classified by Cartan and Schouten in \cite{CS26, CS26bis} and, more recently, in \cite{AgrF10}. It turns out that complete and simply-connected manifolds endowed with  such pairs are isometric to a compact semisimple Lie group with a bi-invariant metric $g_{{\rm b}}$ and the  $3$-form given by $H^{\pm}_{g_{{\rm b}}}(X,Y, Z)=\pm g_{{\rm b}}([X, Y], Z)$, for any left-invariant vector field $X,Y, Z$.  Examples of non-flat homogeneous Bismut-Ricci flat pairs were recently constructed in \cite{CCZ26, LW23, PR23, PR24}. As in the introduction, we will refer to the pair $(g_{{\rm b}}, H^{\pm}_{g_{{\rm b}}})$, or equivalently to $\mc G(g_{{\rm b}}, 0)$ on $(\TT)_{H^{\pm}_{g_{{\rm b}}}}$, as a \emph{standard bi-invariant Bismut-flat metric}.

To conclude this subsection, we recall the definition of generalized Ricci solitons, in its more general version introduced in \cite[Definition 3.1]{HGRF} and  \cite{PR25}:
\begin{defn}
A solution $(\mc G_t)_{t\in I}$ of the generalized Ricci flow on an ECA $E$ is a \textit{generalized Ricci soliton} if  there exists  $c(t)>0$, $c(0)=1$, and a one-parameter family of generalized isometries
$$
F(t)\colon (c(t)\cdot E, \mc G(0))\to (E, \mc G_t)\,,
$$ where the exact Courant algebroid $c(t)\cdot E$ is obtained from $E$ by just rescaling the neutral inner product by a factor of $c(t)$, see \cite[Definition 2.6]{HGRF}. 
\end{defn}
In terms of the  preferred isotropic splitting associated to $\mc G(0)$,  $(\mc G_t)_{t\in I}$ is a generalized Ricci soliton if and only if  there exists $\lambda \in\R $ and $X\in \Gamma(TM)$ such that 
 $$
 \left\{
 \begin{aligned}
 {\rm Ric}_{g_0, H_0}^B=&\, \lambda g_0+\frac12\mathcal L_Xg_0\,, \\
 \Delta_{g_0}H_0=&\, -2\lambda H_0-\mathcal L_XH_0,
 \end{aligned}
 \right.
 $$ see \cite[Proposition 3.4]{HGRF} and \cite{PR25}.

\subsection{Generalized geometry with symmetries}\label{sec:homog}
We now briefly state some basic background information on generalized geometry in the presence of an isometric group action preserving all the structure. We have chosen to present this in a rather more general framework as it may be of independent interest. Later in this section, we will restrict ourselves to the setting that will be of interest in this article, namely that of generalized metrics invariant under a transitive Lie group action.  For  details on this case, see \cite[$\S$4]{HGRF}. 


\begin{defn}\label{defn_GECA}
Let $E\to M$ be an ECA, and assume that there is a Lie group $\G$ acting properly and smoothly on $M$. 
We say $E\to M$ is \emph{$\G$-invariant}, if the $\G$-action on $M$ lifts to an action on $E$ by ECA automorphisms. When $M=\G$, and the action of $\G$ on itself is by left multiplication, we call the ECA \emph{left-invariant}. Furthermore, a generalized metric $\gca$ on $E$ is called \emph{$\G$-invariant} if the action of $\G$ on $E$ is by isometries. 
\end{defn}

An example of $\msf G$-invariant ECA is given by the generalized tangent bundle $(\TT)_H$, endowed  with a $\G$-invariant, closed $3$-form $H\in \Omega^3(M)^\G$, cf.~Example \ref{Htwisted}. Here  the $\G$-action lifts naturally to $\TT$ by linearization. Conversely, we have the following invariant version of Proposition \ref{prop_isotropicsplitting}:

\begin{prop}\label{prop_GinvTT}
Let $E$ be a $\G$-invariant ECA. Then, there exists a $\G$-invariant $H\in \Omega^3(M)^\G$ and a $\G$-equivariant ECA isomorphism $E\simeq (\TT)_H$. Moreover, a  $\G$-invariant generalized metric $\gca$ on $E$ corresponds under said isomorphism to $\gca(g,b)$ (see \eqref{eqn:gca_gb}), where $g, b$ are  $\G$-invariant tensors on $M$. 
\end{prop}

\begin{proof}
Applying \Cref{thm_G-invgenmet}, we consider a $\msf G$-invariant generalized metric $\mc G$. By $\msf G$-invariance,  the associated eigenbundles $E_{\pm}=\ker(\mc G\mp{\rm Id})$ are $\msf G$-invariant.  Moreover, since $\msf G$ acts by ECA isomorphisms, the map $\sigma_{\pm}:=\pi|_{E_{\pm}}^{-1}\colon TM \to E_{\pm}$  is $\msf G$-equivariant, using that   $\pi\circ F=\dd f \circ\pi$,  for any  $F\in {\rm Aut}(E)$ covering $f\in {\rm Diff}(M)$. This allows to infer that $\tau_{\pm}:=\la\sigma_{\pm}\cdot, \sigma_{\pm}\cdot\ra$ is $\msf G$-invariant and consequently
$$
\sigma=\sigma_{\pm}-\frac12\pi^*\tau_{\pm}
$$ is  a $\msf G$-equivariant isotropic splitting. Therefore, using  \Cref{prop_isotropicsplitting}, the associated ECA isomorphism $$
 F\colon \TTH\to E\,, \qquad F(X+\xi)=\sigma X+\frac12\pi^*\xi
$$ is  $\msf G$-equivariant,  since  $\msf G$ acts by automorphisms on $E$ and $\sigma$ is $\msf G$-equivariant. Again exploiting \Cref{prop_isotropicsplitting},  $H(\cdot,\cdot,\cdot)=2\la[\sigma \cdot,\sigma \cdot],\sigma \cdot\ra$, which is thus $\msf G$-invariant.   If $\sigma_1$ is another $\msf G$-equivariant isotropic splitting, then the assertion on $b$ follows from \Cref{eqn_exprb}. Finally, since $g:=\la\mc G\sigma_{+}\cdot, \sigma_{+}\cdot\ra$, we conclude the proof.
\end{proof}

\subsubsection{Left-invariant exact Courant algebroids}
From now on, we let $E\to \G$ be a left-invariant ECA. With respect to a $\G$-equivariant isotropic splitting $\sigma$, Proposition \ref{prop_GinvTT} gives
 \[
    (E, \gca) \simeq_\sigma \left(\TTH, \mc G(g, b)\right) ,
 \]
 where $g, H, b$ are left-invariant tensors on $M=\G$, and $H$ is independent of $\gca$.  Since $\sigma$ and $\pi$ are $\G$-equivariant, this isomorphism is also equivariant when considering the natural lifted action of $\G$ on $\TT$. With respect to the natural vector bundle isomorphism $\TT \simeq \G\times (\gggo)$, said action corresponds simply to left-multiplication on the $\G$-factor. Hence, a left-invariant ECA $E\to \G$ is completely determined by the following \emph{infinitesimal data} at the Lie algebra level:
 \begin{equation}\label{eqn_infinitesimal_ECA}
    E \to \G \quad\leftrightsquigarrow \quad ((\Ggo, \mud) \overset{\pi_1}{\to} (\ggo, \mu), \ip) \quad \leftrightsquigarrow \quad (\Ggo, \mud),
 \end{equation}
where $\f G:=\f g\oplus \f g^*$,  $\mu$ denotes the Lie bracket of $\ggo$,  $\mud = \mu_H$ is the $H$-twisted Dorfman bracket 
\begin{equation}\label{eqn:H_twist_bracket}
    \mu_H(X + \xi, Y + \eta) := \mu(X,Y) - \eta \circ \ad_\mu(X) + \xi \circ \ad_\mu(Y) + \iota_Y \iota_X H, \qquad \ad_\mu(X) := \mu(X, \cdot),
\end{equation}
and the neutral inner product $\ip$ and projection $\pi_1$ are given by: 
\begin{equation}\label{eqn:canonical_ip_pi}
  \la X+\xi, Y+\eta\ra:=\frac12(\xi(Y)+\eta(X)), \qquad \pi_1(X+\xi) := X.
\end{equation}
Recall that, in this case, thanks to (ECA5) and invariance, the Dorfman bracket $\mud$ is skew-symmetric, hence a Lie bracket. In later sections, we will also use the notation $\mu_H$ for skew-symmetric tensors $\mu$ that do not necessarily satisfy the Jacobi identity, and for $3$-forms that are not necessarily closed. 

In this setting, we will exploit the \emph{moving brackets} framework, see \cite{ Lau06, nilRF}. Let $\mathbf G_{\mud}$ be the  simply-connected Lie group with Lie algebra $(\f G, \mud)$. 
Any  $L\in\msf{GL}(\f G)$   yields an isometric isomorphism of metric Lie algebras
$$
L\colon (\f G ,\mud,  L^{-1}\cdot h)\overset{\simeq}\longrightarrow (\f G, L\cdot \mud, h),
$$ where 
\[
    (L \cdot \mud)(\cdot, \cdot):=L\mud(L^{-1}\cdot, L^{-1}\cdot), \quad (L^{-1} \cdot h )(\cdot,\cdot) := h(L \cdot, L \cdot), \quad \mud\in \Lambda^2(\Ggo^*)\otimes \Ggo, \quad h\in \Sym^2_+(\Ggo^*), \quad L\in \msf{GL}(\Ggo).
\]
The  above map can then be integrated to an isometric isomorphism of Lie groups:
\begin{equation}\label{eqn_movingbrack}
 (\mathbf G_{\mud},L^{-1}\cdot h)\simeq (\mathbf G_{L\cdot\mud}, h_{L\cdot \mud}),
\end{equation}
where   $h_{L\cdot \mud}$ is the metric obtained by left-invariant extension  of $h$ on $\mathbf G_{L\cdot\mud}$.  Moreover, since the  natural change-of-basis $\msf{GL}(\f G)$-action
on the space $\Sym^2_+(\f G^*)$ of  inner products on $\f G$  is transitive, the right-hand side in \eqref{eqn_movingbrack} parametrises all left-invariant metrics on $\mathbf G_\mud$ as $L$ varies within $\msf{GL}(\Ggo)$. If we think of $h$ as a fixed, \emph{background} inner product on $\Ggo$, then the only variable is the Lie bracket's structure constants $L\cdot \mud$. 
Motivated by this, in the next sections,   we will  sometimes fix a background metric and denote curvature quantities associated to  $(\mathbf G_{L\cdot\mud}, h_{L\cdot \mud})$ with a subscript highlighting the bracket being considered. 

We note here that  $\msf{GL}(\f G)$  does not preserve the space of left-invariant generalized metrics on $(\f G, \ip)$. However, it is not hard to check that the group $\msf{O}(\f G, \ip )$ acts transitively on such space.   As we shall discuss, the latter action is also related to the generalized Ricci curvature, see \Cref{sec_mm}.  We end this section  by setting some notation we will use throughout the paper.   Given the natural actions of $\msf{GL}(\f G)$ and $\msf{GL}(\f g)$ on, respectively, $\Lambda^2\f G^*\otimes \f G$,  $\Lambda^2\f g^*\otimes \f g$  and $\Lambda^3 \f g^*$, we will denote the corresponding  infinitesimal actions (Lie algebra representations) with $\Theta$, $\theta$ and $\rho$, respectively, see \cite[$\S$4.2]{HGRF} for  explicit definitions.
 


\section{A formula for the generalized Ricci curvature of invariant metrics}\label{section_formulaRc}

Let $E\to \G$ be a left-invariant ECA, determined through a $\G$-equivariant isotropic splitting $\sigma$ by the infinitesimal data  $(\Ggo = \gggo, \mud =  \mu_H)$ as per Section \ref{sec:homog}. Let $\mca^\Ggo := \Sym_+^2(\Ggo^*)$ be the space of positive-definite inner products on $\Ggo$, and let
\[
    \mca_\gen^\Ggo := \{\mc G \in \mathsf{SO}(\Ggo,\ip) \,\, | \,\,  \mc{G}^2 = \id, \,\, \m{\mc G \cdot}{\cdot} > 0\}
\]
be the space of invariant generalized metrics. We consider it as a subset of $\mca^\Ggo$ via 
\[
   \mca^\Ggo_\gen \hookrightarrow \mca^\Ggo, \qquad   \gca \mapsto 2 \, \la \gca \cdot, \cdot \ra.
\]
We include a factor of $2$ in this identification so that, in the case $$\mc G = \begin{pmatrix}0&g^{-1}\\g&0\end{pmatrix},$$ an orthonormal basis with respect to $\mc G \in \mc M^{\Ggo}$ is given by $\{e_i\}\cup \{e^i\}$, where $\{e_i\}$ is an $g$-orthonormal basis  and $\{e^i\}$ is the dual basis. The space $\mca^\Ggo$ is a symmetric space of non-positive curvature, with  symmetric metric given by 
\begin{equation}\label{eqn:g^sym}
    g^\sym_h(A,B) := \tr(AB), 
\end{equation} 
under the identification 
\begin{equation}\label{eqn:T_hM}
    \operatorname{Sym}(\Ggo,h) \cong T_h\mc{M}^\Ggo;\qquad A \mapsto h(A\cdot,\cdot).
\end{equation}
It is easily checked that $\mca_\gen^\Ggo \subset \mca^\Ggo$ is a totally geodesic submanifold, and using \eqref{eqn:T_hM} we have
\[
    T_{\mc G} \mca_\gen^\Ggo = \f{so}(\Ggo,\ip)\cap\operatorname{Sym}(\Ggo, \mc G),
\]
where $\operatorname{Sym}(\Ggo, \mc G) := \operatorname{Sym}(\Ggo, 2\m{\mc G \cdot}{\cdot})$. Finally, we will denote with 
$(\cdot)_{\f{so}} \colon T\mc{M}^\Ggo \to T\mc{M}_\gen^\Ggo$
 the orthogonal projection onto $T\mc{M}_\gen^\Ggo$ with respect to  $g^{\sym}$.

Any inner product $h \in \mc{M}^\Ggo$ corresponds to a left-invariant Riemannian metric --also denoted by $h$-- on $\mathbf G_{\mud}$, i.e.  the simply-connected Lie group with Lie algebra $(\Ggo,\mud)$. We denote by $\operatorname{Ric}_h \in \operatorname{Sym}(\Ggo,h) \simeq T_h \mc{M}^\Ggo$ the (classical) Ricci operator of  $h$. We also define the symmetric endomorphism $B_h \in T_h \mc{M}^\Ggo$ by $h(B_h\cdot,\cdot) := \kappa_{\mud}$, where $\kappa_\mud$ is the Killing form of $(\Ggo,\mud)$. Moreover, for a generalized metric $\mc G \in \mc{M}_{\gen}^{\Ggo}$, we denote its generalized Ricci curvature (at the identity) by $\mc Rc(\mc G) \in \f{so}(\Ggo,\ip)\cap \operatorname{Sym}(\Ggo,  \gca) = T_{\mc G}\mc{M}^\Ggo_\gen$.

Recall also that, as in Section \ref{sec:homog}, any invariant generalized metric $\mc G \in \mc{M}^\Ggo_\gen$ induces a closed $3$-form $H \in \Lambda^3 \f g^*$ and an inner product $g \in \Sym_+^2(\ggo^*)$. As a final piece of notation, given $A\in \End(\ggo)$ we define $\B{A} := A \oplus -A^* \in \sog(\Ggo, \ip)$. Although it may seem that this extension depends on a particular isotropic splitting, in fact it can be seen that $\bar A = (\hat A)_{\f{so}}$, where $\hat A:=\pi^*_{\mc G}g A\pi$. Here, $\pi^*_{\mc G}\colon T^*M\to E$ is constructed by composing  $\pi^* : T^*M \to E^*$ with the isomorphism of $E^*\simeq E$ given  by the inner product $2\langle\mc G\cdot, \cdot\rangle$.



\begin{thm}\label{thm:formula}
The generalized Ricci curvature of an invariant generalized metric $\gca \in \mc{M}_\gen^{\f G}$  satisfies
\begin{equation}\label{eqn:formulaRcgca}
	\mc Rc(\mc G) = \left(\frac 23 \operatorname{Ric}_{\mc G}- \frac 16 B_{\mc G} - \operatorname{Sym}_{\gca}(\ad^{\mud}_U)\right)_{\f{so}},
\end{equation}
where 
$U := g^{-1} \tr \ad_{(\cdot)} \in \f g$ is the mean curvature vector of the metric Lie algebra $(\f g,\mu,g)$, and $g$ is the metric on $\f g$ induced by the isotropic splitting for $\mc G$.  Moreover, the generalized scalar curvature $\mc S$ of $\mc G$ satisfies
\begin{equation}\label{eqn:genS}
	 \mc S(\mc G) = \frac 12 \tr \left(\frac 23 \operatorname{Ric}_{\mc G}- \frac 16 B_{\mc G}\right) - |U|^2.
\end{equation}
\end{thm}

\begin{rmk}\label{rmk_mmclassic}
Before projecting onto $\f{so}(\f G,\ip)$, the right-hand side of \eqref{eqn:formulaRcgca} looks remarkably similar to the formula for the Ricci curvature of a homogeneous Riemannian manifold \cite[7.38]{Bss}:
\begin{equation}\label{eqn:RiemRicg}
    {\rm Ric}_g   = M_g - \frac12 B_g - \Sym_g(\ad_U).
\end{equation}
Here $M_g$ is the moment map term, defined implicitly by: 
$$
{\rm tr}(M_gL)=\frac18\frac{\dd}{\dd t}\Big|_{t=0}|\exp(tL)\cdot\mud|^2_g=\frac14g(\Theta(L)\mud, \mud), \quad L\in \Sym(\f G, g)\,.
$$ 
\end{rmk}

\begin{proof}
Without loss of generality, we work in the canonical isotropic splitting induced by $\gca$ (which is indeed $\G$-equivariant), and assume $\mc G = \begin{pmatrix}0&g^{-1}\\g&0\end{pmatrix}$. 
Then,  $\{e_i\}\cup \{e^i\}$ is an orthonormal basis for $(\Ggo, \gca)$ where $\{e_i\}$ is an orthonormal basis for $(\ggo,g)$ and $\{e^i = ge_i\}$ its dual basis. To prove \eqref{eqn:formulaRcgca}, it suffices to show that
\begin{equation}\label{eqn:formulaRicZ}
    \tr \left( \mc Rc (\gca) Z \right) = \tr \left(\frac23 {\rm Ric}_\gca - \frac16 B_\gca - \Sym_\gca(\ad^{\mud}_U)  \right)Z,
\end{equation}
for every $Z\in \sog(\Ggo,\ip)\cap \Sym(\Ggo,\gca)$. Such a $Z$ is of the form 
\[
    Z = \begin{pmatrix}A&-g^{-1}\alpha g^{-1}\\ \alpha & -A^*\end{pmatrix}, \qquad A\in \Sym(\ggo,g) \subset \End(\ggo), \quad \alpha \in \Lambda^2 \ggo^*.
\]
By \eqref{eqn_genRicci}, the generalized Ricci curvature is given by
\[
	\mc Rc(\mc G) = \begin{pmatrix}\operatorname{Ric}_g - \frac 14 H^2 & \frac 12 g^{-1}\dd_\mu^* H g^{-1}\\ -\frac 12 \dd_\mu^*H&-(\operatorname{Ric}_g - \frac 14 H^2)^*\end{pmatrix}.
\]
Thus, using \eqref{eqn:RiemRicg}, the left-hand side  in \eqref{eqn:formulaRicZ} becomes
\begin{align*}
\operatorname{tr}(\mc Rc(\mc G) Z) 
        &= 2 \tr \operatorname{Ric}_g A - \frac 12 \tr A H^2 -   g(\alpha,\dd_\mu^*H)\\
        &= \frac 12 \,  g(\theta(A)\mu,\mu) - \kappa_\mu(Ae_i,e_i) - 2\tr(\ad_U A) - \frac 12 \tr A H^2 -  g(\alpha,\dd_\mu^*H).
\end{align*}

Regarding the right-hand side, notice first that $\ggo^* \subset \Ggo$ is an abelian ideal for $\mu_H$, thus it is contained in the radical of $\kappa_{\mu_H}$. On the other hand, it is easy to see that, for all $X\in\f g$, $\kappa_{\mu_H}(X,X) = 2 \, \kappa_\mu(X,X).$ It follows that 
\begin{equation}\label{eqn:KillingmuH_mu}
    B_\gca = \begin{pmatrix}2 \, B_g& 0\\ 0 & 0\end{pmatrix}.
\end{equation}
Also, 
\[
	\tr(B_{\mc G}Z)= \sum_{i} \kappa_{\mu_H}(Z e_i,e_i) + \sum_{i}\kappa_{\mu_H}(Z e^i,e^i) = 2 \sum_{i}\kappa_\mu(Ae_i,e_i).
\]
In particular, for the unimodular metric Lie algebra $(\Ggo, \mu_H,  \gca)$,   \eqref{eqn:RiemRicg} gives
\[
	\frac23 \tr(\operatorname{Ric}_{\mc G} Z) = \frac{1}{6}  g (\Theta(Z)\mu_H,\mu_H) - \frac23 \sum_{i=1}^n \kappa_{\mu}(Ae_i,e_i).
\]
The first term is given by 
\begin{align*}
	 \frac16 \,   g \left(\Theta(Z)\mu_H,\mu_H\right) &= \frac16 \,  g\left(\Theta \begin{pmatrix}A & 0\\ \alpha & -A^*\end{pmatrix}\mu_H,\mu_H\right) +  \frac16   \, g\left(\Theta \begin{pmatrix}0 & -g^{-1}\alpha g^{-1}\\0 & 0\end{pmatrix}\mu_H,\mu_H\right)\\
	 &=  \frac16 \,  g\left(\Theta \begin{pmatrix}A & 0\\ 2\alpha & -A^*\end{pmatrix}\mu_H,\mu_H\right)\\
	 &= \frac16 \, g({\theta(A)\mu}_{\rho(A)H - 2\dd_\mu\alpha},\mu_H)\\
     &= \frac12 \,  g(\theta(A)\mu,\mu) + \frac16 g(\rho(A)H,H) - \frac13 g(\dd_\mu \alpha , H)\\
	 &= \frac12 \,  g(\theta(A)\mu,\mu) -\frac12\tr A H^2-  g(\alpha , \dd_\mu^* H)+ g(\alpha, \iota_U H),
\end{align*}
where in the third equality we used \cite[Corollary 4.8]{HGRF}, while in the last equality we used \eqref{eqn_convdstar}.
Finally, using \eqref{eqn:H_twist_bracket}, we note that 
\begin{equation}\label{eqn_adU}
\ad^{\mud}_{U}=\begin{pmatrix}\ad_U& 0 \\ \iota_UH& -\ad_{U}^*\end{pmatrix}\,.
\end{equation} Thus,
$$
{\rm tr}(\Sym_{\mc G}(\ad_U^{\mud})Z)=2{\rm tr}(\ad_UA)-{\rm tr}(g^{-1}\iota_UHg^{-1}\alpha)=2{\rm tr}(\ad_UA)+g(\iota_UH, \alpha)\,.
$$

Putting everything together yields \eqref{eqn:formulaRcgca}. As for  \eqref{eqn:genS},  using \Cref{rmk_mmclassic} and the unimodularity of $(\f G, \mu_H)$, we have that 
$$
\begin{aligned}
\frac12{\rm tr}\left(\frac 23 \operatorname{Ric}_{\mc G}- \frac 16 B_{\mc G}\right)=&\, \frac12{\rm tr }\left( \frac23 M_{\mc G}- \frac12B_{\mc G}\right)=-\frac{1}{12}| \mu_H|^2- \frac12{\rm tr}B_g\\
=&\, -\frac14|\mu|^2-\frac{1}{12}|H|^2-\frac12\tr B_g=\mc S({\mc G})+ |U^2|\,,
\end{aligned}
$$
concluding the proof.
\end{proof}
Using  \Cref{thm:formula}, we can also  describe alternatively the generalized Ricci curvature  as the  negative gradient of the generalized scalar curvature on the space of left-invariant generalized metrics on $(\f G, \ip)$, up to a multiplicative factor. As in the classical case, see \cite[Section 3]{Heb}, this is possible only under the assumption of $\mathsf G$ being unimodular.

\begin{cor}\label{cor_riccigrad}
Let $\msf G$ be a unimodular Lie group.  Then, 
 the generalized Ricci curvature of  an invariant generalized metric $\mc G$ satisfies
\[
	\mc Rc(\mc G) = - 2(\operatorname{grad}_{g^\sym} \mc S)_{\mc G},
\]
where the gradient is taken with respect to the symmetric metric $g^\sym$ on $\mca_\gen^\Ggo$ (see \eqref{eqn:g^sym}). 
\begin{proof}
Using \eqref{eqn:genS}, together with \Cref{rmk_mmclassic} (for $\mc G$), with $U=0$ due to unimodularity, we  have that 
\[
    \mc S({\mc G}) = \frac12{\rm tr }\left( \frac23 M_{\mc G}- \frac12B_{\mc G}\right).
\]
We now compute using the moving brackets approach (see \cite[$\S$4.2]{HGRF}). Given an element  $Z \in \sog(\Ggo,\ip)\cap{\rm Sym}(\f G, \mc G) \cong T_{2\m{\mc G\cdot}{\cdot}}\mca_\gen^\Ggo$, using the identification \eqref{eqn:T_hM} we have
$$
(\dd\mc S)_{\mc G}(Z)= \frac{\dd}{\dd t}\Big|_{t=0}\mc S(\exp\left(-\tfrac t2 Z\right)\cdot \mc G)=\frac{\dd}{\dd t}\Big|_{t=0}\mc S_{\exp\left(\tfrac t2 Z\right)\cdot \bm \mu}=\frac12\frac{\dd }{\dd t}\Big|_{t=0}{\rm tr}\left(\frac 23 M_{\exp\left(\tfrac t2 Z\right)\cdot \bm \mu} - \frac 12 B_{\exp\left(\tfrac t2Z\right)\cdot \bm \mu}\right)\,.
$$ 
Following \cite[$\S$3.2]{homRF}, we compute the differentials of $M$ and $B$:
\begin{equation}\label{eqn_derivatives}
\begin{aligned}
    \frac{\dd }{\dd t}\Big|_{t=0}{\rm tr}(M_{\exp\left(\tfrac t2 Z\right)\cdot \bm\mu}) & =-\frac{1}{4}
    \frac{\dd }{\dd t}\Big|_{t=0}|\exp\left(\tfrac t2 Z\right)\cdot \bm \mu |^2=-\frac14g(\Theta (Z)\bm \mu, \bm \mu)=-\tr(M_{\bm \mu } Z)\,, \\
    \frac{\dd}{\dd t}\Big |_{t=0}B_{\exp\left(\tfrac t2 Z\right)\cdot\bm \mu} & =  
      \frac{\dd}{\dd t}\Big |_{t=0} \exp\left(- \tfrac t2 Z\right) B_\mud \exp\left(-\tfrac t2 Z\right) = -\frac 12 (B_{\bm \mu}Z + Z B_\mud)\,.
\end{aligned}
\end{equation}
Collecting all and using Theorem \ref{thm:formula} we obtain 
$$
(\dd\mc S)_{\mc G} (Z)=-\frac12\left(\frac23\tr(M_{\bm \mu} Z)-\frac12{\rm tr}(B_{\bm\mu}Z)\right) =- \frac12\tr \left( \mc Rc(\mc G) Z \right),
$$
 and the claim follows from \eqref{eqn:g^sym}.
\end{proof}
\end{cor}

\section{Convergence of the generalized Ricci flow on completely solvable Lie groups} \label{sec_ConvCS}

The main goal of this section is to prove the following.

\begin{thm}\label{thm_main_convGRF}
Let $\G$ be a Lie group whose Lie algebra has a positive semi-definite Killing form, and let $[H] \in H^3(\G,\R)^\G$ be an invariant cohomology class. Then, given any invariant generalized Ricci flow solution $(g_t, H_t)_{t\in (0,+\infty)}$  on $(\G,[H])$ and any sequence  $t_k \to  +\infty$,  the rescaled exact Courant algebroids with rescaled metrics $\left( (\G, \tfrac1{t_k} [H]), (\tfrac1{t_k} g_{t_k}, \tfrac1{t_k} H_{t_k}) \right)$ subconverge, in the generalized Cheeger-Gromov sense, to a left-invariant Courant algebroid with invariant metric $\left( (\G_{\infty}, [H_\infty]), (g_\infty, H_\infty) \right)$ that is a generalized Ricci soliton.
\end{thm}

Notice that the assumption on long-time existence is in fact not a serious one, as this is always satisfied under the symmetry assumptions in Theorem \ref{thm_main_convGRF}, thanks to \cite[Theorem A]{HGRF} and a standard result in Lie theory:

\begin{lem} \label{lem_posKilling}
The universal cover of any Lie group whose Lie algebra has a positive semi-definite Killing form is diffeomorphic to $\R^n$. In particular, invariant generalized Ricci flows on such Lie groups exist for all positive times.
\end{lem}
\begin{proof}
    Let $\msf G$ be a simply-connected Lie group with positive semi-definite Killing form, $\f g$ its Lie algebra, and $\kappa \colon \f g \times \f g \to \R$ its Killing form satisfying $\kappa \geq 0$. Let $\msf K \leq \msf G$ be a simply-connected maximal compact subgroup with corresponding Lie algebra $\f k$. Then, there is an $\operatorname{Ad}(\msf K)$-invariant inner product $g$ on $\f g$ obtained by averaging over $\msf K$. If $\{e_i\}_{i=1}^n$ denotes a $g$-orthonormal basis of $\f g$, then for any $X \in \f k$, we have $0 \leq \kappa(X,X) = \operatorname{tr}(\ad_X^2) = \sum_i g(\ad_X^2 e_i,e_i) = -|\ad_X|^2_g \leq 0$, so that $\msf K$ is abelian, and hence trivial by simply-connectedness. This implies the first claim. The second follows from \cite[Theorem A]{HGRF}.
\end{proof}

To prove \Cref{main_solv}, we  first extend the moment map interpretation for the classical and generalized Ricci curvature of invariant metrics on nilmanifolds, to solvable groups whose Killing form is positive semi-definite. (Recall that for a nilpotent Lie group, the Killing form vanishes.) The latter in particular includes completely solvable groups. This will allow us to construct  monotone quantities for the flow which, together with the compactness properties of Dorfman brackets, and some extra considerations, will be enough to show subconvergence to a soliton metric.

\subsection{Augmented brackets}

In this section, 
 we define the space $\mc A_{\mc D}$ of \emph{augmented Dorfman brackets}, which encodes Lie brackets on $\f g \oplus \f g^*$ with positive semi-definite Killing form, see $\S$\ref{sec_augm}. Using this extension, in $\S$\ref{sec_mm} we write the generalized Ricci curvature as a \emph{moment map} for the action of a real reductive Lie group on $\mc A_{\mc D}$.

\subsubsection{Set up}\label{sec_augm}

Fix the $n$-dimensional vector space $\f g$ and endow $\f G :=\f g \oplus \f g^*$ with the canonical neutral pairing $\m{\cdot}{\cdot}$ and anchor map $\pi \colon \f G \to \f g$ as in \eqref{eqn:canonical_ip_pi}. We fix once and for all a background generalized metric on $\f G$ 
\[
    \B {\mc G} = \begin{pmatrix}0&\B g^{-1}\\\B g&0\end{pmatrix} \in \mathsf{SO}(\f G, \m{\cdot}{\cdot}),
\]
where $\B g$ is a fixed metric on $\f g$. By abuse of notation, we will also denote by $\bar g := 2\m{\B{\mc G}\cdot}{\cdot}$ the induced scalar product on $\f G$. We also set 
\[
    \B g(A,B) := \tr(AB^t), \qquad \B g (\mud,\nud) := \sum_{ij=1}^{2n} \B g(\mud(E_i,E_j),\nud(E_i,E_j)),
\]
for all $A,B \in \f{gl}(\f G)$ and all $\mud,\nud \in \Lambda^2\f G^* \otimes \f G$, where $E_i$ is any orthonormal basis of $(\f G, \bar g)$. Notice that using the notation $\mud = \mu_H$ (see \eqref{eqn:H_twist_bracket}) it follows that
\begin{equation}\label{eqn:norm_mu_H}
    | \mu_H |_{\B g}^2 = 3 \, |\mu|_{\B g}^2 + |H|_{\B g}^2\,.
\end{equation}

Consider the vector space
\[
    V :=  \left( \Lambda^2\f G^* \otimes \f G  \right) \oplus \End(\f G)\,.
\]
The space of \emph{augmented Lie brackets} on $\f G$ is the following real algebraic  subvariety of $V$:
\[
\mc A := \{(\mud,C) \in V \,| \,\mud \text{ satisfies the Jacobi identity and } 3B_{\mud} = C^tC\}.
\]
Recall that $B_{\mud} = {\bar g}^{-1} \kappa_\mud$, i.e.~ $\B g(B_\mud \cdot,\cdot) = \kappa_{\mud}$, where $\kappa_\mud$ is the Killing form of $(\Ggo,\mud)$.  Clearly, $\mc A$ is an algebraic subset of $V$, as both conditions defining it are polynomial. Since any positive semi-definite, symmetric endomorphism has a unique  positive semi-definite, symmetric square root, the space of Lie brackets with non-negative Killing form on $\f g \oplus \f g^*$ is precisely the projection of $\mc A$ onto the first summand. Moreover, $\mathcal A$ is scale-invariant, since $B_\mud$ is quadratic in $\mud$. Finally, we also define the real algebraic subset
\[
    \mc A_{\mc D} := \{(\mud,C) \in  \mc A \,| \, \mud(\f g^*,\f g^*) = 0 \text{ and } \m{\mud(\cdot,\cdot)}{\cdot} \in \Lambda^3\f G^*\} \, \, \subset \, \, \mc A,
\]
to be the set of \emph{augmented Dorfman brackets} on $\f G$. Note that any $(\mud,C) \in \mc A_{\mc D}$ satisfies $\mud = \mu_H$ for some Lie bracket $\mu$  and closed $3$-form $H$ on $\f g$ (see \eqref{eqn:H_twist_bracket}, \cite[Lemma 4.6]{HGRF}). Once again, abusing notation, denote by $\B g$ the inner product on $V$ defined by 
\[
    |(\mud,C)|^2_{\B g}  := |\mud|^2_{\B g} + |C|^2_{\B g}\,, \qquad \mud \in \Lambda^2 \Ggo^* \otimes \Ggo, \quad C \in \End(\Ggo)\,.
\] 
The group $\mathsf{GL}(\f G)$ acts linearly on $V$ via
\[
    h\cdot (\mud,C) := (h\cdot \mud,Ch^{-1}) = (h\mud(h^{-1}\cdot,h^{-1}\cdot),Ch^{-1}), \qquad h \in \mathsf{GL}(\f G), \quad (\mud,C) \in V.
\]
Clearly, the orthogonal group $\mathsf{O}(\f G,\B g) \leq \mathsf{GL}(\f G)$ acts by linear isometries on $(V, \bar g)$. Moreover, we make the following crucial observation:
\begin{lem} The action of $\mathsf{GL}(\f G)$ on $V$ preserves the variety $\mc A \subset V$ of augmented Lie brackets. 
\end{lem}
\begin{proof} Let $(\mud,C) \in \mc A$ be an augmented bracket and $h \in \mathsf{GL}(\f G)$. It is well-known (see e.g.~\cite{homRF}) that $B_{h \cdot \mud} = (h^{-1})^t B_\mud h^{-1}$. It follows that $h\cdot(\mud,C) = (h\cdot \mud,Ch^{-1}) \in \mc A$ as claimed.
\end{proof}
\begin{rmk}
    Note that $\mathsf{GL}(\f G)$ does \emph{not} preserve $\mc A_{\mc D}$. Indeed, after acting with $\mathsf{GL}(\f G)$, ${\mathfrak g}^*$ may cease to be an abelian subalgebra.  Conceptually, this is due to the fact that the ECAs represented by the augmented Dorfman brackets in $\mc A_{\mc D}$ come equipped with a distinguished isotropic splitting.
\end{rmk}

\subsubsection{The moment map for augmented brackets}\label{sec_mm}

We now define the key tool that will aid our analysis. The context for this section is real Geometric Invariant Theory, see \cite{RS90,HSchw07,cruzchica,GIT20}. 

The \emph{moment map} for the action of $\mathsf{GL}(\f G)$ on $V$ is the map
\[
M^{\f{gl}} \colon V \to \operatorname{Sym}(\f G, \B{\mc G}); \quad (\mud,C) \mapsto M^{\f{gl}}_{(\mud,C)},
\]
defined implicitly by
\[
    \B g(M^{\f{gl}}_{(\mud, C)},E) := \frac{1}{8}\frac{\dd}{\dd t}\Big|_{0}\left|\exp(tE)\cdot (\mud,C)\right|^2_{\B g} = \frac 1 4 \B g(\Theta(E)\mud,\mud) - \frac 14 \B g(CE,C);\qquad E \in \operatorname{Sym}(\f G,\B{\mc G}), \quad (\mud,A) \in V.
\]
The moment map for the action of $\mathsf{SO}(\Ggo,\ip) \subset \mathsf{GL}(\Ggo)$ is denoted by $M^{\f{so}}\colon V \to \operatorname{Sym}(\f G, \B{\mc G}) \cap \f{so}(\f G,\m{\cdot}{\cdot})$, and defined analogously, by restricting to those $E \in \operatorname{Sym}(\f G, \B{\mc G}) \cap \f{so}(\f G,\m{\cdot}{\cdot})$ in the above formula. Hence, $M^{\f{so}}_{(\mud, C)}=(M^{\f{gl}}_{(\mud, C)})_{\f{so}}$, see \Cref{section_formulaRc} for the definition of $(\cdot)_{\f{so}}$. 

\begin{prop}\label{prop_mmaug}
For any augmented Dorfman bracket $(\mud,C) \in \mc A_{\mc D}$, it holds that 
\[
M^{\f {gl}}_{(\mud,C)} = \operatorname{Ric}_{\mud} - \frac 14 B_{\mud},
\]
where $\operatorname{Ric}_{\mud}$ is the (classical) Ricci operator of the Riemannian Lie group associated to $(\f G,\mud,\B g = 2 \m{\B{\mc G}\cdot}{\cdot})$.
In particular, 
$$
\mc Rc_{\mud}+\Sym(\ad_{U}^{\mud})=\frac23M^{\f {so}}_{(\mud, C)}\,.
$$
\end{prop}

\begin{proof} Note that the Lie bracket $\mud$, being totally skew-symmetric with respect to the neutral inner product, is unimodular. Hence, using the classical formula for the Ricci curvature on homogeneous spaces \eqref{eqn:RiemRicg} and the fact that $C^tC = 3B_\mud$, we find
\begin{align*}
\B g(\operatorname{Ric}_{\mud}-\frac 14 B_\mud,E) &= \frac 14 \B g(\Theta(E)\mud,\mud) - \frac 34 \tr (B_\mud E) = \frac 14 \B g(\Theta(E)\mud,\mud) - \frac 14 \tr (C^tCE) = \B g(M_{(\mud,C)}^{\f{gl}},E),
\end{align*}
for all $E \in \operatorname{Sym}(\f G,\B{\mc G})$, and the claim follows.
\end{proof}
 In particular, for any $(\mud, C)\in \mc A_{\mc D}$, using \Cref{thm:formula} and \Cref{prop_mmaug},  we have that 
 \begin{equation}\label{eqn_scalgen}
\mc S_{\mud}=-\frac{1}{12}|(\mud, C)|^2-|U_{\mu}|^2\,.
 \end{equation}

\subsubsection{The augmented bracket flow}
We now view the generalized Ricci flow as a flow of augmented Dorfman brackets. Following  \cite{BL17}, we denote by $\mc Rc_{\mud}^\star:=\mc Rc_{\mud}+\Sym(\ad^{\mud}_U)$ the \emph{unimodular generalized Ricci curvature} and by ${\rm Ric}_{\mu}^\star:={\rm Ric}_{\mu}+\Sym(\ad_{U_{\mu}})$ the \emph{unimodular Ricci curvature}. Using \eqref{eqn_adU}, it is fairly easy to see that 
$$
\mc Rc_{\mu_H}^\star=\begin{pmatrix}
{\rm Ric}_{\mu}^\star-\frac14H^2 & \frac12\B g^{-1}(\dd^*_{\mu}H-\iota_{U_{\mu}}H)\B g^{-1}\\
\frac12(-\dd^*_{\mu}H+\iota_{U_{\mu}}H) & -({\rm Ric}_{\mu}^\star-\frac14H^2)^*
\end{pmatrix}\,.
$$
In order to treat the pluriclosed flow later, we will in fact consider a family of flows parametrised by a gauge term. First, recall from \cite{HGRF} that we considered the distinguished  subgroup 
\[
    \mathsf{L} := \{  h \in \mathsf{SO}(\f G,\mdot) :  h(\ggo^*) \subset \ggo^* \} \, \leq \, \mathsf{SO}(\f G,\mdot),
\]
whose Lie algebra is given by $\f l := \{L \in \f{so}(\f G,\m{\cdot}{\cdot}) \,| \, L(\f g^*) \subset \f g^*\}$.

\begin{thm} \label{thm_equiv}
Consider any smooth map 
\[
    Q \colon \mc A_{\mc D} \to \f{so}(\f G,\B {\mc G}) \cap \f{so}(\f G,\mdot)
\]
satisfying $\mc Rc_{\mud}^\star - Q_{(\mud,C)} \in \f l$. Then, the \emph{$Q$-gauged augmented bracket flow}, defined by the ODE system
\begin{equation}\label{eqn_augBF}
    \begin{cases}
        \frac{\dd}{\dd t}\mud = -\Theta(\mc Rc_{\mud}^\star - Q_{(\mud,C)})\mud,\\
        \frac{\dd}{\dd t}C = C(\mc Rc_\mud^\star - Q_{(\mud,C)}),
    \end{cases}
\end{equation}
preserves $\mc A_{\mc D}$. Moreover, it is equivalent to the invariant generalized Ricci flow in the following sense: given $(\mc G_t)_{t \in I}$ a maximal invariant generalized Ricci flow solution on the ECA corresponding to $(\Ggo,\mud_0)$ (see \eqref{eqn_infinitesimal_ECA}),  with $\mc G_0 = \B {\mc G}$,
then the maximal solution $(\mud_t, C_t)$ to \eqref{eqn_augBF} with initial conditions 
\[
(\mud_0, C_0) = \big(\mud_0, \sqrt{3} \,  B_{\mud_0}^{1/2} \big)
\]
is defined precisely on the same time interval $I$, 
and there is a smooth family of infinitesimal metric Courant algebroid isomorphisms
\[
    h_t \colon (\f G, \mud_t, \B{\mc G}) \to (\f G, \mud_0, \mc {G}_t), \qquad t \in I.
\]

\end{thm}

\begin{proof}
We first show that the compatibility condition $3 B_{\mud} = C^t C$ defining $\mc A$ is preserved under \eqref{eqn_augBF}. Let $P = P_{\mud_t, C_t} := \mc Rc_{\mud_t}^\star - Q_{(\mud_t,C_t)}$. Then, by \eqref{eqn_augBF} and \cite[$\S$3.2]{homRF}, we have
\[
    \frac{\dd}{\dd t} \left( C^t C - 3 B_\mud \right) =  P^t C^t C + C^t C P - 3 P^t B_\mud - 3 B_\mud P  = 0.
\]
The two extra conditions defining $\mc A_{\mc D}$ are preserved due to the fact that the elements of  $\mathsf{L}$ preserve $\ggo^*$ and $\ip$. Thus, \eqref{eqn_augBF} preserves $\mc A_{\mc D}$.  

The rest of the proof is essentially the same as that of \cite[Thm.~5.5]{HGRF}. 
\end{proof}

\begin{rmk}\label{rmk:Agauge}
    An example of a gauge map from \cite[Thm.~5.5]{HGRF} is given by 
    \[
        Q_{(\mud,C)} = A_\mud := \frac 12 \begin{pmatrix}
        0&\B g^{-1}\dd_{\mu}^*H \B g^{-1}\\ \dd_\mu^* H&0
        \end{pmatrix} = \operatorname{Skew}_{\B{\mc G}}( \dd^*_{\mu}H),
    \]
    for all $\mud = \mu_H \in \mc A_{\mc D}$ with $\mu$ unimodular. Another choice of gauge will be discussed in Section \ref{sec_pf} for the pluriclosed flow.
\end{rmk}

\begin{lem}\label{lem_gauges}
All  gauge maps $Q \colon \mc A_{\mc D} \to \f{so}(\f G,\B {\mc G}) \cap \f{so}(\f G,\mdot)$ satisfying $\mc Rc_{\mud}^\star - Q_{(\mud,C)} \in \f l$ are given by 
\[
    Q_{(\mud, C)} = \B{q_{(\mud,C)}} + \operatorname{Skew}_{\B{\mc G}}( \dd^*_{\mu}H-\iota_{U_{\mu}}H),\qquad \mud = \mu_H,
\]
for some $q : \mc A_{\mc D} \to \f {so}(\ggo, \bar g)$.
\end{lem}
\begin{proof}
${\mc Rc}_\mud^\star -\operatorname{Skew}_{\B{\mc G}}( \dd^*_{\mu}H-\iota_{U_{\mu}}H) \in \f l$, the map $Q - \operatorname{Skew}_{\B{\mc G}}( \dd^*_{\mu}H-\iota_{U_{\mu}}H)$ must be simultaneously in $\f l$ and in $\f{so}(\f G,\B {\mc G}) \cap \f{so}(\f G,\mdot)$, hence of the form $\B{q}$ as stated.
\end{proof}

\subsection{Proof of Theorem \ref{main_solv}}

Let $\mathbb S(V) := \{ v\in V : |v|=1\}$, and  let  $\Theta_{V}$ be the infinitesimal $\msf{GL}(\f G)$-action:
\[
\Theta_V(E)(\mud, C):=(\Theta(E)\mud, -CE)\,, \quad E\in \f{gl}(\f G)\,,\quad (\mud, C)\in V\,.
\]
Before going into the discussion of the  main result that will guarantee the proof of \Cref{main_solv}, we need  to introduce the \emph{normalized $Q$-gauged augmented bracket flow}. 
 \begin{defn}
  Let $(\mud_t, C_t)_{t\in I}\subset \mc A_{\mc D} \cap \mathbb{S}(V)$ be  a  one parameter family of augmented Dorfman brackets. Then $(\mud_t, C_t)_{t\in I}$ is a solution of the \emph{normalized} $Q$-gauged augmented bracket flow if the following evolution equations are satisfied:
  \begin{equation}\label{eqn_normgBRF}
\begin{cases}
\frac{\dd}{\dd t}\mud= -\Theta(\mc Rc_{\mud}^\star-Q_{(\mud,C)})\mud+ \ell \mud\,, \\
\frac{\dd}{\dd t}C= C(\mc Rc_{\mud}^\star-Q_{(\mud,C)}) + \ell C\,, 
\end{cases}\quad \ell:=\bar g(\Theta_V(\mc Rc_{\mud}^\star)(\mud, C), (\mud, C))\,,
  \end{equation}
  where $Q\colon \mc A_{\mc D} \to \f{so}(\f G,\B{\mc G}) \cap \f{so}(\f G,\mdot)$ satisfies the same conditions as in Theorem \ref{thm_equiv}.
 \end{defn}
 \begin{rmk}\label{rmk_propngBRF}
 A standard argument (see e.g.~\cite[Lemma 7.2]{HGRF}) implies that any solution of \Cref{eqn_normgBRF} is obtained from a solution of \Cref{eqn_augBF} after  rescaling and time-reparametrization.  Moreover,  it is not hard to see that $\mc A_{\mc D}\cap \mathbb S(V)$ is invariant under \eqref{eqn_normgBRF}.
 \end{rmk}
We now consider the \emph{energy functional}
\begin{equation}\label{eqn_monqua}
    F \colon \mc A_{\mc D} \cap \mathbb{S}(V) \to [0,\infty); \qquad (\mud,C) \mapsto |M^{\f{so}}_{(\mud,C)}|^2.
\end{equation}
Note that $F$ is invariant under the action of $\msf{SO}(\Ggo,\m{\cdot}{\cdot})\cap \msf{O}(\Ggo,\B{\mc G})$ on $\mathbb{S}(V)$, since $M^{\f{so}}|_{\mathbb{S}(V)}$ is equivariant under the same group (see \cite{GIT20}). The fact that the unimodular generalized Ricci curvature is a scaled moment map now allows us to write the normalized $Q$-gauged augmented bracket flow \eqref{eqn_normgBRF} as gauge equivalent to the negative gradient flow of $F$.

\begin{lem} \label{lem_gradFlow} Let $\mu_0$ be a Lie bracket on $\f g$. Then, a curve $(\mud_t, C_t)_{t\in I}\subset \mc A_{\mc D} \cap \mathbb{S}(V)$ with initial condition $(\mud_0 = (\mu_0)_{H_0},C_0)$ solves \eqref{eqn_normgBRF} if and only if
\begin{equation}
    \frac{\dd}{\dd t} (\mud_t,C_t) = -\frac 32(\mathring{\nabla} F )(\mud_t,C_t) + \Theta_V(Q_{(\mud_t,C_t)})(\mud_t,C_t),   
\end{equation}
for all $t \in I$, where the gradient is taken with respect to the round metric on $\mathbb{S}(V)$.
\end{lem}
\begin{proof} The claim follows from the second claim in Proposition \ref{prop_mmaug}, the definition of $M^{\f{so}}_{(\mud,C)}$, and \cite[Lemma 7.2]{GIT20}.
\end{proof}

The following is   the main technical ingredient for Theorem \ref{main_solv}:

\begin{thm} \label{thm_subconbra}    Let $(\mud_t,C_t)_{t\in[0,+\infty)} \subset \mc A_{\mc D}$ be a non-zero solution of the $Q$-gauged augmented bracket flow \eqref{eqn_augBF}. Then, for every  sequence $t_k\to +\infty$, the rescaled augmented brackets $\frac{(\mud_{t_k}, C_{t_k})}{|(\mud_{t_k},C_{t_k})|}$ subconverge to a non-flat, semi-algebraic, expanding generalized Ricci soliton $ (\nud_{\infty} = (\nu_{\infty})_{H_\infty}, C_{\infty})$ such that $\dd^*_{\nu_\infty}H_{\infty} = \iota_{U_{\nu_{\infty}}}H_{\infty}$. That is, the limit generalized Ricci curvature satisfies 
\[
    \mc Rc_{\nud_\infty}^\star + \ell \id_{\f G} \in \operatorname{Der}(\nud_\infty)\cap (\R \id_{\f G} \oplus \,\f l),\qquad \dd^*_{\nu_\infty}H_{\infty} = \iota_{U_{\nu_{\infty}}}H_{\infty}
\]
for some $\ell > 0$.
\end{thm}

\begin{rmk}
    By non-flat, we mean that $\mc Rc_{\nud_\infty} \neq 0$. Note that it could be the case that $\nu_\infty = 0$. Indeed, there exist non-trivial expanding algebraic solitons on abelian Lie groups. An example is $\mathsf G = \R^3$ with the Euclidean metric and $H$ given by the corresponding volume form.
\end{rmk}

\begin{proof} 
Let $(\nud_0, D_0) := \frac{(\mud_{0}, C_{0})}{|(\mud_{0},C_{0})|} \in \mc A_{\mc D} \cap \mathbb{S}(V)$, and consider the solution $(\nud_t, D_t)$ to \eqref{eqn_normgBRF}  with initial condition $(\nud_0,D_0)$. The latter differs from $(\mud_t,C_t)$ simply by a normalization and a time reparametrization. Thus, the theorem reduces to showing that the omega limit $\Omega$ for $(\nud_t,D_t)$ consists of expanding, non-flat, semi-algebraic soliton brackets as in the statement.

\vskip5pt

\noindent \textbf{Claim 1.} Along $(\nud_t,D_t)$, the function $F$ defined in \eqref{eqn_monqua} satisfies
\[
    \frac{\dd}{\dd t}F = -\frac 32 |\mathring{\nabla} F|^2.
\]


By Lemma \ref{lem_gradFlow}, this holds provided we show that  $\mathring{\nabla} F(\nud,D) \perp \Theta_V(Q_{(\nud,D)})(\nud,D)$ for all $(\nud,D) \in \mc A_{\mc D} \cap \mathbb{S}(V)$. But this is clear, since $F$ is invariant under the action of $\msf{SO}(\Ggo,\m{\cdot}{\cdot})\cap \msf{O}(\Ggo,\B{\mc G})$,  $Q$ takes values in the Lie algebra of said group, and $\Theta_V(Q_{(\mud,C)})(\mud,C) \in T_{(\mud,C)}\mathbb{S}(V)$.

\vskip5pt

It follows that $\Omega$ consists of critical points of $F|_{\mathbb{S}(V)}$, and  $F|_\Omega \equiv F_{\infty} \in \mathbb{R}$.
Now, since $F$ is real-analytic and $\Omega$ is compact, the \L ojasiewicz inequality \cite{Loj63} implies that there exist $\epsilon > 0$, $\gamma \in (0,\frac 12]$, and $c > 0$ such that 
\begin{equation}\label{eqn:LojF}
    |F(\nud,D) - F_\infty|^{1 -\gamma} \leq c |\mathring{\nabla} F(\nud,D)|,
\end{equation}
for all $(\nud,D) \in \mc A_{\mc D}\cap \mathbb{S}(V)$ with  $\operatorname{dist}((\nud,D),\Omega) \leq \varepsilon$. Notice that, thanks to Claim 1, this condition will certainly hold along $(\nud_t,D_t)$ for all $t \geq T(\epsilon)$. This allows us to obtain a second monotone quantity:

\vskip5pt

\noindent \textbf{Claim 2.} Along $(\nud_t,D_t)_{t\geq T(\epsilon)}$, the function $\widetilde F \colon \mc A_{\mc D} \cap \mathbb{S}(V) \to [0,\infty)$ defined by
\[
    \widetilde F(\nud = \nu_H,D) := \tfrac 12 |H|^2 + \tfrac{c}{\gamma}(F(\nud,D) - F_\infty)^{\gamma}
\]
satisfies
\begin{equation}\label{eqn_2monqua}
\frac{\dd}{\dd t} \widetilde F \leq -3 |\dd^*_{\nu_t}H_t-\iota_{U_{\nu_t}}H_t|^2 \leq 0,
\end{equation}
with equality if and only if  $\dd_{\nu_t}^*H_t = \iota_{U_{\nu_t}}H_t$.

\vskip5pt 

To prove this, first notice that \eqref{eqn:LojF} implies that for all $t\geq T(\epsilon)$
\begin{equation}\label{eqn_Loj}
-\frac{\dd}{\dd t} (F(\nud_t,D_t) - F_\infty)^{\gamma} = \frac 32 \gamma (F(\nud_t,D_t) - F_\infty)^{\gamma-1} |\mathring{\nabla} F(\nud_t,D_t)|^2 \geq \frac{3\gamma}{2c}|\mathring{\nabla} F(\nud_t,D_t)|.
\end{equation}
On the other hand, Lemma \ref{lem_gauges} says that 
\begin{equation*}
    Q_{(\nud_t,D_t)} = \B{q_{(\nud_t,D_t)}} + \operatorname{Skew}_{\B{\mc G}} (\dd^*_{\nu_t}H_t-\iota_{U_{\nu_t}}H_t),
\end{equation*}
for some $q_{(\nud_t,C_t)} \in \f{so}(\f g, \B g)$.
Thus,
\begin{align*}
\frac{\dd}{\dd t} \, \tfrac 12 |H_t|^2 &= \B g \left(\frac{\dd}{\dd t}(\nu_t)_{H_t} , 0_{H_t}\right)\\
&= - \tfrac 32\B g\left(\mathring{\nabla} F(\nud_t,D_t), (0_{H_t},0)\right) +\B g(\Theta_V(Q_{(\nud_t,D_t)})\nud_t,0_{H_t})\\
&\leq \tfrac 32|\mathring{\nabla} F(\nud_t,D_t)| + \B g(\rho(q_{(\nud_t,D_t)})H_t,H_t) + \B g( \Theta_V(\operatorname{Skew}_{\B{\mc G}} (\dd^*_{\nu_t}H_t-\iota_{U_{\nu_t}}H_t))(\nu_t)_{H_t},0_{H_t})\\
&=  \tfrac 32|\mathring{\nabla} F(\nud_t,D_t)| + \B g( \Theta_V(\dd^*_{\nu_t}H_t-\iota_{U_{\nu_t}}H_t)(\nu_t)_{H_t},0_{H_t}) - \B g( \Theta_V(\dd^*_{\nu_t}H_t-\iota_{U_{\nu_t}}H_t)0_{H_t},(\nu_t)_{H_t})\\
&= \tfrac 32|\mathring{\nabla} F(\nud_t,D_t)| + \B g(0_{-\dd_{\nu_t}\dd^*_{\nu_t}H_t{+\dd_{\nu_t}\iota_{U_{\nu_t}}H_t}},0_{H_t})\\
&=
\tfrac 32|\mathring{\nabla} F(\nud_t,D_t)|  -\B g(\dd_{\nu_t}\dd^*_{\nu_t}H_t - \dd_{\nu_t}\iota_{U_{\nu_t}}H_t,H_t)
\\
&=
\tfrac 32|\mathring{\nabla} F(\nud_t,D_t)| - 3 \B g(\dd^*_{\nu_t}H_t - \dd_{\nu_t}\iota_{U_{\nu_t}}H_t , \dd_{\nu_t}^*H_t) + 3 \B g(\dd^*_{\nu_t}H_t - \dd_{\nu_t}\iota_{U_{\nu_t}}H_t, \iota_{U_{\nu_t}}H_t) 
\\
&= \tfrac 32|\mathring{\nabla} F(\nud_t,D_t)| - 3 |\dd^*_{\nu_t}H_t-\iota_{U_{\nu_t}}H_t|^2.
\end{align*}
Together with \eqref{eqn_Loj}, the claim follows.

\vskip5pt

From Claim 2,  $\widetilde F$ must be constant along any solution to \eqref{eqn_normgBRF} contained in $\Omega$,  hence $\mathring{\nabla} F (\nud, D) = 0$ and $\dd_{\nu}^*H = \iota_{U_{\nu}}H$, for all $(\nud,D)\in \Omega$. By \cite[Lemma 7.2]{GIT20},  the former gives 
\[
      \Theta(\mc Rc_{\nud}^\star) \nud = \ell  \, \nud,\qquad  (\nud,D)\in \Omega\,,
\]
so that $\mc Rc_{\nud}^\star + \ell \, \id_{\Ggo} \in \operatorname{Der}(\nud)$.  On the other hand, the condition $\dd_{\nu}^*H = \iota_{U_{\nu}}H$ implies $\mc Rc_{\nud}^\star\in \f l$.

The limit is non-flat by \eqref{eqn_scalgen}. 
Regarding the sign of $\ell$, from the identity $\mc Rc_{\nud_\infty}^\star = \frac 23 M^{\f{so}}_{\nud_\infty}$ (see Proposition \ref{prop_mmaug}), and  the moment map critical point equation $\Theta_V(\mc Rc_{\nud_\infty}^\star)(\nud_\infty,C_\infty) = \ell (\nud_\infty,C_\infty)$, we get 
\begin{align*}
   \ell =   \ell \,  \big| (\nud_\infty, C_\infty) \big|^2  &= \bar g \left( \Theta_V(\mc Rc_{\nud_\infty}^\star)(\nud_\infty,C_\infty) , (\nud_\infty,C_\infty)  \right) 
    =  4 \, \bar g \left( M^{\f{so}}_{\nud_\infty} , \mc Rc_{\nud_\infty}^\star \right) 
     = 6 \, \big|\mc Rc_{\nud_\infty}^\star \big|^2 > 0.\qedhere
\end{align*}

\end{proof}

  By a similar argument as in \cite[Theorem 6.4]{HGRF}, we can prove the following lemma.
  \begin{lem}
  Let $(\f G, \ip, \mud, \mc G_t)$ be a non-flat left-invariant solution of the generalized Ricci flow over a Lie group $\msf G$ with positive semi-definite Killing form. Then $\mc S(\mc G_t)\sim -\frac1t$, as $t\to \infty$. 
  \end{lem}
\begin{proof}


    By ODE comparison and the fact that $\mc S_{\mud} = -\frac 1{12} |(\mud,C)|^2-|U_{\mu}|^2$, it suffices to show that there is a large dimensional constant $K >0$ such that  
    $$
    -K\left(\frac{1}{12}|(\mud_t,C_t)|^2+|U_{{\mu}_t}|^2\right)^2 \leq \frac{\dd}{\dd t}\left(\frac{1}{12}|(\mud_t,C_t)|^2+|U_{{\mu}_t}|^2\right) \leq -\frac{1}K\left(\frac{1}{12}|(\mud_t,C_t)|^2+|U_{{\mu}_t}|^2\right)^2
    $$ where $(\mud_t,C_t)$ solves the  $Q$-gauged augmented bracket flow \eqref{eqn_augBF} (note that $\mc S_{\mud} \neq 0$ since the solution is not flat). First of all, using \Cref{prop_mmaug}, we have
    \[
    \frac{1}{12}\frac{\dd}{\dd t} |(\mud_t,C_t)|^2 = - \frac16 \B g(\Theta_V(\mc Rc_{\mud_t}^\star)(\mud_t,C_t),(\mud_t,C_t))=-|\mc Rc_{\mud}^\star|^2\,.
    \]
  We now compute the evolution of $|U|^2$. For the  $Q$-gauged augmented bracket flow, we have that 
   $$
   \frac{\dd}{\dd t}\mu=-\theta({\rm Ric}_{\mu}^\star - \frac 14 H^2-q_{\mu})\mu.
   $$ Therefore, for any $X\in \f g$, it holds that 
   $$
   \B g\left(\frac{\dd}{\dd t}U, X\right)={\rm tr}\left(\frac{\dd}{\dd t}\ad_X\right)={\rm tr}(\ad_{({\rm Ric}_{\mu}^\star - \tfrac 14 H^2 -q_{\mu})X})=\B  g\left(({\rm Ric}_{\mu}^\star - \tfrac 14 H^2 -q_{\mu})X, U\right).
   $$
   Hence,
   $$
   \frac{\dd}{\dd t}|U|^2=2\B  g(({\rm Ric}_{\mu}^\star -\tfrac 14 H^2)U, U)=2{\rm Ric}_{\mu}(U, U)-\frac12H^2(U, U)=-2{\rm tr}(\Sym(\ad_U)^2)-\frac12|\iota_UH|^2=-|\Sym(\ad_U^{\mud})|^2\,,
   $$ where we made use of the fact that  $q_{\mu}\in \f{so}(\f g, \B g)$, $\B g(\Sym(\ad_U)U, U)=0$ and \cite[p.18]{homRF}. Using \Cref{prop_mmaug}, we observe that ${\rm tr}(\mc Rc_{\mud}^\star\Sym(\ad_U^{\mud}))= 0$. Thus, we have 
   \begin{equation}\label{eqn_evolgenscal}
   \frac{\dd}{\dd t}\left(\frac{1}{12}|(\mud, C)|^2+|U|^2\right)=-|\mc Rc_{\mud}|^2\,.
   \end{equation}On the one hand, since $\mc Rc_{\mud}$ is quadratic in $\mud$, we have that 
   $$
\frac{\dd}{\dd t}\left(\frac{1}{12}|(\mud, C)|^2+|U|^2\right)=-|\mc Rc_{\mud}|^2\ge -K'|\mud|^4\ge- K\left(\frac{1}{12}|(\mud, C)|^2+|U|^2\right)^2\,.
   $$
   For the second, we first suppose that $\mc Rc_{\mud} = 0$ for some $(\mud,C) = (\mu_H,C) \in \mc A_{\mc D} \cap\{\mc S_{\mud}=-1\}$. This implies $\operatorname{Ric}_\mu = \frac 14 H^2$, thus $$0 \leq \frac 14 |H|^2 = \operatorname{tr}(\operatorname{Ric}_\mu) = \operatorname{tr}\left(M_{\mu} - \frac 12 B_\mu-\Sym(\ad_U)\right) = -\frac 14 |\mu|^2 - \frac 12 \tr B_\mu-|U|^2 
    \leq 0,$$
    since $B_\mu \geq 0$, so $(\mud,C) = 0$ and $U=0$. This implies that the function $|\mc Rc|^2\colon \mc A_{\mc D}\cap\{\mc S_{\mud}=-1\} \to \R$ from  \Cref{eqn_monqua} is strictly positive, hence, using that $\mc A_{\mc D}\cap\{\mc S_{\mud}=-1\}$ is compact,  $|\mc Rc_{\mud}|^2 \geq \frac 1K (\frac{1}{12}|(\mud,C)|^2+|U|^2)^2$. The claim now follows from the evolution \eqref{eqn_evolgenscal}.
\end{proof}
  
  We are now ready to prove \Cref{main_solv}.
  \begin{proof}[Proof of Theorem \ref{main_solv}]
 Using \Cref{thm_equiv}, we know that any  generalized Ricci flow  solution $(g_t, H_t)_{t\in [0, \infty)}$ is equivalent to a solution of $(\mud_t, C_t)_{t\in[0, \infty)}$ of \Cref{eqn_augBF}.   Now,  up to a time reparametrization, the one parameter family $(\bar {\mud}_t, \bar C_t):=\left(\frac{\mud_t}{|(\mud_t, C_t)|},\frac{C_t}{|(\mud_t, C_t)|}\right)$ solves the  normalized $Q$-gauged augmented bracket flow. Thus, applying \Cref{thm_subconbra}, we have that for every sequence $t_k\to \infty$, $(\bar {\mud}_{t_k},\bar C_{t_k})$ subconverges to a non-flat, expanding, semi-algebraic generalized soliton $(\mud_{\infty}, C_{\infty})$, as $k\to \infty$. To conclude,  we just need to observe that an argument as in \cite[Theorem 8.4]{HGRF} yields that $(\f G, \mud_{t_k}, \mc G_{\mud_{t_k}},  \msf{ G}_{\mu_{t_k}})$ subconverges locally to $(\f G, \mud_{\infty}, \mc G_{\mud_{\infty}}, \msf G_{\mu_{\infty}})$ (cf.~\cite[Thm.~6.14]{Lauret2012}). On the other hand, this convergence  and \cite[Corollary 5.28]{GRFbook} allows to conclude that $(\f G, \mud_{t_k}, \mc G_{\mud_{t_k}},  \msf{ G}_{\mu_{t_k}}, e)$  subconverges in the generalized Cheeger-Gromov sense to an ECA $(E, \mc G, M)$ which is locally isometric to $(\f G, \mud_{\infty}, \mc G_{\mud_{\infty}},  \msf{ G}_{\mu_{\infty}})$.  In particular, the latter subconvergence implies that $(\msf{G}_{\mu_{t_k}}, g_{\mu_{t_k}})$ subconverges to $(M, g)$ in the classical Cheeger-Gromov sense where $g$ is the Riemannian metric on $M$ determined by $\mc G$.   Now, using  \cite[Theorem D.2]{BL17}, we can conclude $(M, g)$ is homogeneous and simply connected. This allows to deduce that $(M, g)$ is isometric to $(\msf G_{\mu_{\infty}}, g_{\mu_{\infty}})$. This last fact is sufficient to infer that  $(E, \mc G, M)$ is isometric to $(\f G, \mud_{\infty}, \mc G_{\mud_{\infty}},  \msf{ G}_{\mu_{\infty}})$, giving  the claim. 
  \end{proof}

We conclude this section with a positive answer to \cite[Question 7.12]{HGRF}, in the wider setting of Lie groups with positive semi-definite Killing form.

\begin{cor}
    Let $(\f G, \ip, \mu_H, \B{\mc G})$ an algebraic generalized Ricci soliton, i.e. $\Theta(\mc Rc_{\mu_H} -A_{\mu_H})\mu_H=\lambda\mu_H$, on a unimodular  Lie group $\msf G$ with $\kappa_{\mu}\ge 0 $. Then, $\dd^*_{\mu}H=0$, that is, $H$ is harmonic. 
\end{cor}

\begin{proof}
Let $D :=\mc Rc_{\mu_H} -A_{\mu_H} + \lambda \id_{\f G} \in {\rm Der}(\mu_H)$. We claim that $C_0 D = 0$, where $C_0 := \sqrt{3}B_{\mu_H}^{1/2}$; this will be proved below. It follows that the augmented bracket $(\mu_H, C_0)\in \mc A_{\mc D}$  associated to our algebraic soliton  is precisely a fixed point of the normalized $Q$-gauged augmented bracket flow \eqref{eqn_normgBRF} with $Q = A$ (see Remark \ref{rmk:Agauge}), and the corollary follows from Theorem \ref{thm_subconbra} specialized in the unimodular case.

Regarding the above claim, 
using  \eqref{eqn:KillingmuH_mu} and the fact that $D(\ggo^*) \subset \ggo^*$ we obtain 
\[
    C_0 D = \sqrt{3} \begin{pmatrix}
        B_\mu^{1/2}  D_\ggo & 0\\ 0&0
        \end{pmatrix},
\]
where $D_\ggo \in {\rm Der}(\mu)$ is the derivation induced by $D$ on the quotient $(\ggo,\mu)$ by the abelian ideal $\ggo^*$. Now $\ggo$ is solvable, thus $ D_\ggo \ggo \subset \mu(\ggo,\ggo)$ is contained in the nilradical of $\ggo$. The Killing form vanishes on the nilradical, thus $B_\mu^{1/2} D_\ggo = 0$, and the claim follows.
\end{proof}

\section{Convergence of the pluriclosed flow on completely solvable Lie groups}
\label{sec_pf}
 In this section, we show subconvergence of the pluriclosed flow  to pluriclosed solitons on SKT  Lie groups with positive semi-definite Killing form. To that end, we first recall some basic notions  in complex geometry and  the connection between generalized Ricci flow and pluriclosed flow.  

\begin{defn}
Let $(M, J, g)$ be a Hermitian manifold. The metric $g$ is called SKT if $\dd\dd^c\omega=0$, where $\omega(\cdot, \cdot):=g(J\cdot, \cdot)$ is the fundamental form associated to $g$ and $\dd^c:=J\dd J^{-1}$.
\end{defn}
Among Hermitian metrics, SKT metrics have a prominent role  in theoretical physics, see for instance \cite{GHR84, SSTVP88}, and in mathematics, mostly for their conjectured role in the classification of compact complex surfaces, see \cite{S20conj}. As detailed in \cite{S20conj}, the principal tool to study the conjecture is the \emph{pluriclosed flow}, a geometric flow of SKT metrics.
\begin{defn}\cite{ST10}
Let $(M, J, g_0)$ be a SKT manifold. A one-parameter family of Hermitian metrics $(g_t)_{t\in I}$ is a solution of the pluriclosed flow if the following evolution equation of $(1,1)$-forms is satisfied:
$$
\frac{\partial }{\partial t}\omega=-(\rho^{B}(\omega))^{1,1}\,, \quad \omega_t(\cdot, \cdot )=g_t(J\cdot, \cdot)\,,
$$ where $\rho^B(\omega)$ is the \emph{Bismut-Ricci form} of $\omega$ defined by:
$$
\rho^B(\omega):=-\frac12{\rm tr}(JR^B)\,,
$$ with $R^B$ being the curvature tensor  of the Bismut connection, see \cite{Bis89}. 
\end{defn}

 In our treatment, the main result we will make use of is the following, stating that  the pluriclosed flow and the generalized Ricci flow are gauge-equivalent. 
 \begin{thm}[\cite{Str13}, Proposition 6.3, 6.4] \label{thm_gaugeGRFvPCF}
 Let $(M, J)$ be a complex manifold  and $(g_t)_{t\in I}$ be a solution of the pluriclosed flow. Then, the following hold:
 \begin{equation}\label{eqn_gaugedGRF}
 \begin{aligned}
\frac{\partial }{\partial t}g=&\, -{\rm Ric}_{g, H}^B-\frac12\mc L_{g^{-1}\eta}g\,,\\
\frac{\partial }{\partial t }H=&\, \frac12\Delta_{g}H-\frac12\mc L_{g^{-1}\eta}H\,,
\end{aligned}
 \end{equation} where $\eta_t := J\dd^{*}_{g_t}\omega$ is \emph{Lee form} of $g_t$ and $H_t:=-\dd^c\omega_t$, is the torsion of the Bismut connection of $g_t$. Therefore, $
(g_t,H_t)_{t\in I}$ is a solution of a gauge-fixed generalized Ricci flow, up to time reparametrization. 
\end{thm}
 Conversely, on a fixed complex manifold $(M,  J)$, the solution of the generalized Ricci flow  starting from a SKT metric  and the torsion of its Bismut connection coupled with the following equation for the complex structure 
 $$
\frac{\partial}{\partial t }J=\Delta J+ [J, g^{-1}{\rm Ric}_g]+ Q(DJ)\,,
 $$ remains  SKT  and it can be gauged to solve the pluriclosed flow, see \cite{ST12}. 

  In the Lie group case, the homogeneous pluriclosed flow was proved to be equivalent to the pluriclosed bracket flow, see  \cite{EFV15, AL19}.
  
  \begin{thm} \label{thm_PBF}
      Let $(\msf G, \mu \in \Lambda^2 \f g^* \otimes \f g, J, \omega = g(J\cdot,\cdot))$ be a simply-connected Lie group with a left-invariant SKT structure. Then, after a constant reparametrization in time, the invariant pluriclosed flow $(g_t)_{t\in I}$ with $g_0 = g$ is equivalent to the pluriclosed bracket flow:
      \begin{equation} \label{eqn_PBF}
            \frac{\dd}{\dd t}\mu_t = -\theta(P_{\mu_t})\mu_t;\qquad \mu_0 = \mu,
      \end{equation}
    where $P_{\mu_t}\in\f{gl}(\f g, J) \cong \f{gl}_n(\C)$ is defined by $\omega(P_{\mu_t}\cdot, \cdot)=(\rho_{\mu_t}^{B}(\omega))^{1,1}(\cdot, \cdot)$. That is, the maximal existence times of both flows coincide and there is a family of biholomorphic isometries
    \[
    \varphi_t\colon (\msf G, J, g_t) \to (\msf G_{\mu_t},J_{\mu_t},g_{\mu_t}),
    \]
    where $(\msf G_{\mu_t},J_{\mu_t},g_{\mu_t})$ is the simply-connected Lie group with left-invariant pluriclosed structure corresponding to the data $(\f g,\mu_t,J,g)$.
  \end{thm}

  \begin{proof} Let $\frac{\dd}{\dd t}\omega_t = -2 (\rho^{B} (\omega_t))^{1,1}$ be a reparametrized pluriclosed flow. We can then observe that $g_t = h_t^{-1}\cdot g$, and $\mu_t = h_t \cdot \mu$, where $h_t \in \msf{GL}(\f g ,J) = \{h \in \msf{GL}(\f g)\,  |\,  [h,J] = 0 \}$ solves $\frac{\dd }{ \dd t}h_t = -h_t P_{\mu,g_t} = - P_{\mu_t,g}h_t$ and $P_{\nu,k} := \omega_k^{-1}(\rho^B_{\nu}(\omega_k))^{1,1} \in \f{gl}(\f g, J)$ and $\omega_k = k(J\cdot,\cdot)$. Then $h_t$ can be integrated to a biholomorphism $\varphi_t$ as claimed.
  \end{proof}

  We will be particularly interested in self-similar solutions of the pluriclosed flow:
\begin{defn}
Let $(M, J,g )$ be SKT manifold. The metric $g$ is a \emph{pluriclosed soliton} if there exist $\lambda\in \R $ and a holomorphic vector field $V\in \Gamma(TM)$ (i.e. $\mc L_VJ=0$)  such that 
$$
(\rho^B(\omega))^{1,1}=\lambda \omega+\frac12\mc L_V\omega\,.
$$
\end{defn}

 Before discussing the proof of \Cref{main_PCF},  we  prove the following. 
\begin{prop}\label{prop_Stanfieldthm}
Let $(\msf G , J)$ be a Lie group endowed with a left-invariant  complex structure and let $g$ be a left-invariant SKT metric which is a semi-algebraic generalized Ricci soliton as in \Cref{thm_subconbra}. Then, $g$ is a pluriclosed soliton.
\end{prop}
\begin{proof}
 By assumption, we have that $H=-\dd^c\omega$ and that there exists $Y\in \Gamma(T\msf G)$ and $\lambda \in \R$ such that  the following equations are satisfied: 
\begin{equation}\label{eqn_SKTsoliton}
{\rm Ric}^B_{g, H} +\frac12\mathcal L_{Y-U}g = \lambda g\,,\quad
 -\dd \dd^*_g H = \mathcal L_{Y-U}H+ 2 \,\lambda H\,,\quad 
 \dd^*_gH=\iota_UH\,.
\end{equation}  

Moreover, the flow of $Y $ is given by $\{\varphi_t\}\subset {\rm Aut}(\msf G)$ such that $(\dd\varphi_t)_e=e^{tD}$, where $D\in {\rm Der}(\f g)\cap \Sym(\f g, g ).$
Using \cite[Proposition 6.3]{Str13}, we have that
\begin{equation}\label{eqn_11part}
(\rho^B(\omega))^{1,1}(\cdot, J \cdot)={\rm Ric}^B_{g, H}+\frac12\mc L_{g^{-1}\eta}g\,.
\end{equation}
  In order to obtain the claim, according to \cite[Theorem 9.6]{HGRF}, we need to show that $ V=Y-g^{-1}\eta -U$ satisfies $\mathcal L_{V}J=0$. Using \Cref{eqn_SKTsoliton,eqn_11part}, we have that 
$$
(\rho^B(\omega))^{1,1}(\cdot, J \cdot)+\frac12\mc L_{V}g=\lambda g\,.
$$ Hence, taking  the $(2,0)$-part, we have that 
$
(\mc L_Vg)^{2,0}=0\,.
$ 
On the other hand, for any $X, W\in\Gamma(T\msf G)$, we have  
$$
\begin{aligned}
0=2(\mc L_Vg)^{(2,0)+(0,2)}(X, W)=&\, (\mc L_Vg)(X, W)-(\mc L_Vg)(JX, JW)\\
=&\, g(J[V, X], JW)+g(JX, J[V, W])- g([V, JX], JW)-g(JX,[V,  JW])\\
=&\, -g((\mc L_VJ)X, JW)- g(JX, (\mc L_VJ)W)=g(J(\mc L_VJ)X, W)-g(JX, (\mc L_VJ)W)\\
=&\, -g((\mc L_VJ)(JX), W)-g(JX, (\mc L_VJ)W)\,,
\end{aligned}
$$ giving  that  $\mathcal L_VJ$ is skew-symmetric.  Using  \cite[Proof of Proposition 4.1]{SU22}, we have that 
$$
(\iota_UH)^{2,0}=(\dd_g^*H)^{2,0}=-\sqrt{-1}(\mc L_{g^{-1}\eta}\omega)^{2,0}.
$$ 
We  note that the difference in sign between  the above equation and \cite[Proof of Proposition 4.1]{SU22}  comes from the different sign in the definition of $H$.  
Now, we have that 
$$
\iota_UH=-\iota_UJ\dd\omega=J\iota_{JU}\dd\omega=J(\mathcal L_{JU}\omega-\dd\iota_{JU}\omega)=J\mathcal L_{JU}\omega+J\dd (gU)=J\mathcal L_{JU}\omega\,,
$$ where we used that  $g(U,[\f g, \f g])=0$ and hence $\dd(gU)=0$. Thus $(\iota_UH)^{(2,0)+(0,2)}=-(\mathcal L_{JU}\omega)^{(2,0)+(0,2)}$.
From this last equality, we can infer, using integrability of the complex structure, that $(\iota_UH)^{2,0}=\sqrt{-1}(\mc L_{U}\omega)^{2,0}$. Therefore $(\mc L_{g^{-1}\eta+U}\omega)^{(2,0)+(0,2)}=0$. 
Hence, for any $X, W\in \Gamma(T\msf G)$, we have that 
$$
\begin{aligned}
0=&\, (\mc L_{g^{-1}\eta+ U}\omega)(X, W)-(\mc L_{g^{-1}\eta+ U}\omega)(JX, JW)\\
=&\, -g(J\mc L_{g^{-1}\eta+ U}X, W)+g(X, J\mc L_{g^{-1}\eta+ U}W)+ g(\mc L_{g^{-1}\eta+ U}JX, W)- g(X, \mc L_{g^{-1}\eta+ U}JW)\\
=&\, g((\mc L_{g^{-1}\eta+ U}J )X, W)- g(X,(\mc L_{g^{-1}\eta+ U}J )W )\,,
\end{aligned}
$$ giving that $(\mc L_{g^{-1}\eta+ U}J)^{t}=\mc L_{g^{-1}\eta+ U}J$. From the definition of $Y$, $(\mc L_Y\omega)_e=\rho(D)\omega$. Thus,
$$
\begin{aligned}
2(\mc L_Y\omega)^{(2,0)+(0,2)}=&\, \rho(D)\omega-J\cdot\rho(D)\omega=\rho(D)\omega-\rho(JDJ^{-1})J\cdot\omega=\rho(D+JDJ)\omega\\
=&\,-g([J, D]\cdot, \cdot)+ g(\cdot, [J, D]\cdot)=0\,, 
\end{aligned}
$$ since $D\in \Sym(\f g, g)$.
Using the same computation as above, we obtain that $(\mc L_YJ)^t=\mc L_YJ$ and, then, $(\mc L_VJ)^t=\mc L_VJ$, which forces $\mc L_VJ=0$ as desired. 
\end{proof}

  We will now apply the  results from Section \ref{sec_ConvCS} to the pluriclosed flow. We will consistently borrow notation from that section.

 Fixing a pluriclosed structure $(\f g, J ,\B \omega = \B g(J \cdot,\cdot))$, we denote the set of \emph{augmented SKT Dorfman brackets} on $\f G = \f g\oplus \f g^*$ as
 $$
\mc A_{\mc D}^{J}:=\{(\mud = \mu_H, C)\in \mc A_{\mc D}\,\, |\, \, \,J \text{ is integrable with respect to } \mu, \quad H = -\dd_{\mu}^c\B \omega\} \subset \mc A_{\mc D}.\,
$$ 
Given a solution $(\mu_t)_{t \in I}$  to the pluriclosed bracket flow \eqref{eqn_PBF}, we may define a curve of Dorfman brackets $\mud_t := (\mu_t)_{H_{\mu_t}}$, where $H_{\mu_t} := -\dd^c_{\mu_t}\B \omega$. If $\mu_0$ has positive semi-definite Killing form, this lifts (non-uniquely) to a smooth curve in $\mc A_{\mc D}^J$.

We now use this fact to prove that the pluriclosed bracket flow is equivalent to a $Q$-gauged augmented bracket flow, which is simply a manifestation of the gauge equivalence of pluriclosed flow and generalized Ricci flow.

\begin{lem}\label{lem_PCFQGauged} If $\mu_0$ has positive semi-definite Killing form, then the solution $(\mu_t)_{t \in I}$ of the  pluriclosed bracket flow lifts to a curve $(\mud_t = (\mu_t)_{H_{\mu_t}},C_t)_{t \in I} \subset \mc A_{\mc D}^J$ solving the $Q$-gauged augmented bracket flow \eqref{eqn_augBF}, with $Q \colon \mc A_{\mc D} \to \f{so}(\f G,\B g) \cap \f{so}(\f G,\mdot)$ given by
      \[   
         Q_{\mu_H} := \operatorname{Skew}_{\B g}(\dd_{\mu}^*H-\iota_{U_{\mu}}H) - \B{\operatorname{Skew}_{\B g}(\ad^\mu_{\B g^{-1}\eta_\mu+ U_{\mu}})} = \begin{pmatrix}
            -\operatorname{Skew}_{\B g} (\ad^\mu_{\B g^{-1}\eta_\mu+ U_{\mu}}) & \frac 12 \B g^{-1}(\dd^*_\mu H-\iota_{U_{\mu}}H) \B g^{-1}\\\frac 12( \dd^*_\mu H -\iota_{U_{\mu}}H)& \operatorname{Skew}_{\B  g}(\ad^\mu_{\B g^{-1}\eta_\mu+ U_{\mu}})^*
        \end{pmatrix},
      \]
      where $\eta_{\mu} = J\dd^*_{\mu}\B\omega$ is the Lee form of the  Hermitian structure $(\f g, \mu,J,\B \omega = \B g(J\cdot,\cdot))$.
\end{lem}

\begin{proof} First we may deduce from the metric evolution in \eqref{eqn_gaugedGRF} that $P_\mu = \operatorname{Ric}^{B}_{\mu} - \operatorname{Sym}_{\B g}(\operatorname{ad}^\mu_{\B g^{-1}\eta_{\mu}})$. Hence, the evolution of $\mu = \mu_t$ is given by $$
\begin{aligned}
\frac{\dd }{\dd t}\mu =& -\theta(P_{\mu})\mu = -\theta(\operatorname{Ric}^{B}_{\mu} - \operatorname{Sym}_{\B g}(\ad^{\mu}_{\B g^{-1}\eta_{\mu}}))\mu = -\theta(\operatorname{Ric}^{B}_{\mu} + \operatorname{Skew}_{\B g}(\ad^{\mu}_{\B g^{-1}\eta_{\mu}}))\mu\\
=&-\theta(\operatorname{Ric}^{B}_{\mu} +\Sym_{\bar g}(\ad_{U_{\mu}})+ \operatorname{Skew}_{\B g}(\ad^{\mu}_{\B g^{-1}\eta_{\mu}+U_{\mu}}))\mu,
\end{aligned}$$
since $\ad^\mu_{\B g^{-1}\eta_{\mu}}$ is a derivation of $\mu$. On the other hand, the equation $H_{\mu} = -\dd_{\mu}^c\B\omega$ again combined with the evolution equation in \eqref{eqn_gaugedGRF} yields $2\dd_{\mu}^c(\rho^B)^{1,1}(\B\omega) = -\dd_{\mu} \dd_{\mu}^*H_{\mu} - \mc L_{\B g^{-1}\eta_{\mu}} H_{\mu}$ for a general pluriclosed structure. Hence we find the evolution of $H_{\mu} = -\dd^c_{\mu}\B \omega$ to be
\begin{align*}
\frac{\dd}{\dd t}H_{\mu} &= \dd^c_{\theta(P_{\mu})\mu} \B \omega = J[\rho(P_\mu),\dd_\mu] J^{-1}\B \omega = [\rho(P_\mu),\dd^c_\mu]\B \omega = -\rho(P_\mu)H_{\mu} - \dd^c_\mu \rho(P_\mu)\B \omega \\
&= -\rho(P_\mu)H_{\mu} + 2 \dd^c_{\mu}(\rho^B)^{1,1}(\B \omega) = -\rho(P_\mu)H_{\mu} - \dd_\mu \dd^*_\mu H_\mu - \mc L_{\B g^{-1}\eta_\mu} H_\mu \\
&= -\rho(P_\mu + \ad^\mu_{\B g^{-1}\eta_\mu})H_\mu +\Delta_{\mu} H_\mu = -\rho(\operatorname{Ric}^{B}_\mu + \operatorname{Skew}_{\B g}(\ad^\mu_{\B g^{-1}\eta_\mu}))H_\mu + \Delta_\mu  H_\mu\\
&=-\rho(\operatorname{Ric}^{B}_\mu+\Sym_{\B g}(\ad^{\mu}_{U_{\mu}}) + \operatorname{Skew}_{\B g}(\ad^\mu_{\B g^{-1}\eta_\mu+U_{\mu}}))H_\mu + \Delta_\mu  H_\mu+\dd\iota_{U_{\mu}}H\,.
\end{align*}
Here we used that $\dd_{\theta(A)\mu} = [\rho(A),\dd_\mu]$, for all $A \in \operatorname{End}(\f g)$ and that $\ad^{\mu_H}_{U}\in {\rm Der}(\mu_H)$, thus, $\rho(\ad^{\mu}_{U_{\mu}})H=\dd\iota_{U_{\mu}}H$. This implies that $\mud = \mu_{H_{\mu}}$ evolves by $$
 \frac{\dd}{\dd t}\mud = -\Theta(\mc Rc_{\mud}^{\star} - Q_{\mud})\mud\,,
$$ by \cite[Proposition 5.2]{HGRF}. Finally, since $\mu_0$ has positive semi-definite Killing form, so does $\mud_0$, so there is $C_0 \in \operatorname{Sym}(\f G, \B g)$ such that $(\mud_0,C_0) \in \mc A_D^J$. We conclude the proof by letting $(C_t)_{t\in I}$ be the unique curve solving $$\frac{\dd}{\dd t}C =  C(\mc Rc_{\mud}^{\star} - Q_{\mud})\,,
$$ with initial condition $C_0$.
\end{proof}

We are now in the position to prove Theorem \ref{main_PCF}.

\begin{proof}[Proof of Theorem \ref{main_PCF}] Let $(g_t)_{t\in[0,\infty)}$ be an invariant pluriclosed flow on a simply-connected Lie group $\msf G$ with positive semi-definite Killing form and left-invariant complex structure $J$. By Theorem \ref{thm_PBF}, it is equivalent up to biholomorphisms to the pluriclosed bracket flow \eqref{eqn_PBF}. Now, by Lemma \ref{lem_PCFQGauged}, the pluriclosed bracket flow lifts to a $Q$-gauged augmented bracket flow, $(\mud_t,C_t)_{t\in [0,\infty)}$. Using Theorem \ref{thm_subconbra}, the same argument as in the proof of Theorem \ref{thm_main_convGRF} yields that, for any  sequence $t_k \to \infty$,   the spaces $(\msf G,J,\frac{1}{1 + t_k}g_{t_k})$ subconverge in the complex Cheeger--Gromov topology to a simply-connected pluriclosed Lie group $(\msf G_\infty,J_\infty,g_\infty)$ such that $(g_\infty,H_\infty = -\dd^c \omega_\infty)$ is a non-flat, expanding, semi-algebraic generalized Ricci soliton with $\dd^*_{g_{\infty}}H_{\infty}=\iota_{U}H_{\infty}$. Finally, Proposition \ref{prop_Stanfieldthm} implies that $(\msf G_\infty,J_\infty,g_\infty)$ is in fact an expanding pluriclosed soliton.
\end{proof}

\section{Stability of bi-invariant Bismut-flat metrics}\label{sec_stab}

The main aim of this section is to prove Theorem \ref{main_stabBF}. To that end, after introducing the relevant gauge group for this problem, in $\S$\ref{sec:2ndvar} we  compute the second variation of $\mc{S}$ with respect to left-invariant directions, for a general unimodular Lie group. We refine these formulas in $\S$\ref{sec:linearstab} to establish linear stability (modulo gauge) of standard bi-invariant Bismut-flat metrics on compact, simply connected, semisimple Lie groups. A quick application of centre manifold theory in $\S$\ref{sec:dynstab} then yields dynamical stability with respect to left-invariant variations.

Throughout this section, 
 we consider a left-invariant ECA $E\to \G$  endowed with a left-invariant generalized metric $\mc G$. The corresponding isotropic splitting yields the infinitesimal data $(\f G = \ggo \oplus \ggo^*, \ip, \mud = \mu_H, \mc G(g, 0))$ for some $g \in {\rm Sym}^2_+(\f g^*)$, see $\S$\ref{sec:homog} for details.  As described in Section \ref{section_formulaRc},   the space $\mc M_\gen^{\f G}$ of left-invariant generalized metrics on $(\f G , \ip)$ is a symmetric space of non-positive curvature. Thus, geodesics emanating from $\mc G\in \mc M_\gen^{\f G}$ are of the form $\gamma(t)=\exp(tZ)\cdot \mc G$, for some $Z\in\f{so}(\f G, \ip )\cap{\rm Sym}(\f G, \mc G )$ (see e.g. \cite[Chapter XI, Theorem 3.2]{KobNom96ii}).

\subsection{Gauge group}

The following \emph{gauge group} will play a key role in the stability analysis. Note that, since the action of $\msf G$ on itself by left multiplication lifts to an action on $E$, we may consider $\msf G$ as a subgroup of $\operatorname{Aut}(E)$. We then consider its normaliser:
\[
    \Aut(E)^\G := \{ \varphi \in \mathsf{N}_{\operatorname{Aut}(E)}(\msf G) \, :\,  \varphi(E|_e) \subset E|_e\}.
\]

By evaluating at the zero section of $E$, one can see that any $\varphi \in \operatorname{Aut}(E)^\G$ covers an automorphism of $\G$. Moreover, $\varphi$ maps invariant sections to invariant sections, and thus induces a map 
\[
    F : \f G \to \f G, \qquad \f G = \ggo \oplus \ggo^*.
\]
Given that $\varphi$ preserves the Dorfman bracket, the neutral inner  $\ip$, and the projection $\pi$, it is easy to see that $F$ is an element of
\[
    {\sf Aut}(E)^\ggo :=  \{ F \in  {\sf Aut}(\f G) \cap {\sf SO}(\f G, \ip) \, :\,  F (\ggo^*) \subset \ggo^* \}\,,
\]
where ${\sf Aut}(\f G) = {\sf Aut}(\f G,\mu_H)$ is the automorphism group of the Lie algebra $(\f G, \mu_H)$ (see \eqref{eqn:H_twist_bracket}).
We notice that, in the case of a semisimple Lie group, the condition $F(\ggo^*) \subset \ggo^*$ is automatically satisfied, as $\ggo^* \subset \f G$ is the radical, and it is thus preserved by any automorphism.

Conversely, if $\G$ is simply-connected, any $F\in {\sf Aut}(E)^{\f g}$ gives rise to a $\varphi \in \Aut(E)^\G$. We have thus proved:

\begin{lem}\label{lem_autEG}
Evaluation on invariant sections induces an injective homomorphism $\Aut(E)^\G \hookrightarrow \Aut(E)^\ggo$ which is bijective if $\msf G$ is simply-connected.
\end{lem}

The group  $\Aut(E)^\ggo$ acts naturally on the space $\mc M_\gen^{\f G}$ of left-invariant generalized metrics via pull-back. Due to the naturality of the generalized scalar curvature, we clearly have:
\begin{lem}\label{lem:SAutinv}
 The generalized scalar curvature $\mc S :\mc M_\gen^{\f G} \to \R$ is $\Aut(E)^\ggo$-invariant.   
\end{lem}

In particular, it follows from Lemma \ref{lem:SAutinv} that critical points of $\mc S$ come in families ($\Aut(E)^\ggo$-orbits), and that the Hessian of $\mc S$ vanishes in \emph{gauge directions} (i.e.~ those tangent to $\Aut(E)^\ggo$-orbits).

\subsection{Second variation of the generalized scalar curvature}\label{sec:2ndvar}

 As a first step towards proving \Cref{main_stabBF}, we compute the Hessian of the generalized scalar curvature at an arbitrary left-invariant generalized metric on a unimodular Lie group $\G$:

\begin{prop}\label{prop:gradAndHessOfS}
    Let $(\f G, \ip, \mud, \mc G)$ be a left-invariant ECA over  a unimodular Lie group $\msf G$, endowed with a left-invariant generalized metric $\mc G$.  
    Then, for all $Z =\left(\begin{smallmatrix}A&-g^{-1}\alpha g^{-1}\\\alpha&-A^*\end{smallmatrix}\right)\in \f{so}(\f G, \ip)\cap {\rm Sym}(\f G, \mc G)$ we have  
    \begin{equation}\label{eqn_HessianS}
     {\rm Hess}
        (\mc S)_{\mc G}(Z, Z)=-\frac14\left(\frac13|\Theta(Z)\bm \mu|^2+ \tr(B_{\bm \mu}Z^2)\right).
    \end{equation}
    If in addition ${\mc Rc} (\mc G) = 0$, then 
    \begin{equation}\label{eqn_HessianS_BRF}
        {\rm Hess}(\mc S)_{\mc G}(Z, Z)= \frac14\left(-| \theta(A) \mu |^2 - \frac13 |\rho(A)H|^2   - \frac23 | \dd_\mu \alpha|^2  + \frac43 \, g(\dd_\mu \alpha, \rho(A)H )  - 2 \, \tr B_\mu A^2\right).
    \end{equation}
\end{prop}
\begin{proof}

 Applying the moving brackets framework we get
$$
(\dd^2\mc S)_{\mc G}(Z, Z)=\frac{\dd^2}{\dd t^2}\Big|_{t=0}\mc S(\exp\left(-\tfrac t2 Z \right)\cdot \mc G)=\frac{\dd^2}{\dd t^2}\Big|_{t=0}\mc S_{\exp\left(\tfrac t2 Z \right)\cdot\bm \mu}=\frac12\frac{\dd^2}{\dd t^2}\Big|_{t=0}{\rm tr}\left(\frac23 M_{\exp\left(\tfrac t2 Z \right)\cdot\bm \mu}-\frac12B_{\exp\left(\tfrac t2 Z \right)\cdot\bm \mu}\right)\,.
$$ Now, using \Cref{eqn_derivatives},
$$
\begin{aligned}
\frac{\dd^2}{\dd t^2}\Big|_{t=0}{\rm tr}(M_{\exp\left(\tfrac t2 Z \right)\cdot\bm\mu})=&\, -\frac14\frac{\dd^2}{\dd t^2}\Big|_{t=0}|\exp\left(\tfrac t2 Z \right)\cdot  \bm \mu|^2=-\frac14\frac{\dd}{\dd t}\Big|_{t=0} g(\Theta(Z)\exp\left(\tfrac t2 Z \right)\cdot \bm \mu, \exp\left(\tfrac t2 Z \right)\cdot \bm \mu)\\
=&\, -\frac14\left( \frac12g(\Theta(Z)^2\bm \mu, \bm \mu)+\frac12|\Theta(Z)\bm \mu|^2\right)=-\frac14|\Theta(Z)\bm\mu|^2\,.
\end{aligned}
$$ 
Moreover, 
\begin{align*}
\frac{\dd^2}{\dd t^2} \Big|_{t=0} \tr(B_{\exp\left(\tfrac t2 Z \right)\cdot \bm \mu}) = \frac{\dd}{\dd t}\Big|_{t=0}\tr(B_{\exp\left(\tfrac t2 Z \right)\cdot \bm \mu}Z) =  \tr(B_{\bm \mu}Z^2),
\end{align*}
and \eqref{eqn_HessianS} follows.

Regarding the second claim,  we first note that the Bismut-Ricci flat condition yields
\begin{equation}\label{eqn:BRF_implications}
\dd_\mu^* H = 0, \qquad M_\mu = {\rm Ric}_\mu + \frac12 B_\mu  =  \frac14 H^2 + \frac12 B_\mu\,.
\end{equation}
Expanding  $|\Theta(Z) \mu_H|^2$ with the help of  \cite[Corollary 4.8]{HGRF}, we obtain
    \begin{equation}\label{eqn_firstterm}
    \begin{aligned}
    |\Theta(Z)  \mu_H|^2 =& 
    \left|\Theta \left( \begin{smallmatrix}A&0\\\alpha&-A^*\end{smallmatrix} \right) \mu_H + \Theta \big( \begin{smallmatrix}0&-g^{-1}\alpha g^{-1}\\0 &0\end{smallmatrix}\big)  \mu_H \right|^2\\
    =&\left| (\theta(A)\mu)_{\rho(A)H - \dd_\mu \alpha}\right|^2 + 2\, g\left( (\theta(A)\mu)_{\rho(A)H - \dd_\mu\alpha} \, , \,  \Theta \big( \begin{smallmatrix}0&-g^{-1}\alpha g^{-1}\\0 &0\end{smallmatrix}\big)  \mu_H\right)
     + \left | \Theta \big( \begin{smallmatrix}0&-g^{-1}\alpha g^{-1}\\0 &0\end{smallmatrix}\big)  \mu_H\right|^2\,.
    \end{aligned}
    \end{equation}
    We now deal with each of these terms separately.  Using \eqref{eqn:norm_mu_H}, we get 
    \begin{equation}\label{eqn_easyterm}
| (\theta(A)\mu)_{\rho(A)H - \dd_\mu \alpha}|^2 = 3 \, |\theta(A)\mu|^2+|\rho(A)H-\dd_{\mu} \alpha|^2\,.
    \end{equation} On the other hand,   
\begin{equation}\label{eqn_mildterm}
\begin{aligned}
    g\left( {(\theta(A)\mu)_{\rho(A)H - \dd_\mu\alpha}, \Theta \big( \begin{smallmatrix}0&-g^{-1}\alpha g^{-1}\\0 &0\end{smallmatrix}\big) \mu_H}\right) =&  \,  g\left(\Theta \left( \begin{smallmatrix}0&0\\\alpha &0\end{smallmatrix} \right) (\theta(A)\mu)_{\rho(A)H - \dd_\mu\alpha} \, ,  \,  \mu_H\right)\\
    =&\, -g(\dd_{\theta(A)\mu}\alpha, H) 
    =  -g(\dd_{\mu}\alpha, \rho(A)H)\,,
\end{aligned}
\end{equation} 
    where we have used that $\Theta\big(\begin{smallmatrix}0& -g^{-1}\alpha g^{-1}\\ 0&0\end{smallmatrix} \big)^t = \Theta\left(\begin{smallmatrix}0 &0 \\\alpha & 0 \end{smallmatrix}\right)$, \cite[Corollary 4.8]{HGRF},
   $\dd_{\theta(A)\mu}=[ \rho(A), \dd_{\mu}]$, harmonicity of $H$ and $\rho(A)^t  = \rho(A)$, in this order.
    Finally the fact that 
$\Theta$ is a Lie algebra representation gives
$$
\begin{aligned}
\left|{\Theta\left(\begin{smallmatrix}0 & -g^{-1}\alpha g^{-1}\\
0& 0 \end{smallmatrix}\right)}\mu_H\right|^2=&\,  g\left(\Theta\left(\left[ {\left(\begin{smallmatrix}0&0\\\alpha &0\end{smallmatrix} \right)}, {\big(\begin{smallmatrix}0&-g^{-1}\alpha g^{-1}\\0&0\end{smallmatrix} \big) }\right]\right) \mu_H, \mu_H\right)+ \left| {\Theta \big(\begin{smallmatrix}0&0\\\alpha &0\end{smallmatrix} \big) } \mu _H\right|^2\\
=&\, g\left(\Theta {\left(\begin{smallmatrix} g^{-1}\alpha g^{-1}\alpha & 0 \\ 0 & -\alpha g^{-1}\alpha g^{-1}\end{smallmatrix} \right)} \mu_H,  \mu_H\right) + |\dd_{\mu}\alpha|^2 \,.
\end{aligned}
$$ Now, we observe that $g^{-1}\alpha g^{-1}\alpha=-\alpha^2$, where $\alpha^2(X, Y):=g(\iota_X\alpha, \iota_Y \alpha)$, for all $X, Y \in \f g$. Indeed, 
$$
g((g^{-1}\alpha g^{-1}\alpha) X, Y)=\alpha( g^{-1}\iota_X\alpha , Y)=-\alpha(Y, g^{-1}\iota_X\alpha )=-g(\iota_X\alpha , \iota_Y\alpha )=-\alpha^2(X, Y)\,.
$$ Hence, 
\begin{equation}\label{eqn_hardterm}
\begin{aligned}
\left|{\Theta\big(\begin{smallmatrix}0 & -g^{-1}\alpha g^{-1}\\
0& 0 \end{smallmatrix} \big)}\mu_H\right|^2=&\, |\dd_{\mu}\alpha|^2- g\left( {\Theta \big(\begin{smallmatrix} \alpha^2 & 0 \\ 0 & -\alpha^2\end{smallmatrix} \big) } \mu_H,  \mu_H\right)
= |\dd_{\mu}\alpha|^2- 3\, g(\theta(\alpha^2)\mu, \mu)- g(\rho(\alpha^2)H, H)\, \\
=&  \,  |\dd_{\mu}\alpha|^2- 12 \, \tr \left( M_\mu \alpha^2 \right) + 3 \tr\left( H^2 \alpha^2 \right) 
=  |\dd_{\mu}\alpha|^2- 6 \, \tr \left( B_\mu  \alpha^2 \right),
\end{aligned}
\end{equation}
where the last step uses \Cref{eqn:BRF_implications}.
Finally, \Cref{eqn:KillingmuH_mu} gives 
\begin{equation}\label{eqn:trBmuHZ2}
    \tr \left(B_{\mu_H} Z^2\right) = 2 \tr \left( B_\mu (A^2 + \alpha^2) \right). 
\end{equation}
Inserting \eqref{eqn_easyterm}--\eqref{eqn:trBmuHZ2}
into \eqref{eqn_firstterm} immediately yields \eqref{eqn_HessianS_BRF}.
\end{proof}

\begin{cor}
For a left-invariant ECA $(\f G, \ip, \mud)$ over a unimodular Lie group $\G$ with positive semi-definite Killing form, the generalized scalar curvature $\mc S : \mc M_\gen^{\f G} \to \R$ is a weakly concave function, and it is strictly concave if considered modulo gauge.  
\end{cor}

\begin{proof}
Using \Cref{eqn_HessianS}, it is straightforward to notice that if  $B_{\mu}\ge 0$ then  ${\rm Hess} (\mc S)_{\mc G}\le 0 \,.$ As regards the second claim, we observe that ${\rm Hess}(\mc S)_{\mc G}(Z, Z)=0$ if and only if $|\Theta(Z)\mud|^2=0$ and ${\rm tr}(B_{\mud}Z^2)=0$. From the first, we  deduce that $Z\in {\rm Der}(\f G , \mud)\cap {\rm Sym}(\f G, \mc G )$. Now, we know that $\kappa_{\mud}$ makes derivations skew-symmetric, which amounts to saying that 
$$
ZB_{\mud}+ B_{\mud}Z=0\,,
$$ giving that ${\rm tr}(B_{\mud}Z^2)=0\,.$    
\end{proof}

\subsection{Linear stability of standard Bismut-flat pairs with respect to left-invariant variations}\label{sec:linearstab}

  As mentioned in \Cref{sec:preliminaries}, all Bismut-flat pairs on compact simply-connected manifolds are isometric to compact semisimple Lie groups with a bi-invariant metric $g$ and $H=g\mu\,.$ Our main goal here is to prove  the following.

\begin{prop}[Linear stability up to gauge]\label{prop_linearstab}
    Let $(\f g,\mu)$ be a compact semisimple Lie algebra, $g$ a bi-invariant metric on $\f g$, and $E = (\gggo)_H$ the left-invariant ECA defined by $H := g\mu \in \Lambda^3 \f g^*$. The generalized metric $\mc G = \mc G(g,0)$ satisfies
    \[
        {\rm Hess}(\mc S)_{\mc G}(Z,Z) \leq 0,
    \]
     for all $Z \in T_{\mc G} \mc M^{\f G }_{\operatorname{gen}}$, and equality holds if and only if $Z \in T_{\mc G} \big(\Aut(E)^\ggo \cdot \mc G\big)$.
\end{prop}

Before we proceed, we need to establish some technical lemmas. 
\begin{lem}\label{lem:ssHessian}
Under the assumptions of Proposition \ref{prop_linearstab}, we have that
\begin{equation}\label{eqn_lastterm}
\tr(B_\mu A^2) = -\frac{1}{4}\left(|\theta(A)\mu|^2 + \frac 13 |\rho(A)H|^2 \right)\, , \qquad H^2 = -B_\mu.
\end{equation} 
\end{lem}

\begin{proof}
    Fix $\{e_i\}$ to be a $g$-orthonormal basis of $\f g$  such that $Ae_i=\lambda_ie_i$. Beginning with the second claim, we have using skew-symmetry of $\mu$, that for all $X \in \f g$,
    \[
    H^2(X,X) = \sum_{ij}g(\ad_X e_i,e_j)g(\ad_X e_i,e_j) = \sum_i g(\ad_X e_i,\ad_X e_i) = -\sum_{i}g(\ad_X^2 e_i,e_i) =  -g(B_\mu X,X),
    \]
    as required. It also follows that
    \[
    \operatorname{tr}(B_\mu A^2) = -\sum_i \lambda_i^2 H^2(e_i,e_i) = -\sum_{ijk}\lambda_i^2(\mu_{ij}^k).
    \]
    Moreover, for any $1\leq i,j,k \leq n$, it holds that
$$
(\rho(A)H)_{ijk}=-(\lambda_i+\lambda_j+\lambda_k)\mu_{ij}^k\,, \quad (\theta(A)\mu)_{ij}^k=(\lambda_k-\lambda_i-\lambda_{j})\mu_{ij}^k.\,
$$
 Hence, 
 $$
 \begin{aligned}
|\rho(A)H|^2=&\, \sum_{i,j,k}(\rho(A)H)^2_{ijk}=\sum_{i,j,k}(\lambda_i+\lambda_j+\lambda_k)^2(\mu_{ij}^k)^2=\sum_{i,j,k}(\lambda^2_i+\lambda_j^2+\lambda_k^2+2\lambda_i\lambda_j+2\lambda_i\lambda_k+2\lambda_k\lambda_j)(\mu_{ij}^k)^2\\
=&\, 3\sum_{i,j,k}\lambda_i^2(\mu_{ij}^k)^2+6\sum_{i,j,k}\lambda_i\lambda_j(\mu_{ij}^k)^2\,,
\end{aligned}
 $$and
 $$
 \begin{aligned}
 |\theta(A)H|^2=&\, \sum_{i,j,k}(\rho(A)H)^2_{ijk}=\sum_{i,j,k}(\lambda_i-\lambda_j-\lambda_k)^2(\mu_{ij}^k)^2=\sum_{i,j,k}(\lambda^2_i+\lambda_j^2+\lambda_k^2+2\lambda_i\lambda_j-2\lambda_i\lambda_k-2\lambda_k\lambda_j)(\mu_{ij}^k)^2\\
=&\, 3\sum_{i,j,k}\lambda_i^2(\mu_{ij}^k)^2-2\sum_{i,j,k}\lambda_i\lambda_j(\mu_{ij}^k)^2\,,
\end{aligned}
 $$ 
 where the last line in both computations is obtained using skew-symmetry of $\mu$. The claim follows immediately by substitution.
\end{proof}

In what follows, for matrices $A,B \in \f{gl}(\f g)$, $[A,B] = AB - BA$ and $\{A,B\} = AB + BA$ denote the commutator and anticommutator respectively.
\begin{lem}\label{lem:ssdStarRho}
    Let $(g,H = g \mu)$ be a  standard bi-invariant Bismut-flat metric on a semisimple Lie algebra $(\f g , \mu)$. Then, for all $A \in \f{gl}(\f g)$, it holds that
    \[
        \dd^*_\mu \rho(A)H = \frac 12 g\circ [B_\mu,\operatorname{Sym}(A)] + \frac 12 g\circ \{B_\mu,\operatorname{Skew}(A)\} - \operatorname{tr}(A \ad_{\mu(\cdot,\cdot)}).
    \]
    In particular, if $A \in \operatorname{Sym}(\f g, g)$, then $\dd^*_\mu \rho(A)H = \frac 12 g[B_\mu,A]$. 
\end{lem}  

\begin{proof} Since $g$ is bi-invariant, its Levi-Civita connection satisfies $\nabla_XY = \frac 12 \mu(X,Y)$ and hence $\dd^*_\mu = -\iota_{e_i} \circ \rho(\ad_{e_i})$,  for  any $\{e_i\}$ $g$-orthonormal basis of $\f g$. Also note that the adjoint maps are skew-symmetric derivations, hence $\rho(\ad_{e_i})H = 0$,  for all $1 \leq i \leq n$. It follows that
\[
\dd^*_\mu \rho(A)H = -\frac 12 \iota_{e_i} \rho(\ad_{e_i})\rho(A)H = -\frac 12 \iota_{e_i}\rho([\ad_{e_i},A])H.
\]
Now, for all $X,Y \in \f g$, using $H = g\mu$ and $B_\mu =\ad_{e_i}\ad_{e_i}$, we compute
\begin{equation}\label{eqn_2dstarho}
\begin{aligned}
2\dd^*_\mu \rho(A)H(X,Y) &= -\rho([\ad_{e_i},A])H(e_i,X,Y)\\
&= g(\ad_{[\ad_{e_i},A]e_i}X,Y) + g(\ad_{e_i}[\ad_{e_i},A]X,Y) + g(\ad_{e_i}X,[\ad_{e_i},A]Y)\\
&= g(\ad_{\mu(e_i,Ae_i)}X,Y)  + g(\ad_{e_i}[\ad_{e_i},A]X,Y) + g([\ad_{e_i},A^t]\ad_{e_i}X,Y)\\
&= g(\ad_{\mu(e_i,Ae_i)}X,Y)  + g(B_\mu A X,Y) - g(\ad_{e_i}A\ad_{e_i}X,Y) \\
&\qquad + g(\ad_{e_i}A^t\ad_{e_i}X,Y) - g(A^tB_\mu X,Y).
\end{aligned}
\end{equation}
If $A \in \operatorname{Sym}(\f g , g)$, then $\mu(e_i,Ae_i) = 0$, $\tr(A \ad_{\mu(\cdot,\cdot)}) = 0$, and the claim follows. On the other hand, if $A \in \f{so}(\f g, g)$, then $\ad_{\mu(e_i,Ae_i)} = [\ad_{e_i},\ad_{Ae_i}] = -2 \ad_{Ae_i}\ad_{e_i}$. Indeed, we have, for any $X, Y\in\f g$, 
$$
\begin{aligned}
g([\ad_{e_i}, \ad_{Ae_i}]X,Y )=&\, g(\ad_{e_i}\ad_{Ae_i}X, Y)- g(\ad_{Ae_i}\ad_{e_i}X, Y)=g(\ad_{Ae_i}X, \ad_Y e_i)-g(\ad_{e_i}X, \ad_{Y}Ae_i)\\
=&\, {\rm tr}(\ad_Y\ad_XA+A\ad_Y\ad_X)=2{\rm tr}(A\ad_Y\ad_X)=-2g(\ad_{Ae_i}\ad_{e_i}X, Y)\,.
\end{aligned}
$$Using  the above computation in \Cref{eqn_2dstarho} gives
\[
2\dd^*_\mu \rho(A)H(X,Y) = -2 g(\ad_{Ae_i}\ad_{e_i}X,Y) - 2 g(\ad_{e_i}A\ad_{e_i}X,Y) + g((B_\mu A + AB_\mu)X,Y).
\]
Now focusing on the first two terms, the skew-symmetry of $A$ and all the adjoint maps implies
\begin{align*}
-g(\ad_{Ae_i}\ad_{e_i}X,Y) -g(\ad_{e_i}A\ad_{e_i}X,Y) &= g(\ad_{e_i}X,\ad_{Ae_i}Y) +g(A\ad_{e_i}X,\ad_{e_i}Y)\\
&= -g(\ad_{e_i}X,\ad_{Y}Ae_i) - g(A\ad_{e_i}X,\ad_Ye_i)\\
&= -g(A\ad_Y\ad_{e_i}X,e_i) + g(\ad_YA\ad_{e_i}X,e_i)\\
&= g([A,\ad_Y]\ad_{X}e_i,e_i)\\
&= \tr ([A,\ad_Y]\ad_X) = - \tr (A\ad_{\mu(X,Y)}),
\end{align*}
and the claim follows for skew-symmetric endomorphisms.
\end{proof}
\begin{cor} \label{cor:ddstaralpha}
    If $\alpha \in \Lambda^2 \f g^*$, then $\dd_\mu^* \dd_\mu \alpha = -\frac 12 g \circ \{B_\mu,g^{-1}\alpha\}  +\operatorname{tr}(g^{-1}\alpha \ad_{\mu(\cdot,\cdot)})$.
\end{cor}
\begin{proof}
    This follows from Lemma \ref{lem:ssdStarRho} and the fact that $\dd_\mu\alpha = - \rho(g^{-1}\alpha)H$.
\end{proof}
 Given any bi-invariant metric $g$ on $\f g$, we can consider the orthogonal decomposition of $\f g = \f g_1 \oplus \dots \oplus \f g_k$ into the eigenspaces of $B_\mu$: write $B_\mu = -\bigoplus_{i=1}^k b_i \id_{\f g_i}$, where $b_i > 0$. Note that each eigenspace $\f{g}_i$ is an ideal of $\f g$. This induces orthogonal decompositions
\[
\f{gl}(\f g) = \sum_{i,j}\f g_i^* \otimes \f g_j, \qquad \f{so}(\f g,g) = \bigoplus_{i,j}\f{so}_{ij},\qquad \operatorname{Sym}(\f g ,g) = \bigoplus_{i,j} \operatorname{Sym}_{ij},
\]
where $\f{so}_{ij} := \{A - A^t\, :\,  A \in \f g_i^* \otimes \f g_j\}$ and $ \operatorname{Sym}_{ij} := \{A+A^t\, :\,  A\in \f g_i^* \otimes \f g_j\}$. By Lemma \ref{lem:ssdStarRho}, the map
\[
\f{gl}(\f g) \ni A \mapsto g^{-1} \dd^*_\mu\rho(A)H \in \f{so}(\f g , g),
\]
preserves the above decomposition in the sense that it maps $\f g_i^* \otimes \f g_j + \f g_j^* \otimes \f g_i$ to $\f{so}_{ij}$ for all $1 \leq i,j \leq k$. Moreover, it follows that $\dd^*_\mu \dd_\mu \colon \Lambda^2 \f g^* \to \Lambda^2 \f g^*$ also respects the decomposition, sending $g(\f{so}_{ij})$ to $g(\f{so}_{ij})$.

\begin{proof}[Proof of Proposition \ref{prop_linearstab}] Using that $T_{\mc G(g,0)} \mc M^{\f G }_{\operatorname{gen}}=\f{so}(\Ggo,\mdot) \cap \operatorname{Sym}(\mc G(g,0))$,  we consider an arbitrary $Z := \left( \begin{smallmatrix}A&-g^{-1}\alpha g^{-1}\\\alpha&-A^*\end{smallmatrix} \right)\in \f{so}(\f G, \ip)\cap {\rm Sym}(\f G, \mc G(g, 0))$. We decompose $A \in \operatorname{Sym}(\f g,g)$ as $A = \sum_{ij} A_{ij}$, where $A_{ij} \in \f g_i^* \otimes \f g_j$ and $A_{ij} = A_{ji}^t$. For $1 \leq i\leq j \leq k$, denote by $\alpha_{ij}$ the orthogonal projection of a $2$-form $\alpha$ onto the subspace $g(\f{so}_{ij}) \subset \Lambda^2 \f g^*$. By the above remarks and Corollary \ref{cor:ddstaralpha}, we may decompose
\begin{align*}
\dd^*_\mu \dd_\mu \alpha &= \sum_{i} (\dd_\mu^*\dd_\mu \alpha)_{ii} - 
\frac 12\sum_{i <  j}g\{B_\mu,g^{-1}\alpha_{ij}\} +\sum_{i<j}\tr\left(g^{-1}\alpha_{ij}\ad_{\mu(\cdot,\cdot)}\right).
\end{align*}
Since each $\f g_i$ is an ideal, the last term in the above equality vanishes, so we get
\begin{equation} \label{eqn:ddstaralpha}
    \dd^*_\mu \dd_\mu \alpha = \sum_{i} (\dd^*_\mu \dd_\mu \alpha)_{ii} + \frac 12 \sum_{i<j}(b_i + b_j)\alpha_{ij}\,.
\end{equation}
Similarly, using $A_{ij} = A_{ji}^t$, we obtain a decomposition
\begin{align*}
    \dd^*_\mu \rho(A)H &= \sum_i (\dd^*_\mu \rho(A)H)_{ii} +2 \sum_{i < j} \dd^*_\mu \rho(\operatorname{Sym}(A_{ij}))H = (\dd^*_\mu \rho(A)H)_{ii} + \sum_{i < j} g[B_\mu,\operatorname{Sym}(A_{ij})]\,.
\end{align*}
Hence, 
\begin{equation}\label{eqn:drhoA}
   \dd^*_\mu \rho(A)H = \sum_i (\dd^*_\mu \rho(A)H)_{ii} + \sum_{i<j}(b_i-b_j)g\operatorname{Skew}(A_{ij})\,.
\end{equation}
We note that, for $i < j$,  since $A_{ij} \colon \f g_i \to \f g_j$, it holds that $|A_{ij}|^2 = 2 |\operatorname{Skew}(A_{ij})|^2$.
Thus, \Cref{eqn_HessianS_BRF}, combined with \Cref{lem:ssHessian} and equations \eqref{eqn:ddstaralpha} and \eqref{eqn:drhoA} yields
\begin{align*}
{4\rm Hess}( \mc S)_{\mc G(g, 0)}(Z, Z) &= 2 \tr (B_\mu A^2) - \tfrac 23 |\dd_\mu \alpha|^2 + \frac{4}{3}g(\dd_\mu \alpha, \rho(A)H)\\
&= 2 \tr (B_\mu A^2) - 2g(\alpha, \dd^*_\mu \dd_\mu \alpha) + 4 g(\alpha, \dd^*_\mu\rho(A)H)\\
&= -2 \sum_i b_i |A_{ii}|^2 - \tfrac 23 \sum_{i} |\dd_\mu \alpha_{ii}|^2 \\
&\qquad - 2\sum_{i<j} (b_i + b_j) |A_{ij}|^2   -\sum_{i<j}g(\alpha_{ij},(b_i + b_j)\alpha_{ij}) - 4\sum_{i<j} g(\alpha_{ij},(b_j-b_i)g\operatorname{Skew}(A_{ij}))\,.
\end{align*}
Therefore, by Cauchy--Schwarz inequality and the fact that $|\frac{b_i - b_j}{b_i + b_j}| < 1$, we obtain 
\begin{align*}
4{\rm Hess}_{\mc G(g, 0)}(Z, Z)&\leq \sum_{i<j}(b_i+b_j)\left(-2|A_{ij}|^2 - |\alpha_{ij}|^2 + 4 \frac{b_i-b_j}{b_i + b_j} g(g^{-1}\alpha_{ij},\operatorname{Skew}(A_{ij}))\right)\\
&\leq  \sum_{i<j}(b_i+b_j)\left(-2|A_{ij}|^2 - |\alpha_{ij}|^2 + 4 |\alpha_{ij}||\operatorname{Skew}(A_{ij})|\right)\\
&\leq \sum_{i<j}(b_i+b_j)\left(-2|A_{ij}|^2 - |\alpha_{ij}|^2 + 2\sqrt{2} |\alpha_{ij}||A_{ij}|\right)\\
&= -2\sum_{i<j}(b_i + b_j) \left(|A_{ij}| - 2^{-\frac 12}|\alpha_{ij}|\right)^2 \leq 0\,.
\end{align*}
 We see that equality holds if and only if $A = 0$ and $\alpha_{ij}$ vanishes if $i \neq j$ or is closed if $i = j$ for all $1 \leq i \leq j \leq k$. That is, if and only if $Z = \operatorname{Sym}_{\mc G(g,0)}(\alpha)$ for $\alpha \in \Lambda^2 \f g^* \subset \operatorname{Hom}(\f g,\f g^*) \subset \f{gl}(\Ggo)$ closed. Since derivations of $(\f g, \mu)$ are skew-symmetric for  bi-invariant metrics, this is precisely the condition that $Z \in T_{\mc G(g,0)}(\Aut(E)^\ggo \cdot \mc G(g,0))$.
\end{proof}

\subsection{Dynamical stability  of standard Bismut-flat metrics under homogeneous generalized Ricci flows}\label{sec:dynstab}

In this subsection, we  derive dynamical stability of standard bi-invariant Bismut-flat metrics from \Cref{prop_linearstab}.
Let $\mc G_{{\rm b}}$ be a  standard bi-invariant Bismut-flat  metric on a left-invariant ECA $E\to \G$, where $\G$ is a compact, simply-connected, semisimple Lie group. After fixing an invariant isotropic splitting, such an ECA is determined by the infinitesimal data $((\f G, \mu_H) \overset{\pi_1}\longrightarrow (\ggo,\mu), \ip)$ (see $\S$\ref{sec:homog}). The main theorem of this subsection reads as follows:

\begin{thm}
Let $\mc G_{{\rm b}}$ be a standard bi-invariant Bismut-flat  metric on an invariant ECA $E\to \G$ over a compact, simply-connected, semisimple Lie group. Then, there exists $\delta>0$ such that if $d(\mc G_0, \mc G_{{\rm b}})<\delta$, then the solution $\mc G_t$ of the generalized Ricci flow starting at $\mc G_0$ exists for all positive time and converges exponentially fast in the $C^\infty$ sense to a bi-invariant Bismut-flat  metric $\gca_\infty = \varphi\cdot \gca_{\rm b}$, for some $\varphi \in \Aut(E)^\G$.
\end{thm}

\begin{proof}
The proof follows almost directly from the linear stability in Proposition \ref{prop_linearstab}, and a version of the Center Manifold Theorem. A key observation is that the orbit $\Aut(E)^{\f g} \cdot \gca_{\rm b} \subset \mca^{\f G}_{\rm gen}$ is a center manifold for the generalized Ricci flow ODE on the space $\mca^{\f G}_{\rm gen}$ of left-invariant generalized metrics. Indeed, it is invariant, as it consists entirely of bi-invariant Bismut-flat pairs, which are fixed points, and its tangent space at $\gca_{\rm b}$ is precisely $\ker({\rm Hess}(\mathcal{S})_{\mc G_{{\rm b}}})$, again by  \Cref{prop_linearstab} and \Cref{lem_autEG}.
Since the center manifold consists entirely of fixed points, $\gca_{\rm b}$ is stable for the induced flow on such manifold. Hence, we may apply  \cite[Theorem 2 (b)]{Carr81}, from which it follows that any solution starting sufficiently close to $\gca_{\rm b}$ will be exponentially close to a solution on the center manifold. That is, it converges exponentially fast to a bi-invariant Bismut-flat metric $\mc G_\infty = F \cdot \mc G_{\rm b} \in \Aut(E)^{\f g} \cdot \gca_{\rm b}$. By simply-connectedness, $F$ integrates to an automorphism $\varphi \in \operatorname{Aut}(E)^{\G}$ such that $\mc G_\infty = \varphi \cdot \mc G_{\rm b}$ as claimed.
\end{proof}
\begin{rmk}\label{rmk_conj}
Let $(g_{{\rm b}}, H_{g_{{\rm b}}}) $ be the pair associated to $\mc G_{{\rm b}} $. Notwithstanding the fact that $\mc G_{\infty}$ might be different from $\mc G_{{\rm b}}$, their  associated pairs coincide. To see this,   we know by \Cref{prop_aut} and \Cref{lem_autEG} that $\varphi=\bar f e^b$,  for some $f\in {\rm Aut}(\msf G) \subseteq {\rm Isom}(\msf G, g_{{\rm b}})$ and $b\in \Omega^2(\msf G)^{\msf G} $,  such that $f^*H_{g_{{\rm b}}}=H_{g_{{\rm b}}}+ \dd b$.  Since $H_{g_{{\rm b}}}$ is bi-invariant,  $f^*H_{g_{{\rm b}}}=H_{g_{{\rm b}}}$ and, hence, $\dd b=0$. Now the claim follows from the fact that the pair associated to $\varphi\cdot\mc G_{{\rm b}}$ is given by $((f^{-1})^* g_{{\rm b}}, (f^{-1})^*(H_{g_{\rm b}}+\dd b))$. 
\end{rmk}

\section{Dynamically unstable Bismut-flat metrics}\label{sec_BRF}

In this section, we prove the dynamic instability of certain product Bismut-flat metrics on compact Lie groups.

Let $\msf G = \msf K \times \msf K$, where $\msf K$ is a compact, simply-connected, simple Lie group, $\f g = \f k \oplus \f k$ and $\f k$ the Lie algebras of $\msf G$ and $\msf K$, and $B$ and $B_{\f k}$, respectively,  the Killing form of $\f g$ and of $\f k$. The Cartan $3$-form $\B H := -B  [\cdot,\cdot]$ is bi-invariant, hence the action of $\mathsf G \times \msf G$ on $\msf G$ defined by $(g,h)\cdot x := gxh^{-1}$ lifts to an action on the ECA $(\gggo)_{\B H}$.

We consider generalized metrics on $(\gggo)_{\B H}$ that are invariant under the action of the subgroup $$\msf G \times \Delta \msf K = \{(x,(y,y)) \in \msf G \times \msf G\} \leq \msf G \times \msf G.$$ By Proposition \ref{prop_GinvTT}, any such metric is of the form $\mc G(g,\alpha)$, where $g$  is a $(\msf G \times \Delta \msf K)$-invariant metric and $\alpha \in \Omega^2(\msf G)^{\msf G \times \Delta \msf K}$. After identifying $g$ and $\alpha$ with their restrictions to $\f g$, their invariance under $\msf G \times \Delta \msf K$ is equivalent to $g \in \operatorname{Sym}^2_+(\f g)^{\operatorname{Ad}(\Delta \msf K)}$ and $\alpha \in (\Lambda^2 \f g^*)^{\operatorname{Ad}(\Delta \msf K)}$. To understand the structure of such  $g$ and $\alpha$, we define the maps
	\[
		F_{\pm} = (\cdot)^{\pm} \colon \f k \to \f g;\qquad F_{\pm}(X) = X^\pm := \frac 1{\sqrt{2}}\left(X , \pm X\right), \quad X\in \f k.
	\]
    Letting $\B g := - B$ and $\B g_{\f k} := -B_{\f k}$, the maps $F_{\pm}\colon (\f k, \B g_{\f k}) \to (\f g,\B g)$ are isometries onto their images and we have an $\Ad(\Delta \msf K)$-invariant decomposition 
    \[  
    \f g = \f k_+ \oplus \f k_-, \qquad \f k_{\pm} := \{ X^\pm : X\in \f k \}.
    \]
    With a slight abuse of  notation, we will denote the Lie brackets on $\f k$ and $\f g$ both by $\mu$. With this in mind, for all $X,Y \in \f k$, we have the bracket relations
	\[
		\mu(X^+,Y^+) = \mu(X^-,Y^-) = \frac{1}{\sqrt{2}} \mu(X,Y)^+, \qquad \mu(X^+,Y^-) = \frac 1{\sqrt{2}}\mu(X,Y)^-.
	\]
    We are now able to describe the space $(\Lambda^2 \f g^*)^{\operatorname{Ad}(\Delta \msf K)}$.
	
	\begin{lem}\label{lem_beta}
    It holds that $(\Lambda^2 \f{g}^*)^{\Ad(\Delta \msf K)} = \R \beta$, where $\beta \in (\Lambda^2 \f g^*)^{\Ad(\Delta \msf K)}$ is defined by 
    \[
        \beta(X^+,Y^+) = \beta(X^-,Y^-) = 0,\qquad \beta(X^+,Y^-) = -\beta(Y^-,X^+) =  -B_{\f k}(X,Y), \qquad X,Y\in \f{k}.
    \]
	\end{lem}

    \begin{proof}
    Let $\rho : \msf K \to \msf{Gl}(\f k)$ denote the adjoint representation. As $\Ad(\Delta \msf K)$-module, we have $\f g^* = \f k^*_+ \oplus \f k^*_- \simeq  2 \, \rho$, hence  the decomposition of $\Lambda^2 \f g^*$ into $\Ad(\Delta \msf K)$-modules is 
	\begin{align*}
		\Lambda^2 \f g^* &= \Lambda^2 \f k_+^* \oplus \Lambda^2 \f k_-^* \oplus (\f k_+^* \otimes \f k_-^*) \\
            &\simeq 3 \, \Lambda^2 \rho \oplus \Sym^2 \rho.
	\end{align*}
    Thus, 
    \[
        (\Lambda^2 \f{g}^*)^{\Ad(\Delta \msf K)} = 3 \, (\Lambda^2 \rho)^{ \msf K} \oplus (\Sym^2 \rho)^{\msf K}  = 3\cdot 0 \oplus \R B_{\f k},
    \]
    since there are no alternating invariant bilinear forms on $\f k$, and the only symmetric one up to scaling is the Killing form. Transforming $B_{\f k}$ back under the above isomorphisms yields $\beta$.
	\end{proof}

    \begin{lem} \label{lem:dbeta}The only non-zero components of $\dd \beta$ are given by
		\[
			\dd \beta (X^-,Y^+,Z^+) = \frac1{\sqrt{2}} \B{H}_{\f k}(X,Y,Z), \qquad \dd \beta (X^-,Y^-,Z^-) = -\frac{3}{\sqrt{2}} \B H_{\f k}(X,Y,Z) ,
		\]
		for $X,Y,Z \in \f k$, where $\B H_{\f k} := -B_{\f k}\mu$.
	\end{lem}
	\begin{proof} Using the standard formula for the exterior derivative, the first equality can be computed by
	\begin{align*}
		\dd \beta(X^-,Y^+,Z^+) &= -\frac 1{\sqrt{2}} \left(\beta(\mu(X,Y)^-,Z^+) + \beta(\mu(Y,Z)^+,X^-) + \beta(\mu(Z,X)^-,Y^+)\right)\\
		&= -\frac{1}{\sqrt{2}}\left(B_{\f k}(\mu(X,Y),Z) - B_{\f k}(\mu(Y,Z),X) + B_{\f k}(\mu(Z,X),Y)\right) = \frac 1{\sqrt{2}}H_{\f k}(X,Y,Z).
	\end{align*}
	Similarly,
	\begin{align*}
		\dd \beta (X^-,Y^-,Z^-) = -\frac{1}{\sqrt 2}\sum_{\operatorname{cyc}(X,Y,Z)}\left(\beta(\mu(X,Y)^+,Z^-)\right) = -\frac{3}{\sqrt{2}}H_{\f k}(X,Y,Z),
	\end{align*}
    as claimed.
	\end{proof}

    The $\Ad(\Delta \msf K)$-representations $\f k_+$ and $\f k_-$ are equivalent. Hence, by Schur's Lemma, any metric $g \in \operatorname{Sym}^2_+(\f g)^{\operatorname{Ad}(\Delta \msf K)}$ takes the form
    \[
         g_{a,b,\ell} := \B g P, \qquad  P = \begin{pmatrix}a \id_{\f k_+}&\ell F_+ F_-^{-1}\\ \ell F_- F_+^{-1}&b\id_{\f k_-}\end{pmatrix},
    \]
    in the splitting $\f g = \f k_+ \oplus \f k_-$, where $a,b > 0$, $\ell \in \R$, and $ab -\ell^2 > 0$. By Lemma \ref{lem_beta}, any $(\msf G \times \Delta \msf K)$-invariant metric on $(\gggo)_{\B H}$ is of the form $\mc G(g_{a,b,\ell},r\beta)$, for $r \in \R$. Clearly, $\mc G(g_{1,1,0},0)$ is Bismut-flat.

    We now observe that there is a simply-transitive action of $\msf G$ on itself induced by the subgroup 
    \begin{equation}\label{eqn_twistedAction}
        \msf G \cong \{e\}\times \msf K \times \Delta \msf K \leq \msf G \times \Delta \msf K.
    \end{equation}
    Using this, we are able to distinguish another invariant Bismut-flat generalized metric on $(\gggo)_{\B H}$.

    \begin{prop}\label{prop_BF} The above action induces an isometry between $((\gggo)_{\B H},\mc G(g_{1,3,1},\beta))$ and the ECA $((\gggo)_{-2\B H_{\f k}\oplus \B H_{\f k}} ,\mc G(2 \B g_{\f k} \oplus \B g_{\f k},0)).$ In particular, $\mc G(g_{1,3,1},\beta)$ is Bismut-flat. 
    \end{prop}

    \begin{proof}
        First, the isotropic splitting induced by $\mc G(g_{1,3,1},\beta)$ gives rise to an (equivariant) isometry between $((\gggo)_{\B H},\mc G(g_{1,3,1},\beta))$ and $((\gggo)_{\B H + \dd \beta},\mc G(g_{1,3,1},0))$. Moreover, the simply-transitive action of the subgroup defined in \eqref{eqn_twistedAction} induces a diffeomorphism of $\msf G$. We claim that, after pulling back by this diffeomorphism, $g_{1,3,1} = 2\B g_{\f k} \oplus \B g_{\f k}$ and $\B H + \dd \beta = -2\B H_{\f k} \oplus \B H_{\f k}$. Indeed, if $(\cdot)_{i} \colon \f k \to \f g, i = 1,2$, denotes the inclusion into the first and second factor, the action field induced by \eqref{eqn_twistedAction} evaluated at the identity is given by
        \[
            (X_1 + Y_2)^*_e = -X_1 - X_2 + Y_2 = -\sqrt{2}X^+ + \frac 1{\sqrt{2}} (Y^+ - Y^-),
        \]
        for all $X,Y \in \f k$. A simple computation then shows that
        \[
            |(X_1 + Y_2)^*_e|^2_{g_{1,3,1}} = 2|X|_{\B g_{\f k}}^2 + |Y|_{\B g_{\f k}}^2 = |X_1 + Y_2|^2_{2 \B g_{\f k} \oplus \B g_{\f k}}.
        \]
        Similarly, one can compute $\B H + \dd \beta$ on action fields using Lemma \ref{lem:dbeta} to yield the claimed identification with $-2\B H_{\f k} \oplus \B H_{\f k}$. Thus, the diffeomorphism induced by the action of the subgroup in \eqref{eqn_twistedAction} lifts to the claimed isometry of metric ECAs.
    \end{proof}
    Note that Proposition \ref{prop_BF} implies that on left-invariant ECAs over compact, simply connected Lie groups, Bismut-flat metrics are \emph{not} unique.

    We now turn to computing the generalized scalar curvature of the left-invariant generalized metrics $\mc G(g_{a,b,\ell},r\beta)$ in order to prove dynamical instability for the generalized Ricci flow.

    	\begin{lem}\label{lem_normH}
			With respect to the metric $g := g_{a,b,\ell}$ we have
            \begin{align}
            \begin{split}
                |\B H|_g^2 &=\frac{n}{2} \frac{b^3  + 3 a^2b + 6(a + b)\ell^2 }{(ab - \ell^2)^3}, \\
                \m{\dd \beta}{\B H}_g &= \frac{3n}{2}\frac{\ell(4 a^2 - (a+b)^2)}{(ab - \ell^2)^3},\\
                |\dd\beta|^2_g &= \frac{3n}{2} \frac{2(3a - b)(ab-\ell^2) +3a(a-b)^2}{(ab-\ell^2)^3},
            \end{split}
            \end{align}
            where  the $2$-form $\beta \in (\Lambda^2 \f g^*)^{\Ad(\Delta \msf K)}$  is defined in Lemma \ref{lem_beta}.
	\end{lem}
	\begin{proof}
        Starting with the last formula, let $\{v_i\}$ be a $\B g_{\f k}$-orthonormal basis of $\f k$, and write $H_{ijk} := \B H_{\f k}(v_i,v_j,v_k)$. Then, by Lemma \ref{lem:dbeta},
		\begin{align*}
			|\dd \beta|^2_g &= \m{P\cdot \dd \beta}{\dd \beta}_{\B g} = 3 (P\cdot \dd \beta) (v_i^-,v_j^+,v_k^+)(\dd \beta)(v_i^-,v_j^+,v_k^+) + (P \cdot \dd \beta )(v_i^-,v_j^-,v_k^-)(\dd \beta)(v_i^-,v_j^-,v_k^-)\\
			&= \frac{3}{\sqrt{2}}H_{ijk}\left(P\cdot \dd \beta (v_i^-,v_j^+,v_k^+) - P \cdot \dd \beta (v_i^-,v_j^-,v_k^-)\right).
		\end{align*} 
		Now, if we let 
        \begin{equation}\label{eqn_P-1}
        P^{-1} = \begin{pmatrix}u\id_{\f k_+}&wF_+ F_-^{-1}\\w F_- F_+^{-1}&v\id_{\f k-}\end{pmatrix} = \frac{1}{ab - \ell^2} \begin{pmatrix}b\id_{\f k_+}&-\ell F_+ F_-^{-1} \\ -\ell F_- F_+^{-1} & a\id_{\f k_-}\end{pmatrix}\,,
        \end{equation}
        then
		\begin{align*}
			P\cdot \dd \beta (v_i^-,v_j^+,v_k^+) &=\dd \beta (P^{-1}v_i^-,P^{-1}v_j^+,P^{-1}v_k^+)\\
			 &= vw^2\, \dd \beta(v_i^-,v_j^-,v_k^-) + vu^2\, \dd \beta(v_i^-,v_j^+,v_k^+) +  w^2u\,\dd \beta(v_i^+,v_j^-,v_k^+) + w^2u\,\dd \beta(v_i^+,v_j^+,v_k^-)\\
			 &= \frac{1}{\sqrt{2}}\left(-3vw^2 +2uw^2 + vu^2\right)H_{ijk}
			 = \frac{-3a\ell^2 + 2 b \ell^2 + a b^2}{\sqrt{2}(ab - \ell^2)^3}H_{ijk}.
		\end{align*}
		Similarly,
		\begin{align*}
			P \cdot \dd \beta (v_i^-,v_k^-,v_k^-) &=\dd \beta (P^{-1}v_i^-,P^{-1}v_j^-,P^{-1}v_k^-)\\
			&=  v^3 \dd\beta(v_i^-,v_j^-,v_k^-) +  vw^2\left( \dd\beta(v_i^-,v_j^+,v_k^+) + \dd\beta(v_i^+,v_j^-,v_k^+) +  \dd\beta(v_i^+,v_j^+,v_k^-) \right)\\
			&= \frac{-3v^3 + 3vw^2}{\sqrt{2}}H_{ijk}
			= \frac{3(a\ell^2 -a^3)}{\sqrt{2}(ab-\ell^2)^3}H_{ijk}.
		\end{align*}
		Putting these together with $\sum_{ijk}H_{ijk}^2 = n := \dim \f k$ yields the claimed formula.
       Regarding the inner product term, we have
		\begin{align*}
			\m{\B H}{\dd \beta}_g &= \m{P\cdot \B H}{\dd \beta}_{\B g} =  3 P\cdot \B H (v_i^-,v_j^+,v_k^+)(\dd \beta)(v_i^-,v_j^+,v_k^+) + P \cdot\B H (v_i^-,v_j^-,v_k^-)(\dd \beta)(v_i^-,v_j^-,v_k^-)\\
			&= \frac{3}{\sqrt{2}}\left(P\cdot \B H (v_i^-,v_j^+,v_k^+)- P \cdot\B H (v_i^-,v_j^-,v_k^-)\right)H_{ijk}\,.
		\end{align*}
        The stated formula follows from this, together with the  two computations below:
		\begin{align*}
			P \cdot \B H(v_i^-,v_j^+,v_k^+) &= 	\B H(P^{-1}v_i^-,P^{-1}v_j^+,P^{-1}v_k^+)\\
			&= wu^2\,\B H(v_i^+,v_j^+,v_k^+) + w^3\, \B H(v_i^+,v_j^-,v_k^-) + vuw \,\B H(v_i^-,v_j^+,v_k^-) + vwu\,\B H(v_i^-,v_j^-,v_k^+)\\
			&= \frac{1}{\sqrt{2}}(u^2w + w^3 + 2uvw)H_{ijk}
			= -\frac{1}{\sqrt{2}} \frac{\ell^3 + b^2 \ell + 2 a b \ell}{(ab - \ell^2)^3}H_{ijk},
		\end{align*}
		\begin{align*}
			P \cdot \B H(v_i^-,v_j^-,v_k^-) &= \B H(P^{-1}v_i^-,P^{-1}v_j^-,P^{-1}v_k^-)
			= w^3\, \B H (v_i^+,v_j^+,v_k^+) + 3 wv^2\,\B H(v_i^+,v_j^-,v_k^-)\\
			&= \frac{w^3 + 3wv^2}{\sqrt{2}}H_{ijk}
			= -\frac1{\sqrt 2}\frac{\ell^3 + 3 \ell a^2}{(ab - \ell^2)^3}H_{ijk}.
		\end{align*}
        Finally, we note that $\B H$ is supported on $\Lambda^3 \f k_+ \oplus \f k_- \wedge \f k_- \wedge \f k_+$, hence
		\begin{align*}
			|\B H|_g^2 &= P\cdot \B H (v_i^+,v_j^+,v_k^+)\B H(v_i^+,v_j^+,v_k^+) + 3  P\cdot \B H (v_i^+,v_j^-,v_k^-)\B H(v_i^+,v_j^-,v_k^-)\\
			&= \frac{1}{\sqrt{2}}\left( P\cdot \B H (v_i^+,v_j^+,v_k^+) + 3  P\cdot \B H (v_i^+,v_j^-,v_k^-)\right)H_{ijk}.
		\end{align*}
        We compute:
		\begin{align*}
			P\cdot \B H(v_i^+,v_j^+,v_k^+) &= u^3 \B H(v_i^+,v_j^+,v_k^+) + 3 uw^2 \B H(v_i^+,v_j^-,v_k^-) = \frac{u^3 + 3 uw^2}{\sqrt{2}}H_{ijk} = \frac{b^3 + 3 b \ell^2}{\sqrt{2}(ab - \ell^2)^3}H_{ijk}.
		\end{align*}
		and
		\begin{align*}
			P\cdot \B H(v_i^+,v_j^-,v_k^-) &= uw^2 \B H(v_i^+,v_j^+,v_k^+) + uv^2 \B H(v_i^+,v_j^-,v_k^-) + vw^2 \B H(v_i^-,v_j^+,v_k^-) + vw^2\B H(v_i^-,v_j^-,v_k^+)\\
			&= \frac{uw^2 + uv^2 + 2 vw^2}{\sqrt{2}}H_{ijk}
			= \frac{b\ell^2 + a^2b + 2 a \ell^2}{\sqrt{2}(ab - \ell^2)^3}H_{ijk}.
		\end{align*}
        The lemma follows.
	\end{proof}

We are now in the position to explicitly compute the generalized scalar curvature of $\mc G(g_{a, b, \ell}, r\beta)$.    \begin{prop}\label{prop_genscaladhinv}
		The generalized scalar curvature of $\mc G(g_{a,b,\ell},r\beta)$ is given by
		\begin{align*}
		\mc S(\mc G(g_{a, b, l},r\beta)) 
            =& \frac{n}{8} \frac{4a^2 b + ab^2 - a^3 - 2 \ell^2(a+b)}{(ab-\ell^2)^2}  \\
            & -\frac{n}{8}\left(\frac{3 a^2 b + b^3 + 6 \ell^2 (a+b)}{3 (a b - \ell^2)^3} 
                +\frac{2r \ell (4 a^2 - (a+b)^2)}{ (ab-\ell^2)^3}  + \frac{r^2 (3 a (a-b)^2 + 2 (3 a - b)(ab - \ell^2))}{ (ab-\ell^2)^3}\right),
		\end{align*}
        where $n = \dim \msf K$.
	\end{prop}
	\begin{proof}
    The generalized scalar curvature of $\mc G(g_{a,b,\ell},r\beta)$ is given by
    \[
        \mc S(\mc G(g_{a,b,\ell},r\beta)) = \operatorname{scal}(g_{a,b,\ell}) - \frac 1{12} |\B H + r\, \dd \beta|^2_{g_{a,b,\ell}}.
    \]
    Thus, in view of Lemma \ref{lem_normH}, we simply need to compute the classical scalar curvature of $g:=g_{a,b,\ell}$, which is given by $\operatorname{scal}(g) = -\frac{1}{4}|\mu|^2_g + \frac{1}2\operatorname{tr}(P^{-1})$. The norm term is 
		\begin{align*}
			\sqrt{2}|\mu|^2_g &= \m{P\cdot \mu (v_i^+,v_j^+)}{\mu(v_i,v_j)^+}_{\B g} + \m{P\cdot \mu (v_i^-,v_j^-)}{\mu(v_i,v_j)^+}_{\B g} + 2\m{P\cdot \mu (v_i^+,v_j^-)}{\mu(v_i,v_j)^-}_{\B g}\\
			&= \left(\m{P\cdot \mu (v_i^+,v_j^+)}{v_k^+}_{\B g} + \m{P\cdot \mu (v_i^-,v_j^-)}{v_k^+}_{\B g} + 2 \m{P\cdot \mu (v_i^+,v_j^-)}{v_k^-}_{\B g} \right)H_{ijk}.
		\end{align*}
		Thus,  writing $P^{-1}$ as in \eqref{eqn_P-1}
         the first term is
		\begin{align*}
			\m{P\cdot \mu (v_i^+,v_j^+)}{v_k^+}_{\B g} &= \B H(P^{-1}v_i^+,P^{-1}v_j^+,P v_k^+)\\
			&= u^2a \B H(v_i^+,v_j^+,v_k^+) + uw\ell \B H(v_i^+,v_j^-,v_k^-) + wu\ell\B H(v_i^-,v_j^+,v_k^-) + w^2 a\B H(v_i^-,v_j^-,v_k^+)\\
			&= \frac{1}{\sqrt{2}}\left(u^2a + 2 uw \ell + w^2a\right)H_{ijk}
			= \frac{1}{\sqrt{2}} \frac{ab^2 - 2 b\ell^2 + a\ell^2}{(ab - \ell^2)^2}H_{ijk}.
		\end{align*}
		Similarly,
		\begin{align*}
				\m{P\cdot \mu (v_i^-,v_j^-)}{v_k^+}_{\B g} &= \B H (P^{-1}v_i^-,P^{-1}v_j^-, Pv_k^+)\\
				&= aw^2 \B H(v_i^+,v_j^+,v_k^+) + wv\ell \B H(v_i^+,v_j^-,v_k^-) + vw\ell\B H(v_i^-,v_j^+,v_k^-) + v^2a\B H(v_i^-,v_j^-,v_k^+)\\
				&= \frac{aw^2 + 2vw\ell + av^2}{\sqrt{2}} H_{ijk}
				= \frac{a^3 - a\ell^2}{\sqrt{2}(ab - \ell^2)^2}H_{ijk}.
		\end{align*}
		Finally,
		\begin{align*}
			\m{P\cdot \mu (v_i^+,v_j^-)}{v_k^-}_{\B g} &= \B H(P^{-1} v_i^+, P^{-1} v_j^-, P v_k^-)\\
			&= uw\ell\, \B H(v_i^+,v_j^+,v_k^+) + uvb\B H(v_i^+,v_j^-,v_k^-) + w^2b\B H(v_i^-,v_j^+,v_k^-) + wv\ell\B H(v_i^-,v_j^-,v_k^+)\\
			&= \frac{w\ell(u+v) + uvb + w^2b}{\sqrt{2}} H_{ijk}
			= \frac{-a \ell^2+ ab^2}{\sqrt{2}(ab-\ell^2)^2}H_{ijk}.
		\end{align*}
		Thus,
		\[
		|\mu|_g^2 =  \frac{n}{2} \frac{3ab^2 + a^3 -2 \ell^2(a + b)}{(ab-\ell^2)^2},
		\]
		so
		\[
			\operatorname{scal}(g) = -\frac{n}{8}\frac{3ab^2 + a^3 -2 \ell^2(a + b)}{(ab-\ell^2)^2} + \frac{n}{2} \frac{a + b}{ab-\ell^2} = \frac{n}{8} \frac{4a^2 b + ab^2 - a^3 - 2 \ell^2(a+b)}{(ab-\ell^2)^2}.
		\]
	\end{proof}

    Since $\mc G (g_{1,3,1},\beta)$ is Bismut-flat, $(1,3,1,1)$ is clearly a critical point of the function $$(a,b,\ell,r)\mapsto \mc S(\mc G(g_{a,b,\ell},r\beta))\,.$$ We will see, however, that it is not a local maximum point, and the following standard lemma will yield dynamical instability (see also \cite{AK06}).

    \begin{lem}\label{lem_unstable}
    Let $(M,g)$ be an analytic Riemannian manifold and $f \colon M \to \R$ an analytic function. On $M$, consider the gradient flow
    \begin{equation}\label{eqn:GF}
        \frac{\dd}{\dd t} x = \operatorname{grad}_g f (x)\,.
    \end{equation}
 Suppose $p \in M$ is a critical point of $f$ which is not a local maximum. Then there exist $\varepsilon > 0$, and solutions of \eqref{eqn:GF} beginning arbitrarily close to $p$ which become trapped in $M\setminus \B{B_{\varepsilon}(p)}$ in finite time.
\end{lem} 

\begin{proof}
    Since $f$ and $M$ are real analytic, on a neighbourhood $W$ of $p$ we have the following {\L}ojasiewicz inequality \cite{Loj63}: there exist constants  $\gamma \in (0,\frac12]$, $c>0$ such that for all $y\in W$ we have
    \[
        |f(y) - f(p)|^{1-\gamma} \leq c \, |{\rm grad}_g f(y)|.
    \]
    This implies that,
    for some $\varepsilon > 0$,  $f \equiv f(p)$ on $\{y \in \B{B_\varepsilon(p)} : (\operatorname{grad}_gf)(y) = 0\}$. Now, since $p$ is not a local maximum point of $f$, there exists a point $x_0\in \B{B_{\varepsilon}(p)}$ such that $f(x_0) > f(p)$, which can be chosen arbitrarily close to $p$. Let $x(t)$ be the solution to \eqref{eqn:GF} with $x(0) = x_0$. The $\omega$-limit of $x_0$ lies in the set $\{y \in M : (\operatorname{grad}_gf)(y) = 0\}$, since $f(x(t))$ is monotone increasing, and stationary if and only if $\operatorname{grad}_g(f) = 0$. The claim is equivalent to  proving that there exists $T = T(x_0) > 0$ such that $x(t) \in M \setminus \B{B_\varepsilon(0)}$, for all $t > T$. Suppose this is not the case and let $\{t_k\}_{k\in \mathbb N}$ be a sequence of times with $t_k \to \infty$ and $x(t_k) \in \B{B_{\varepsilon}(p)}$. By compactness, $\{x(t_k)\}$ sub-converges to $x_\infty \in \B{B_\varepsilon(p)} \cap (\operatorname{grad}_gf)^{-1}(0) \subset f^{-1}(f(p))$, thus $f(x_\infty) = f(p)$. But since $f(x(t))$ is monotone increasing, 
    $f(x_\infty) \geq f(x_0) > f(p)$, which is a contradiction.
\end{proof}
We are now ready to prove the instability part of \Cref{main_newBRF}.
\begin{thm}\label{thm_dynUnstable}
    The Bismut-flat metric $\mc G(g_{1,3,1},\beta)$ on $(\gggo)_{\B H}$ is dynamically unstable for the homogeneous generalized Ricci flow.
\end{thm}

\begin{proof}
    Consider the curve $\mc G_t := \mc G(g_{1,3-2t,1-t},(1-t)\beta)$, on $(\gggo)_{-2\B H_{\f k}\oplus \B H_{\f k}}$. By Proposition \ref{prop_genscaladhinv}, 
    \begin{equation}\label{eqn_scalcurve}
        \mc S(\mc G_t) = n \left(\frac{-4t^3 + 9t^2 - 6}{3(t^2-2)^3}\right) = n\left(\frac 14 + \frac{t^3}{6}\right) + O(t^4).
    \end{equation}
    Thus, $\mc G_0 = \mc G(g_{1,3,1},\beta)$ is not a local maximum of the generalized scalar curvature, and the result follows from Lemma \ref{lem_unstable} and Proposition \ref{prop_BF}.
\end{proof}

\begin{rmk} \label{rmk_nonStandardUnstable}
     We note that under the diffeomorphism induced by the action \eqref{eqn_twistedAction}, each of the generalized metrics $\mc G(g_{1,3-2t,1-t},(1-t)\beta)$, since they are right $\Delta \msf K$ invariant, are isometric to generalized metrics on the ECA on $(\gggo)_{-2\B H_{\f k}\oplus \B H_{\f k}}$. It therefore also follows from the computations in Theorem \ref{thm_dynUnstable} that the bi-invariant generalized metric $\mc G(2\B g_{\f k}\oplus \B g_{\f k},0)$ is dynamically unstable for the homogeneous generalized Ricci flow on $(\gggo)_{-2\B H_{\f k}\oplus \B H_{\f k}}$. Hence, Theorem \ref{main_newBRF} does not generalize to non-standard bi-invariant Bismut-flat metrics.
\end{rmk}

We remark that the curve is constructed in Theorem \ref{thm_dynUnstable} by computing the Hessian of the function $(a,b,\ell,r) \mapsto \mc S(\mc G(g_{a,b,\ell},\beta))$ at the critical point $(1,3,1,1)$. One finds that the Bismut-flat metric $\mc G(g_{1,3,1},\beta)$ is linearly semi-stable, and that the Hessian has a one-dimensional kernel, to which the curve in Theorem \ref{thm_dynUnstable} is chosen to be tangent. We also note that instability is originated by the  third-order term in \eqref{eqn_scalcurve},  as one would have expected from the linear stability of Bismut-flat metrics, see \cite[Corollary 1.2]{Lee24}.   Instability can also be obtained by a standard but computationally more involved center manifold analysis.

Numerical analysis leads one to conjecture that $(\msf G \times \Delta \msf K)$-homogeneous generalized Ricci flows in the unstable manifold of $\mc G(g_{1,3,1},\beta)$ converge to the Bismut-flat metric $\mc G(g_{1,1,0},0)$. Such a heteroclinic orbit would yield an interesting eternal solution to the homogeneous generalised Ricci flow on $\msf G$.

We also note that $\mc S (\mc G(g_{1,3,1},\beta)) = \frac n4 \leq \frac n3 = \mc G(g_{1,1,0},0)$. Numerical considerations also  suggest that this is the maximum possible value the generalized scalar curvature can take. So, while invariant Bismut-flat metrics are certainly not unique on a given semisimple Lie group, it is still reasonable to ask Question \ref{quest_globmax} mentioned in the Introduction. Note that a positive answer would imply the generalized scalar curvature is bounded, in stark contrast to the classical case. 




\subsection{Generalized Ricci solitons on compact semisimple Lie groups}\label{sec:solitons}

We conclude this section collecting some results on semi-algebraic generalized Ricci solitons on  compact semisimple Lie groups.
We recall $(\f G,\bm \mu,\mc G)$ is a semi-algebraic generalized Ricci soliton if 
\begin{equation}\label{eqn:semialgsol}
\mc Rc_{\mud} = \lambda \id +\operatorname{Sym}_{\mc G}(\mathbf D)\,,
\end{equation}
for some $\lambda \in \R$ and $\mathbf D \in \operatorname{Der}_{\lambda}(\mud)$, see \cite[Definition 3.2]{HGRF} for the precise definition of the latter space. The soliton is called algebraic if $\mathbf D$ is symmetric.

\begin{thm} On a compact semisimple Lie group, semi-algebraic generalized Ricci solitons with $\lambda \leq 0$ are Bismut--Ricci flat. Furthermore, if the torsion is the Cartan form of a bi-invariant metric, then the pair is Bismut-flat and bi-invariant.  
\end{thm}
\begin{proof}
Choosing the preferred isotropic splitting  with respect to $\mc G$, we can identify $\f G = \f g \oplus \f g^*$ and  $\bm \mu = \mu_H$,  for some closed $H \in \Lambda^3 \f g^*$. Hence,  we know from \cite[Lemma 7.6]{HGRF} that
\[
\mathbf D = -\lambda \id + \begin{pmatrix}\lambda\id +D & 0\\ \alpha & -\lambda \id -D^*\end{pmatrix}; \qquad D \in \operatorname{Der}(\mu), \qquad \rho(D)H = \dd_{\mu} \alpha + 2\lambda H ,\qquad \lambda \in \R.
\]
Since $\mu$ is semisimple, we deduce that $D = \ad_X = \mu(X,\cdot)$, for some $X \in \f g$. From \Cref{eqn:semialgsol},  we have $\alpha = - \dd_{\mu}^*H$. Thus \Cref{eqn:semialgsol}  can be rewritten as
\begin{align}
\operatorname{Ric}_{\mu} - \frac 14 H^2 &= \lambda \id+\operatorname{Sym}_g(\ad_X),\\
\rho(\ad_X)H &= -\dd_{\mu}\dd_{\mu}^*H + 2\lambda H.
\end{align}
Using $\tr (\ad_X) = 0$, and the fact that ${\rm tr}(\operatorname{Ric}_{\mu} \ad_X)=0$, we have
\begin{equation}\label{eqn_symadz}
\begin{aligned}
-3|\dd_{\mu}^* H|^2 + 2 \lambda |H|^2 &= g(\rho(\ad_X)H,H) = - 3 \tr(\ad_X H^2) \\
    &= -12 \tr(\ad_X\left(\operatorname{Ric}_{\mu} - \operatorname{Sym}_g(\ad_X)\right) = 12 |\operatorname{Sym}_g(\ad_X)|^2,
\end{aligned}
\end{equation}
so
\[
2\lambda |H|^2 = 12 |\operatorname{Sym}_g(\ad_X)|^2 + 3 |\dd_{\mu}^* H|^2.
\]
 Since $\lambda \leq 0$, we deduce that $\dd_{\mu}^* H = 0$ and $\operatorname{Sym}_g(\ad_X) = 0$. In the case in which  $\lambda = 0$, we are done, since $\operatorname{Ric}_{\mu} = \frac 14 H^2$ and $\dd_{\mu}^*H = 0$. On the other hand, if $\lambda < 0$, then $H = 0$, so $\operatorname{Ric}_{\mu} = \lambda \id < 0$. But this cannot happen on a compact semisimple Lie group.

To prove the last statement, assume that $H=g_{\rm b}\mu$, where $g_{\rm b}$ is a bi-invariant metric. Then, $\rho(\ad_X)H=0$, for bi-invariance.
 Then, we obtain that 
$$
\dd_{\mu}\dd_{\mu}^*H=2\lambda H\,
$$ and that ${\rm Sym}_g(\ad_X)=0$, using \Cref{eqn_symadz}. Now, if $\lambda\ne 0 $, then $H$ is $\dd_{\mu}$-exact. On the other hand, since any cohomology class in  $H^3(\mathsf G)$ is represented by a unique bi-invariant form, this gives that $H=0$, which is not possible. Hence, $\lambda=0$ and  the pair $(g, H)$ is a Bismut--Ricci flat. Now, we can use \cite[Proposition 6.1 (ii)]{LW23} to conclude that $g=g_{\rm b}$.
\end{proof}
\appendix
\section{Existence of \texorpdfstring{$\msf G$}{G}-invariant generalized metrics}
 In this section we prove the existence of $\msf G$-invariant generalized metrics on  $\msf G$-invariant ECAs, as per \Cref{defn_GECA}.  We denote by $E$ a $\msf G$-invariant ECA over a manifold $M$. We assume that a Lie group $\msf G$ acts properly and smoothly (but not necessarily transitively) on $M$, and we denote by  $\tau\colon \msf G\to {\rm Aut}(E)$ the lifted $\msf G$-action on $E$. 
 

\begin{lem}\label{lem_A1}
Let  $A\in {\rm End}(E)$ be  $\msf G$-invariant and such that $\langle A\cdot,\cdot \rangle$ is symmetric and positive definite on $\Gamma(E)$. Then, there exists a $\msf G$-invariant generalized metric.
\end{lem}
\begin{proof}
Since $\iip:=\langle A\cdot,\cdot \rangle$ is symmetric and positive definite,  $A$ is $\iip$-symmetric  and $A^2$ is positive definite.  Hence, we can consider $P\in {\rm End}(E)$ the unique positive definite and $\iip$-symmetric square root of $A^2$.  This allows us to define $\mc G:=AP^{-1}$.  By symmetry of $A$, we have that $[A, P]=[A, P^{-1}]=0$ and hence $\mc G^2=AP^{-1}AP^{-1}=A^2P^{-2}={\rm Id}$.  Moreover, we have that $\langle\mc G\cdot, \cdot \rangle$ is again positive definite and symmetric. Indeed, for any $a, b\in \Gamma(E)$, have that 
$$
\langle \mc G a, b\rangle=\lla P^{-1}a, b\rra=\lla a, P^{-1}b\rra=\langle Aa, P^{-1}b\rangle=\langle a, \mc G b\rangle
$$ and $\langle \mc G a, a\rangle=\lla P^{-1}a, a\rra\ge 0 $ with equality if and only if $a=0$. Therefore, $\mc G$ is a generalized metric. To show that $\mc G$ is $\msf G$-invariant it is sufficient to show that $P$ is $\msf G$-invariant. Using that  $A$ is $\msf G$-invariant, for any  $g\in \msf G$,  $\tau(g^{-1})P\tau(g)$  is a square root of $A^2$. Moreover, it is $\iip$-symmetric and positive definite,  since $\tau(g)\in \msf{SO}(E, \iip)$,  by $\msf G$-invariance of $A$ and $\tau(g)\in \msf{SO} (E,  \ip)$. Thus, thanks to uniqueness of $P$, $P=\tau(g^{-1})P\tau(g)$, for any $g\in \msf G$. The claim now follows.  
\end{proof}

 Stemming from \Cref{lem_A1}, we can derive the following corollary. 
 \begin{cor}\label{cor_A2}
 Let $\msf G$ be a compact Lie group and $E \to M$ a $\msf G$-invariant ECA.  Then, there exists a $\msf G$-invariant generalized metric on $E$. 
 \end{cor}
\begin{proof} 
Given a generalized metric $\B{\mc G}$, we can  define $A\in {\rm End}(E)$ such that, for any  $x\in M$,
\[
    A_x:=\int_{\msf G}\tau(g^{-1})_{g\cdot x}\B{\mc G}_{g\cdot x}\tau(g)_x\dd m(g)\,,
\] 
where $\dd m(g)$ is the bi-invariant Haar measure on $\msf G$ normalized with volume $1$.  Since 
$$
\langle Aa, b\rangle=\int_{\msf G}\langle\tau(g^{-1})\B{\mc G}\tau(g) a, b\rangle \dd m(g)=\int_{\msf G}\langle\B{\mc G}\tau(g) a, \tau(g)b\rangle \dd m(g)\,, \quad a, b\in\Gamma(E)\,,
$$ it is easy to see that $\langle A\cdot, \cdot\rangle$ is symmetric and  positive definite  on $\Gamma(E)$ and that $A$ is $\msf G$-invariant.  We can apply \Cref{lem_A1} and conclude. 
\end{proof}
 With this in hand, we  derive the existence of $\msf G$-invariant generalized metric in the  general case.

\begin{thm}\label{thm_G-invgenmet}
Let $E \to M $ be a $\msf G$-invariant ECA, and assume that the action of $\msf G$ is proper. Then, there exists a $\msf G$-invariant generalized metric on $E$.
\end{thm}

\begin{proof} For any $p\in M$, let $\msf G_p$ be the stabilizer of $p$ in $\msf G$, a compact subgroup by properness. By the Slice Theorem \cite{Pal61},  there exists a  \emph{slice} $S_p$  at $ p\in M$ and a $\msf G$-equivariant map
\[
    \varphi \colon \msf G\times_{\msf G_p}S_p\to M
\] 
which is a diffeomorphism onto a tubular neighbourhood $U_p$ of the orbit $\msf G\cdot p$. Let $\tilde E := \varphi^* (E|_{U_p})$ denote the pull-back bundle, and denote by $\Phi : \tilde E \to E|_{U_p}$ the corresponding vector bundle isomorphism covering $\varphi$. After endowing $\tilde E$  with the pull-back ECA structure from $E|_{U_p}$, as well as the pull-back $\msf G$-action $\gamma\colon \msf G \to {\rm Aut}(\tilde E)$,  $\gamma(g):=\Phi^{-1}\tau(g) \Phi$, the map $\Phi$ becomes a $\G$-equivariant ECA isomorphism. Given a background generalized metric $\B{\mc G}$ on $\tilde E$, we first define $A^S \in \Gamma( (\End{\tilde E})|_{S_p})$ along the slice by 
\[
    A^S_q:=\int_{\msf G_p}\gamma(u^{-1})_{u\cdot q}\B{\mc G}_{u\cdot q}\gamma(u)_{q}\dd m(u)\,, \quad q\in S_p\,,
\]
where $\dd m$ is the bi-invariant Haar measure on $\msf G_p$ normalized with volume $1$. Using that $\tilde E|_{S_p} \to S_p$ is a $\G$-equivariant vector bundle, it is easy to see that $A^S$ is $\msf G_p$-equivariant, i.e. $A^S_{h\cdot q}=\gamma(h)_q A^S_q\gamma(h^{-1})_{h\cdot q}$, for any $h\in \msf G_p$ and $q\in S_p$. Now consider the $\G$-equivariant extension of $A^S$, $A\in \Gamma(\End(\tilde E))$, defined by
\[
 {A}_{g\cdot q}:=\gamma(g){A}^S_{q}\gamma(g^{-1})\,, \quad g \in \msf G\,, q\in S_p.
\]
We note that $A$ is well-defined and $\G$-invariant, by $\msf G_p$-invariance of ${A^S}$. Moreover, if $\ip' = \varphi^* \ip$, then for any $g\in\msf G$ and $q\in S_p$  we have that
\[
    \left\la { A}a, b \right\ra'_{g\cdot q}   =    \int_{\msf G_p}  \left\la\B{\mc G}\gamma(hg^{-1})a, \gamma(hg^{-1}) b \right\ra' \dd m(h)\,,\quad a, b \in \Gamma(\tilde E),
\] 
is symmetric and positive definite, thanks to the fact that $\B{\mc G}$ is a generalized metric. There is a corresponding section $A^{U_p} \in \Gamma(\End E|_{U_p})$ on $M$ with analogous properties.


Finally, by \cite{Pal61}, we can consider a $\msf G$-invariant partition of unity $\{f_\alpha\}$ subordinate to tubular neighbourhoods $\{U_{p_{\alpha}}\}$ given by the slice theorem as above, and define $A^M:=\sum_{\alpha}f_\alpha A^{U_{p_\alpha}}\in \Gamma({\rm End}(E))$, which is $\msf G$-invariant. It is easy to see that $\langle A\cdot, \cdot\rangle$ is symmetric and positive definite on $\Gamma(E)$, thus the theorem follows from  \Cref{lem_A1}.
\end{proof}

\bibliographystyle{amsalpha}
\bibliography{grf}

\end{document}